\documentclass[11pt]{article}
\usepackage{amsthm,amsmath,amssymb}
\usepackage{todo}
\usepackage{bm}
\RequirePackage[numbers]{natbib}
\usepackage{bbold}
\usepackage{ulem}
\usepackage{xcolor}
\RequirePackage[colorlinks,citecolor=blue,urlcolor=blue]{hyperref}
\usepackage{graphicx}

\def\bR{\mathbb R}

\def\bR{\boldsymbol{R}}

\def\bTheta{\boldsymbol{\Theta}}

\newcommand{\E}{\operatorname{\mathbb{E}}}
\renewcommand{\P}{\operatorname{\mathbb{P}}}
\newcommand{\Var}{\operatorname{Var}}

\newcommand{\var}[1]{\operatorname{Var}\left(#1\right)}

\newtheorem{definition}{Definition}
\newtheorem{remark}{Remark}
\newtheorem{Theorem}{Theorem}
\newtheorem{theorem}{Theorem}
\newtheorem{lemma}[Theorem]{Lemma}
\newtheorem{Corollary}[Theorem]{Corollary}
\newtheorem{proposition}[Theorem]{Proposition}

\def\argmin{\mathop{\rm arg\,min}}

\def\diag{\mathop{\rm diag}\nolimits}

\def\Hyp{{\rm Hyp}}
\def\id{\mathop{\rm id}\nolimits}

\def\Var{\mathop{\rm Var}\nolimits}

\def\beq{\begin{equation}}
\def\eeq{\end{equation}}

\def\E{{\mathrm E}}
\def\var{{\mathrm {var}}}
\def\P{{\mathrm P}}
\def\cpset{\mathcal{D}}
\def\pp{{\mathrm p}}

\def\Hyp{{\mathrm H}}
\def\one{{\mathbf 1}}
\def\eps{{\varepsilon}}
\def\const{{\mathrm{const}}}

\def\arg{{\mathrm{arg}}}
\def\diag{{\mathrm{diag}}}

\def\bR{{\mathbb R}}

\def\cN{{\mathcal N}}

\def\T{{ \mathrm{\scriptscriptstyle T} }}

\makeatletter
\g@addto@macro{\UrlBreaks}{\UrlOrds}
\makeatother

\usepackage{natbib}
\bibpunct[, ]{(}{)}{,}{a}{}{,}%
\makeatletter
\g@addto@macro{\UrlBreaks}{\UrlOrds}
\makeatother

\DeclareOldFontCommand{\bf}{\normalfont\bfseries}{\mathbf}
\begin{document}
	%%%%%%%%%%%%%%%%
\begin{center}
	\begin{minipage}{.8\textwidth}
		\centering 
		\LARGE Change-Point Detection in Dynamic Networks with Missing Links\\[0.5cm]
		
		\normalsize
		\textsc{Farida Enikeeva}\\[0.1cm]
		\verb+farida.enikeeva@math.univ-poitiers.fr+\\
		Laboratoire de Math\'ematiques et Applications, UMR CNRS 7348, Universit\'e de Poitiers, France \\
		and \\
		\textsc{Olga Klopp}\\[0.1cm]
		\verb+ Olga.Klopp@math.cnrs.fr,+\\
		ESSEC Business School and CREST, ENSAE, France
	\end{minipage}
\end{center}

	\begin{abstract}
		Structural changes occur in dynamic networks quite frequently and its detection is an important question in many situations such as fraud detection or cybersecurity. Real-life networks are often incompletely observed due to individual non-response or network size. In the present paper we consider the problem of change-point detection  at a temporal sequence of partially observed networks. The goal is to test whether there is a change in the network parameters. Our approach is based on the Matrix CUSUM test statistic and allows growing size of networks. We show that the proposed test is minimax optimal and robust to missing links. We also demonstrate the good behavior of our approach in practice through simulation study  and a real-data application. 
	\end{abstract}

\section{Introduction}
Most of the real-life networks, such as  social networks or  biological networks of neurons connected by their synapses, evolve over the time.  Detecting possible changes in a temporal sequence of networks is an important task with applications in  such areas as intrusion detection or health care monitoring. In this problem we observe a sequence of graphs each of which is usually sparse with large dimension and heterogeneous degrees. %In other words, we observe a sequence of adjacency matrices %and 
	%We assume that each adjacency matrix corresponds to one independent realization from an unknown inhomogeneous random graph  model. 
The underlying distribution of this sequence of graphs  may  change at some unknown time moment called change-point.  The goal of this paper is to design  a testing procedure that allows the detection of the presence of a change-point. % and localization of this change-point.
	
Many of the real-life networks are only partially observed \citep{MissingHandcock,MissingGuimer}. The exhaustive exploration of all interactions in a network requires significant efforts and can be expensive and time consuming. For example, graphs constructed from survey data are likely to be incomplete, due to non-response or drop-out of participants. Another example is online social network data. The gigantic size of these networks requires  working with a sub-sample of the network \citep{crawling}. In all these situations, being able to  infer the properties of the networks from their partial observations is of particular interest.  To the best of our knowledge,   change-point detection for  networks with missing links has not been considered in the literature. In the present paper we focus  on the case when we only have access to partial observations of the network and we  propose an efficient procedure for detection and estimation of a change-point that is adaptive to the missing data.

In real-life networks both the entities and relationships in a network can vary over  time. In the literature we can often find  approaches that detect vertex-based changes in a time series of graphs assuming a fixed set of nodes. Nevertheless, changes in the set of nodes are also quite frequent in applications. For example, the vertices of a citation network are scientific papers and an edge connects two papers if one of them cites the other one. New vertices are constantly added to such networks. Likewise, the set of nodes of the Web, social networks or, more generally,  communication networks are constantly evolving and changing.  To account for this possibility we generalize our approach to the case when nodes of the network may come and go. Considering graphs with varying number of nodes calls for a more general non-parametric model.  In the present paper we build on a popular graphon model and show that our detection procedure can be adapted to this more general setting.

\subsection{Contributions and comparison with previous results}
We consider the problem of hypothesis testing for the presence of a change-point, commonly known as {\it change-point detection}. The related problem of estimating  the unknown change-points and their number is  usually called  the {\it change-point localization} problem. Our primary motivation is to address  the following fundamental question: {\it What is the detection boundary in the problem of change-point detection in dynamic networks?} That is, we are interested in the smallest amount of change in the underlying distribution of the network such that successful detection of a change-point is still possible.  

The change-point detection in dynamic networks  and the related problem of the change-point localization  have attracted  considerable attention in the past few years.  Among others, the problem of change-point detection has been considered in \citep{PeelChangePointDetection} where the authors introduce a generalized hierarchical random graph model with a Bayesian hypothesis test. \cite{Wang2017} consider hierarchical aspects of the problem.  A model based on  non-homogeneous Poisson point processes with cluster dependent piecewise constant intensity functions and common discontinuity points is considered by \cite{Corneli2017}.  The authors of \citep{Corneli2017} propose a variational expectation maximization algorithm for  the change-point localization in this setting.  More recently, in \citep{Hewapathirana2020} the spectral embedding approach  is applied to the problem of   detection of vertex-based changes in dynamic networks. In \citep{Zhang2020} the focus is on the  online detection of a change in the community structure using a subspace projection procedure based on  Gaussian model setting. An eigenspace based statistics is applied in \citep{Cribben2015} to the problem of  localization of changes in  the community structure for stochastic block model.  These works focus mainly on the computational aspects of the problem without theoretical justifications.

In \citep{Wang2014} the authors introduce two types of scan statistics for change-point detection for
time-varying stochastic block model.  Although the authors of \citep{Wang2014} provide an asymptotic analysis of the test power for a simple situation of change in a single community connection probability, the  testing procedure lacks a threshold allowing for testing at a given significance level.
%{\color{red} \bf In \citep{Liu2018} subspace tracking is applied to detect  changes in  community structure Il n'y a pas de change point dans ce paper??}
 The problem of multiple change-point  localization is studied by \cite{Zhao2019ChangepointDI}. The change-point algorithm proposed in this paper  is built upon a refined network estimation procedure. The authors of \citep{Zhao2019ChangepointDI} show consistency of this procedure. A very complete comparison of different change point localization algorithms  (including the one introduced in the present paper) is given in \cite{sulem2022graph}.

A different line of work describes a general nonparametric approach to change-point detection for a general data sequence \citep{Chu2019,Chen2015, Wang2018, pilliat2020optimal}. For example,   a graph-based non-parametric testing procedure is proposed in \citep{Chen2015} and is shown to attain a pre-specified level of type I error. Another related problem  is anomaly detection in dynamic networks. Here the task is to detect abrupt deviation of the network from its normal behavior. A comprehensive survey on this topic is given in~\citep{Akoglu2015}.

To the best of our knowledge, none of the existing works provides  the minimax separation rate for the change-point detection problem in dynamic network.  We can  only find some partial answers. In particular, the results of \citep{CarpentierChangePointGraphs} can be applied to the case of a known location of the change-point.   \citeauthor{CarpentierChangePointGraphs} obtain the minimax separation rate for  the  two sample test  and consider separation in spectral  and Frobenius norms. In the present paper we are interested in a more general setting where no prior information about the existence or location  of change-point is given, which is more realistic and applicable to real data. The problem of two sample test is also considered in \citep{chen2020hypothesis} where the authors propose   a test statistic similar to the one introduced in ~\citep{CarpentierChangePointGraphs} which is based on the singular value of a generalized Wigner matrix. They  apply this approach to the problem of change-point localization and derive consistency results. 

We provide the  minimax separation rate for the spectral norm separation (in Section \ref{sec:MatrixCUSUM} we explain  why we think that the spectral norm is an appropriate choice for this problem) in the case of an unknown change-point  location with  missing links. Besides the lower  bound on the minimax separation rate, we also provide a test procedure based on the spectral norm of the matrix CUSUM statistics that is nearly minimax optimal. Our focus is on the challenging case  when the networks are only partially observed. Missing values are a very common problem in the real life data. Usual imputation methods require observations from a homogeneous distribution. In the presence of unknown change-points such methods are not expected to perform well. This, in turn, will impact the performances of the change-point detection and estimation  methods. These methods, in their  large  majority, are designed for the case of  complete observations and we fall into a vicious circle.  The only works  considering missing values in the context of change-point detection that we are aware of are \citep{BuhlmannChangePointMissing}, \citep{Xie_2013}, and \citep{bertille_2022}. In \citep{BuhlmannChangePointMissing} the authors consider the problem of multiple change-point  localization  for graphical models.  \citeauthor{Xie_2013} propose a fast method for online tracking of a dynamic submanifold underlying very high-dimensional noisy data.  \citep{bertille_2022} considers change-point estimation in partially-observed, high-dimensional time series.

We also propose a new procedure for {\it change-point estimation}.  The problem of a single change-point estimation in a network generated by a dynamic stochastic block model  has been considered, among others, in \citep{MichailidisChangePoint,YuChangePointGraphs,yu2021optimal}.   \citeauthor{MichailidisChangePoint} establish the rate of convergence for the least squares estimate of the change-point and of the parameters of the model. They also derive the asymptotic distribution of the change-point estimator. The problem of multiple change-points localization has been studied by \cite{YuChangePointGraphs}. The authors of \citep{YuChangePointGraphs}  provide optimal localization rate in the case when the magnitudes of the changes in the data generating distribution is measured using Frobenius norm. More recently the methods of \citep{YuChangePointGraphs} have been extended in \citep{yu2021optimal} to the problem of online change-point localization.   The algorithms proposed in \citep{yu2021optimal} and  \citep{YuChangePointGraphs} require two independent samplings. Unlike these  methods, the algorithm that we introduce only requires one independent sample which is a more realistic scenario in many applications.

Usually in works on change-point detection and localization in dynamic networks the set of nodes is assumed to be fixed, e.g. \citep{CarpentierChangePointGraphs,YuChangePointGraphs} or, in asymptotic setting, the number of nodes is assumed to go to infinity, e.g. \citep{chen2020hypothesis}. Both settings may be limiting in some practical situations where the set of network's nodes may change but not forced to go to infinity.  Another common assumption in the existing literature is to suppose that the observed networks are mutually independent, notable exception being \citep{padille2022change,xu2023online}. Using a popular graphon model, we propose a different paradigm when the underlying distribution of the network is independent of the number of nodes and networks may be time-dependent. We  provide an upper bound condition on the  separation rate that guarantees the detection of a change 
for two commun classes of graphons, $K$-step graphons and H\"older continuous graphons. We prove that our procedure is minimax optimal for $K$-step graphons (up to a logarithmic factor and dependency on $K$).

 To summarize, the key contributions of the paper are:  1) we obtain sharp (up to numerical constants)  minimax detection boundary in spectral norm for  the problem of the change-point detection in dynamic networks; 2) we introduce a new and more flexible notion of sparsity (see Section \ref{modeling_sparsity}); 
3) we generalize our setting to  graphon model which allows  us to consider time-dependent networks with a set of nodes that can change; 4) we obtain upper and lower bounds on the minimax separation rate for  change-point detection in the step graphon model; 
5) we introduce a new procedure for  change-point estimation and show its consistency.
% 5) a new bound on the Hadamard product of two matrices (Lemma \ref{lem:Hadamard_LB2}) which may be of independent interest in the setting with non-uniform sampling.}

%\subsection{Outline}
The rest of the paper is organized as follows. We start by summarizing the main
notation used throughout the paper in Section \ref{subsec:notation}.  We introduce our model  in the case of a fixed set of nodes in Section \ref{sec:model} and, in Section \ref{sec:change_point}, we provide our main results for this case.  We consider a more general setting which allows changes in the set of nodes and dependent networks in Section \ref{changing_nodes}. The numerical performance of our method is  illustrated in Section~\ref{sec:simulations}.

\subsection{Notation}
\label{subsec:notation}
We start with some basic notation used in this paper. For any matrix $M$, we denote by $M_{ij}$ its entry in the $i$th row and $j$th column.  The notation $\diag(M)$ stands for the diagonal of a square matrix $M$. The column vector of dimension $n$ with unit entries is denoted by $\one_n=(1,\dots,1)^\T$ and the column vector of dimension $n$ with zero entries is denoted by $\boldsymbol{0}_n=(0,\dots,0)^\T$. The identity matrix of dimension $n$ is denoted by $\mathrm{id}_n$. For a set $A$, we denote by $\one_{A}$ its indicator function.

For two matrices $M$ and $N$ of the same size, their  Hadamard (elementwise) product is denoted by $M\odot N$.  
For any matrix $M \in \mathbb{R}^{n \times n}$, $\| M \|_F$ is its Frobenius norm, $\|M\|_{2\rightarrow 2}$ is its operator norm. and $\|M\|_{\infty} = \max_{ij} |M_{ij}|$ is the largest absolute value of its entries. The column-wise 1,$\infty$-norm of $M$ is denoted by $\|M \|_{1,\infty} = \max_j \sum_i |M_{ij}|$. 
For two sequences $a_n$ and $b_n$ we write $b_n=\Omega(a_n)$ if there exist $C>0$ and $n_0\in\mathbb N$ such that  $b_n\geq Ca_n$ for any $n>n_0$.
We define the function $q(t)=\sqrt{t(1-t)}$ for $t\in[0,1]$ that will control the impact of the change-point location on the rate. 
%	\item  We denote by $[n]$ the set of integers from $1$ to $n$, by $\mathcal{I}$ the set of inliers, and by $\mathcal{O}$ the set of outliers. The set of pairs of inliers is denoted $ I \triangleq\mathcal{I} \times \mathcal{I}$, and its complement is denoted by $O \triangleq [n]\times [n] \setminus I$. For a set of indices $\mathcal{S}$ and a matrix $\bM \in \mathcal{R}^{n\times n}$, we write $\bM_{\mid \mathcal{S}} \triangleq \mathds{1}_{\cS} \odot \bM$ where $\mathds{1}_{\cS}$ is the indicator matrix of the set $\cS$. For any set $\mathcal{S}$, we denote by $\mid \mathcal{S} \mid$ its cardinality.
We denote by $\mathcal{C}_n$ the set of all symmetric connection probability matrices:
\begin{equation*}%\label{eq:set_sparse_Theta}
\mathcal{C}_n=\Bigl\{\Theta\in [0,1]^{n\times n}:\ \Theta=\Theta^\T \Bigr\}
\end{equation*}
and by $\mathcal{C}^0_n$  the set of all symmetric connection probability matrices with zero diagonal:
\begin{equation*}%\label{eq:set_sparse_Theta}
\mathcal{C}^0_n=\Bigl\{\Theta\in \mathcal{C}_n :\ \diag(\Theta)=0 \Bigr\}.
\end{equation*}

\section{Modeling  dynamic networks with missing links}\label{sec:model}

Assume that we have $T$ consecutive independent observations of a network modeled by the adjacency matrix $A^t\in\{0,1\}^{n\times n}$  ($1\le t\le T$) of a simple undirected graph  $G^t$
with $n$ vertices. 
 %The assumption of independent $A^t$ is plausible in several applications, for example, for transportation networks or  Internet of Things (IoT) networks.   In Section \ref{London} we apply our procedure to Transport for London (TfL) Open Data for which we may assume independent $A^t$. 
We will relax the assumption of independent $A^t$ in Section~\ref{changing_nodes} considering a more general graphon model.

We assume that at  time  $t$ the network follows  {\it inhomogeneous random graph model:}  for $i<j$, the elements $A_{ij}^t$ of the matrix $A^t$ are independent Bernoulli random variables with the success probability $\Theta_{ij}^{t}\in [0,1]$.  It means that the edge $(i,j)$ is present in the graph $G^t$ with the probability $\Theta_{ij}^{t}$.  We consider undirected  graphs with no loops. Then, the matrix $A^t$ is symmetric and has zero diagonal.  The corresponding matrix of connection probabilities is denoted by $\Theta^{t}\in [0,1]^{n\times n}$; it is also symmetric with zero diagonal.

Often in practice the dynamic network  is only partially observed. In this case, instead of observing $A=\bigl\{A^t,\  1\le t\le T\bigr\}$, we observe  a sequence of matrices  $ Y= \bigl\{Y^t,\ 1\le t\le T\bigr\}$ where each matrix $Y^t$ contains  the entries of the adjacency matrix $A^t$ that are available at time $t$.  We say that we sample the pair $(i,j)$ at the time moment $t$, if we observe the presence or absence of the corresponding edge. We denote by $\Omega^t$ the sampling matrix such that $\Omega^t_{ij} = 1$ if the pair  $(i,j)$ is sampled at the time $t$, $\Omega^t_{ij} = 0$ otherwise. As the graph is undirected, the sampling matrix is symmetric. Importantly, our methods for change-point detection and estimation do not require the knowledge of the sample matrices $\Omega^t$.

 We assume that for each $t$, the entries $\Omega^t_{ij}$ ($1\le i<j\le n$) are independent random variables and that $\Omega^t$ and $A^t$ are also independent.  We denote the  expectation of $\Omega^t$ by  $\Pi \in \mathbb{R}^{n\times n}$. Then, for any pair $(i,j)$ and any $t$, $\Omega^t_{ij}\sim\mathrm{Bernoulli}(\Pi_{ij})$ and,  for any $i=1,\dots, n$, we set $\Pi_{ii}=1$. Note that we can attribute any value to the diagonal $\text{diag}(\Omega^t)$, since  observing or not the diagonal elements does not carry any information about the diagonal of $A^t$ that vanishes by definition. We assume that, for any pair $(i,j)$, we have non-zero probability to observe $A_{ij}$, that is,  $\Pi_{ij}>0$. For simplicity, we also assume that $\Pi$ does not depend on $t$. This means that the probability of observation for each vertex is not changing over the time. This assumption is realistic in many situations, for example, when analysts sub-sample a very large network. Our proofs may be extended to the situation of $t$-dependent $\Pi$  at the price of additional technicalities and we choose to avoid~it.  We will use  notation $Y^t\sim \mathrm{IRGML} (\Theta^t,\Pi)$ that stands for a realization of an inhomogeneous random graph with missing links according to  connection probability matrix $\Theta^t$ and sampling matrix~$\Pi$.

We can write our observations using  following ``signal-plus-noise" model: %of the observed data $Y^t$:
\begin{equation}\label{Missing_model}
Y^{t}=\Pi\odot\Theta^{t}+W^{t},\quad 1\le t\le T,
\end{equation}
where  $W^t\in [-1,1]^{n\times n}$ is the matrix of centered independent Bernoulli random variables  $W_{ij}^{t}$  with the success probability $\Pi_{ij}\Theta^t_{ij}$ and $Y^t \in \{0,1\}^{n\times n}$ is the matrix with the elements
$Y_{ij}^{t}=\Omega^t_{ij}A^t_{ij}$. 

%\sout{In case of the fully observed network $A=(A^1,\dots,A^T)$, we will use the following ``signal+noise" representation of this dynamic network:}
%\begin{equation}\label{complete_model}
%A^t=\Theta^t+E^t\quad (1\le t\le  T),
%\end{equation}
%\sout{where  $E^t\in [-1,1]^{n\times n}$ is the matrix of independent  centered Bernoulli random variables $E_{ij}^{t}$ taking values in $\bigl\{-\Theta_{ij}^t,1-\Theta_{ij}^t\bigr\}$ with the success probability $\Theta_{ij}^t$.}

%\sout{We work in the setting of growing network size $n\to\infty$ for fixed or small number of observations~$T$.} 

\subsection{Modeling sparse networks} \label{modeling_sparsity}

 Real-life networks are  usually sparse with  a number of connections  that is  much smaller than the maximum possible one,  which is proportional to $n^2$. This implies that  $\Theta^t_{ij}$ may change with $n$ and, in particular, $\Theta^t_{ij}\rightarrow 0$ as $n\rightarrow \infty$ for some (or all) $(i,j)$.
Let $\rho_n=\underset{t}{\max}|\Theta^t|_{\infty}$.
In the literature on the sparse network estimation $\rho_n$ is usually called sparsity parameter and it is assumed that $\rho_n\to 0$ as $n\rightarrow \infty$. In the present paper we consider a more general notion of sparsity based on the column-wise 1,$\infty$-norm of the connection probability matrix.

Let $\kappa_n= \underset{t}{\max}\;\| \Theta^t\|_{1,\infty}$. We will say that the network is \textit{sparse} if, on average, the degree of each node is much smaller than the maximal possible number of connections in the network, that is, we assume that  $\kappa_n/n\rightarrow 0$ as $n\rightarrow \infty$. In particular, we have  $\kappa_n\leq \rho_n n$. 
We will work with the set of all symmetric, zero diagonal connection probability matrices of sparsity at most $\kappa_n$:
$$
\mathcal M_n(\kappa_n)=\Bigl\{\Theta\in \mathcal{C}^0_n:\   \|\Theta\|_{1,\infty}\le\kappa_n\;%\text{and}\; \diag(\Theta)=0 
\Bigr\}.
$$
This new notion of sparsity is more general than the  one based on the sup norm and has multiple advantages. First of all, this notion of sparsity  allows  each $\Theta_{ij}^t$  to decay at its own rate (or to be constant for some of them) which is a much more realistic scenario. Moreover, the methods for sparse network estimation often require the knowledge of the sparsity parameter $\rho_n$ (see, e.g., \citep{klopp_graphon, gao2015optimal})  whose  estimation is a tricky problem. In contrast to $\rho_n$, the parameter $\kappa_n$ can be estimated using the observed  degrees.  

In the case of model with missing links, we will need an additional parameter $\omega_{n}$, an upper bound on the mean of the observed degree of each node, $\max_{1\le t\le T}\|\Pi\odot \Theta^t\|_{1,\infty}\le \omega_{n}$.  In the particular case of uniform sampling with probability $p$, we have  $\omega_{n}=p\kappa_n$ and $\omega_{n}\le np\rho_n$. 
In the case of missing links, we will work with the  pairs of matrices $(\Theta^t,\Pi)$ with the sparsity level bounded by $\omega_{n}$ and consider the following set:
\begin{equation}\label{eq:set_sparse_Theta_Missing}
	\mathcal S_n(\omega_n)=\left \{(\Theta,\Pi)\in \mathcal C_n^0\times \mathcal C_n:  \|\Pi \odot\Theta\|_{1,\infty}\le\omega_n \right \}.
\end{equation}

\section{Change-point detection problem}\label{sec:change_point}
%\subsection{Testing the presence of a change}

We suppose that the connection probability matrix $\Theta^t$ might change at some location $\tau\in\{1,\dots,T-1\}$: 
$$
\Theta^t=\Theta^0\one_{\{1\le t\le \tau\}}+(\Theta^0+\Delta\Theta^\tau)\one_{\{\tau+1\le t\le T\}},\quad t=1,\dots,T.
$$
Here $\Theta^0$ is the connection probability matrix before the change and $\Delta\Theta^\tau\in[-1,1]^{n\times n}$ is a symmetric jump matrix of a change that occurs at time $\tau$.  If $\Theta^t$ does not change, then $\Delta\Theta^{\tau}=0$ for all $1\le  \tau\le T-1$. 
 
We consider the problem of testing whether there is a change in $\Theta^t$ at some (possibly unknown) point  $\tau\in \cpset_T\subset \{1,\dots T-1\}$.  Depending on the set of possible change-point positions $\cpset_T$, we can formulate  two  different testing problems:  problem (P1) of testing the presence of a change at a given point $\tau$ with $\cpset_T=\{\tau\}$ and problem (P2) of  testing the change  at an unknown location within the set of all possible change-point locations $\cpset_T=\{1,\dots,T-1\}$.

The difficulty of assessing the existence of a change-point can be quantified  by what  is called the change-point energy. It is defined as the product of the operator norm of the jump in  the parameter matrix and the function $q(t)=\sqrt{t(1-t)}$ for $t\in[0,1]$. The function $q(t)$  quantifies the impact of  change-point location on the difficulty of detecting the change. Thus, we write the detection problem  in terms of  $q(\tau/T)\|\Pi\odot \Delta\Theta^{\tau}\|_{2\to 2}$,  {\it  the  change-point energy}. In the case of full observations the sampling matrix is given by $\Pi=\one_n\one_n^T$ and the energy equal to $q(\tau/T)\| \Delta\Theta^{\tau}\|_{2\to 2}$.

%We will define the  jump energy at the location $\tau$ as the product of the function $q(\tau/T)$ and the operator norm of the Hadamard product of the sampling matrix $\Pi$ and the jump matrix $\Delta \Theta^\tau$. 
To formulate the hypothesis testing problem we define the set of pairs of matrices before and after the change at some location $\tau$ with the jump energy at least $r>0$:
\begin{equation}\label{def_GammaSet_Missing}
\mathcal W_{n,T}^\tau(\omega_n,r)=\Bigl\{\bigl \{(\Theta^a,\Pi), (\Theta^b,\Pi )\bigr\}\in \mathcal S^{\otimes2}_n(\omega_n):\  q(\tau/T)\bigl \|\Pi\odot (\Theta^a-\Theta^b)\bigr \|_{2\to 2}\ge r  \Bigr\}.
\end{equation}

Let $\mathcal W_n(\omega_n,0)$ denote the set without a jump:
\begin{align*}
\mathcal W_n(\omega_n,0)&=\Bigl\{\left \{\left (\Theta^a,\Pi\right ),\left (\Theta^b,\Pi\right )\right \}\in \mathcal S^{\otimes2}_n(\omega_n):\; \Theta^a=\Theta^b \Bigr\}.
\end{align*}
%	\begin{equation*}
%	\mathcal W_n(\omega_n,K_n,0)=\Bigl\{\bigl(\Theta^a,\Theta^b,\Pi\bigr):\,\left (\Theta^a-\Theta^b,\Pi\right )\in \mathcal S_n(\omega_n,K_n):\ \Theta^a=\Theta^b  \Bigr\}
%	\end{equation*}
 Then,  the detection problem can be written as testing whether the  jump $\Delta\Theta^{\tau}$ is zero under the null,
\begin{equation}\label{eq:h_0_missing}
\Hyp_0:\ \left \{ (\Theta^0+\Delta\Theta^\tau,\Pi), (\Theta^0,\Pi)\right \}\in \mathcal W_n(\omega_n,0),\ \mbox{for all} \  \tau\in\cpset_T
\end{equation}
against the alternative hypothesis of a change in $\Theta^t$
\begin{equation}\label{eq:h_1_missing}
\Hyp_1:\ \left \{ (\Theta^0+\Delta\Theta^\tau,\Pi), (\Theta^0,\Pi)\right \}\in \mathcal W_{n,T}^\tau(\omega_n,\mathcal R_{n,\cpset_T})\ \mbox{for some}\ \tau\in\cpset_T.
\end{equation}
It is well known (see, e.g., \citep{Ingster&Suslina:2003}) that the performance of  a  test depends on how close the sets of measures under the null and under the alternative hypotheses are. As a consequence,  to determine  the smallest possible distance between the null  and the alternative hypothesis is a crucial question in minimax hypothesis testing. This question is formulated in terms of  the  radius $ \mathcal R_{n,\cpset_T}$,  the minimal  amount of energy  that guarantees the change-point detection.
%which determines how well the null hypothesis is separated from the set of alternatives.
In this paper, we are interested in conditions on the minimal  energy  that separated a detectable change from an undetectable one. 

\subsection{Matrix CUSUM statistic}\label{sec:MatrixCUSUM}
We start by introducing our test procedure. We observe the dynamic network~$Y=\{Y^t,\ 1\le t\le T\}$ defined in~\eqref{Missing_model}. 
We call \textit{a test} (or decision rule) any measurable binary function $\psi_{n,T}: Y\to \{0,1\}$ of the data. If its value  is equal to 1, we reject the null hypothesis and say that there is a change in $\Theta^t$. Otherwise we say that the dynamic network has no change in the connection probability matrix.

Recall that $Y^{t}\in\{0,1\}^{n\times n}$ with the elements $Y_{ij}^{t}=\Omega^t_{ij}A^t_{ij}$, where $\Omega^t$ is the sampling matrix with mean $\Pi$ (if there is no missing data, then $\Pi=\one_n\one_n^\T$ is the matrix of unit entries).
Define the following matrix process
\begin{equation}\label{def_Z}
Z_T(t)=%\biggl(\frac{t(T-t)}T\biggr)^{1/2}
\sqrt{\frac{t(T-t)}T}\left (\dfrac{1}{t}\sum_{s=1}^{t}Y^s-\dfrac{1}{T-t}\sum_{s=t+1}^{T}Y^s\right ),\quad t=1,\dots,T-1.
\end{equation}
This process measures the difference between the average number of connections before and after the point $t$. Intuitively,  if there is a change in some entries of the parameter matrix $\Theta^t$ at time $\tau$, then, with high probability, the value of the process $Z_T$ at these entries will be maximal in the neighborhood of $\tau$.
We call the process~given by (\ref{def_Z}) Matrix CUSUM (Matrix Cumulative Sum) process  since it is related to the cumulative sums of $Y$ as
$$
Z_T(t)= \sqrt{ \frac{T}{t(T-t)}}\left [\sum_{s=1}^t Y^s-\frac tT \sum_{s=1}^T Y^s\right ].
$$

We can write our model~(\ref{Missing_model}) in the equivalent form
$$
Z_T(t)=-\mu_T^\tau(t) \Pi\odot\Delta\Theta^\tau+\xi(t),\quad t=1,\dots,T-1,
$$
where 
\begin{equation}\label{def_mu}
\mu_T^\tau(t)= \sqrt{\frac{t(T-t)}T}\left (\frac{\tau}{t}\one_{\{\tau+1\le t\le T\}}+\frac{T-\tau}{T-t}\one_{\{1\le t\le\tau\}}\right )
\end{equation}
and the random matrices
\begin{equation}\label{def_xi}
\xi(t)= \sqrt{\frac{t(T-t)}T}\left (\dfrac{1}{t}\sum_{s=1}^{t}W^{s}-\dfrac{1}{T-t}\sum_{s=t+1}^{T}
W^{s}\right )
\end{equation}
are centered. %Under the alternative hypothesis $\Hyp_1$,
The function $\mu_T^\tau(t)$ attains its maximum  equal to $\sqrt T q(\tau/T)$ at the true change-point  $t=\tau$. Thus, for any  matrix norm, we have 
$$
\max_{1\le t\le T-1} \| \mathbb E (Z_T(t)) \|= \sqrt T q(\tau/T)\|\Pi\odot \Delta\Theta^\tau\|.
$$ 
  We need to  ensure that the norm of the jump in the parameter matrix is larger than the norm of the noise term $\xi(t)$.  We will use the test statistics based on the operator norm of~$Z_{T}(t)$ since we can control the operator norm of the noise term  using the matrix Bernstein inequality. On the other hand, we can easily see that the Frobenius norm is not a suitable choice here.   Indeed, assuming that for any $t$, $\Theta_{ij}^t \approx \rho_n \rightarrow 0$, we get $$\frac{\mathbb E \Vert \xi(t)\Vert^2_F}{n^2}\approx \rho_n  \gg \frac{\Vert \Delta\Theta^\tau\Vert^2_F}{n^2} \approx \rho^2_n.$$
Thus our detection procedure is the following one: if the operator norm of the Matrix CUSUM statistic is sufficiently large  at some point $t\in \cpset_T$, we conclude that there is a change in the connection matrix $\Theta^t$ of the network.
\subsection{Definitions from the minimax testing theory}\label{sec:defminimax}
%Let $Y=(Y^1,\dots,Y^T)$ be observed data satisfying model~(\ref{Missing_model}). 
In this section, we recall some basic definitions from the minimax testing theory.
 Let $Y=\{Y_t,\ 1\le t\le T\}$. Denote by $\P_{(\Theta^0,\Pi)}$ the measure of observations $Y$ under $\Hyp_0$ with $Y^t\sim \mathrm{IRGML}(\Theta^0,\Pi)$ and by $\P_{(\Theta+\Delta\Theta^\tau,\Pi),(\Theta,\Pi)}$  the measure of $Y$ under $\Hyp_1$ with  $Y^t\sim \mathrm{IRGML}(\Theta^t,\Pi)$ for $\Theta^t=\Theta^0\one_{\{1\le t\le\tau\}}+(\Theta^0+\Delta\Theta^\tau)\one_{\{\tau< t\le T\}}$, $\tau\in\cpset_T$.  
Let $\psi_{n,T}:Y\to \{0,1\}$ be a test for one of the problems (P1) or  (P2).
\begin{definition}
	The type I error of $\psi_{n,T}$ is given by
	$$
	\alpha(\psi_{n,T})=\sup_{\left (\Theta^0,\Pi\right )\in \mathcal S_n(\omega_n)}\P_{(\Theta^0,\Pi)}\Bigl\{ \psi_{n,T}=1 \Bigr\}
	$$
	and the type II  error of  $\psi_{n,T}$ is defined as  
	$$
	\beta(\psi_{n,T},R_{n,\cpset_T})
	=\sup_{\tau\in\cpset_T} \Biggl \{\sup_{\bigl \{ (\Theta^0+\Delta\Theta^\tau,\Pi),(\Theta^0,\Pi) \bigr \}\in \mathcal W_n^\tau(\omega_n,\mathcal R_{n,\cpset_T}) } 
	\P_{(\Theta+\Delta\Theta^\tau,\Pi),(\Theta,\Pi)}\Bigl\{ \psi_{n,T}=0\Bigr\}\Biggr \}.
	$$
\end{definition}

Let $\alpha\in (0,1)$ be a given significance level. Denote by $\Psi_\alpha$  the set of all tests of level at most $\alpha$:
$$
\Psi_\alpha=\left\{\psi_{n,T}:\ \alpha(\psi_{n,T})\le \alpha \right\}.
$$
\begin{definition}
	A test $\psi_{n,T}^*\in \Psi_\alpha$ is called minimax if
	$$
	\beta(\psi_{n,T}^*,\mathcal R_{n,\cpset_T}) =\inf_{\psi_{n,T}\in\Psi_\alpha} \beta(\psi_{n,T}, R_{n,\cpset_T}).
	$$
\end{definition}
It is known (see \citep{Ingster&Suslina:2003}, Theorem 2.1, p. 55)  that, for any $\mathcal R_{n,\cpset_T}>0$, the minimax test $\psi_{n,T}^*$ exists and 
\begin{equation}\label{eq:lbTV}
\beta(\psi_{n,T}^*,\mathcal R_{n,\cpset_T}) =\inf_{\psi_{n,T}\in\Psi_\alpha} \beta(\psi_{n,T}, \mathcal R_{n,\cpset_T})\ge 1-\alpha-\frac12\inf_{\mathbb P_1\in[\mathcal P_1]} \|\mathbb P_0-\mathbb P_1\|_{1},
\end{equation}
where  $\mathbb P_0$ is the measure of observations $Y$ corresponding to the null hypothesis $\Hyp_0$ and 
$[\mathcal P_1]$ is the convex hull of set of measures $\mathcal P_1=\mathcal P_1(\mathcal R_{n,\cpset_T})$ corresponding to the alternatives $\Hyp_1$, and $\|\cdot\|_1$ is the $L_1$-distance. It might happen that the minimax test $\psi_{n,T}^*$ is trivial, i.e. $\beta(\psi_{n,T}^*,\mathcal R_{n,\cpset_T})=1-\alpha$, $\forall \alpha\in(0,1)$. In this case the global risk of testing defined as the sum of two testing errors is equal to 1 and the hypotheses $\Hyp_0$ and $\Hyp_1$ are not distinguishable.  This happens if some points of the set of alternatives are too close to the null hypothesis set. To avoid this problem, we remove a ball of radius $\mathcal R_{n,\cpset_T}$ from the set of alternatives $\Hyp_1$. 
 Therefore, it is of crucial interest to know what are the conditions on the  minimum radius $\mathcal R_{n,\cpset_T}$ that guarantee the existence of a non-trivial minimax test. These conditions are formulated in terms of the minimax separation rate.

\begin{definition}\label{def:nonasymp_sep_rate}
	Let $\alpha,\beta\in(0,1)$ be given. Let $\Psi_\alpha$ be the set of all tests $\psi$ of level at most  $\alpha$. We say that the radius $\mathcal R^*_{n,\cpset_T}$ is  $(\alpha,\beta)$-minimax detection boundary in problems (P1)--(P2) of testing against the alternative $\mathcal V_n(\kappa_n,\mathcal R_{n,\cpset_T})$ if 
	$$
	\mathcal R^*_{n,\cpset_T} =\inf_{\psi\in\Psi_\alpha} \mathcal R_{n,\cpset_T}(\alpha,\psi),
	$$
	where 
	$$
	\mathcal R_{n,\cpset_T}(\alpha,\psi)=\inf \Bigl\{\mathcal R>0:\ \beta(\psi,\mathcal R)\le \beta \Bigr\}.
	$$ 	
\end{definition}
The minimax detection boundary is often written as the product $\mathcal R^*_{n,\cpset_T}=C\varphi_{n, \cpset_T}$, where $\varphi_{n,\cpset_T}$ is called {\it minimax separation rate} and  $C$ is a constant  independent of $n$ and $T$.   We say that the radius $\mathcal R_{n,\cpset_T}$ satisfies the upper bound condition  if there exists a constant $C^*>0$ and a test $\psi_{n,T}^*\in \Psi_\alpha$ such that $\forall C>C^*$ $\beta(\psi_{n,T}^*, \mathcal R_{n,\cpset_T})\le \beta$.  We say that $ \mathcal R_{n,\cpset_T}$ satisfies the lower  bound condition if for any $0<C\le C_*$  there is no test of level $\alpha$ with type II error smaller than $\beta$. Our goal is to find the minimax separation rate  $\varphi_{n,\cpset_T}$ and two constants $C_*$ and $C^*$ such that 
$$
C_* \varphi_{n,\cpset_T}\le \mathcal R^*_{n,\cpset_T}\le C^* \varphi_{n,\cpset_T}.
$$

%\section{Results}

\subsection{Change-point  detection at a given location}\label{sec:test_singlepoint}
We first consider problem (P1) of testing the presence of a change at some given point $\tau$. Using the Matrix CUSUM  statistic, we define the following decision rule
\begin{equation}\label{eq:test:P1_wo_log}
\psi_{n,T}^\tau(Y)= \one_{\bigl\{\|Z_{T}(\tau)\|_{2\rightarrow 2}>H_{\alpha,n}\bigr\}},
\end{equation}
where, given the significance level $\alpha\in(0,1)$,  the threshold is 
\begin{equation}\label{eq:H_P1_wo_log}
H_{\alpha,n}=2\sqrt 2 (1+\epsilon) \sqrt {\omega_n} + C_\epsilon \log\Bigl(\frac {2n} \alpha\Bigr).
\end{equation}
Here $\epsilon\in(0,1/2]$ and $C_\epsilon$ is an absolute constant depending on $\eps$, see Lemma~\ref{lem:matrixBernstein}. The result below provides the upper  detection boundary.

\begin{theorem}\label{th:upper_bound:P1_wo_log}
	Let $\alpha,\beta \in(0,1)$ be given significance levels. Assume that for some universal constant $C$ depending only on $\epsilon\in(0,1/2]$, 
	%for some $\delta>0$ $\kappa_n=\Omega(\log^{2+\delta} n)$ and that 
	\begin{equation}\label{MinimaxRateP1}
	\mathcal R_{n,\tau} \ge 4\sqrt 2 (1+\epsilon) \left(\frac{\omega_n}T\right)^{1/2} +\frac{C}{T^{1/2}} \log 
	\left (\frac{n}{\alpha \wedge\beta}\right )
	\end{equation}
	Then, for the test $\psi_{n,T}^\tau$ defined in~(\ref{eq:test:P1_wo_log}) with threshold~(\ref{eq:H_P1_wo_log}) we have 
	$$
	\alpha(\psi_{n,T}^{\tau})\le\alpha\quad\mbox{and}\quad \beta(\psi_{n,T}^{\tau},\mathcal R_{n,\tau})\le \beta.
	$$
\end{theorem}
\begin{proof}
	The proof of this theorem is based on Lemmas~\ref{lem:typeIerror_wo_log} and~\ref{lem:typeIIerror_wo_log}. Lemma~\ref{lem:typeIerror_wo_log} implies  that $\alpha(\psi_{n,T})\le\alpha$. By Lemma~\ref{lem:typeIIerror_wo_log}, the type II error smaller than $\beta$ is guaranteed if $\mathcal R_{n, \tau}$ satisfies~(\ref{eq:rate_nonasymp}). Condition~(\ref{eq:rate_nonasymp}) follows from~\eqref{MinimaxRateP1} for a sufficienly large constant $C$. 
\end{proof}

 The key point of this proof is  to find a bound on the spectral norm of the noise term \eqref{def_xi} that allows to control the Type I and II errors. This is done using  the matrix concentration inequality from \citep{bandeira2016}, see Lemma \ref{lem:matrixBernstein}.
In the next theorem, we provide the lower bound condition on the energy under which hypotheses $\Hyp_0$ and $\Hyp_1$ are indistinguishable. The result is written in terms of the global testing risk $\eta=\alpha+\beta$ and involves the constant  $C_\eta=\log\bigl(1+4(1-\eta)^2\bigr)$ that will be used throughout the paper.

\begin{theorem}\label{th:lb_tau_unknown}
	Let $\alpha\in(0,1)$, $\beta\in(0,1-\alpha]$  be given significance levels and $\eta=\alpha+\beta$. 
	Assume that $\omega_n>\sqrt{2C_\eta}$. % where $C_\eta=\log\bigl(1+4(1-\eta)^2\bigr)$. %$\omega_n>\Bigl(2\log\bigl(1+4(1-\eta)^2\bigr)\Bigr)^{1/2}$.
	Then the $(\alpha,\beta)$-minimax detection boundary in problem~(P1) satisfies the following lower bound
$$
	\mathcal R_{n,\tau}^* \ge  (2 C_\eta)^{1/4}\Bigl(\frac {\omega_n}{T}\Bigr)^{1/2}. 
%		\mathcal R_{n,\tau}^* \ge  \Bigl(2\log\bigl(1+4(1-\eta)^2\bigr) \Bigr)^{1/4}\Bigl(\frac {\omega_n}{T}\Bigr)^{1/2}. 
$$
%	Then thfor sufficiently large $n$ and $T\ge 2$, 
%	\begin{equation}\label{eq:lb_tauknown}
%	R_{n,\tau}<  \sqrt 2 \Bigl(\frac {\omega_n}{T}\Bigr)^{1/2} \Bigl(1+4(1-\eta)^2\Bigr)^{1/4}.
%	\end{equation}	
%	Then, the hypotheses $\Hyp_0$ and $\Hyp_1$ are indistinguishable: %$\forall n,T\ge
%	$\inf\limits_{\psi_{n,T}\in \Psi_\alpha} \beta(\psi_{n,T},R_{n,\cpset_T}) \ge \beta$.
\end{theorem}

	The proof of Theorem  \ref{th:lb_tau_unknown} uses the fact that the minimax testing error in deterministic setting is always  greater than the one in the Bayesian setting. We reduce our testing problem to the Bayesian hypothesis testing  by imposing appropriate prior distributions on the parameters $\Theta^0$ and $\Theta^0+\Delta\Theta$ under $\Hyp_0$ and $\Hyp_1$  respectively. The prior is built in the following way:  under $\Hyp_0$ we have an Erd\H{o}s--R\'enyi model with the connection probability $\rho_n$ and, under $\Hyp_1$, we observe networks with the probability matrix that deviates from the  Erd\H{o}s--R\'enyi model  by a  matrix  with i.i.d. Rademacher entries. In particular, the choice of priors guarantees that,    $\Theta^0$ and $\Theta^0+\Delta\Theta$ belong to the sets defined in~\eqref{eq:set_sparse_Theta_Missing} and~\eqref{def_GammaSet_Missing} with high probability.   We have to find  a condition on the  change-point energy that guarantees a non-trivial testing with the type II error at least $\beta$. We use~\eqref{eq:lbTV}, the lower bound  on the type II error, which is written in terms of the total variation distance between the corresponding random measures that we need to bound from above. Usually, it is  hard to bound the total variation distance, and we use the following relationship between the TV-distance and the $\chi^2$-divergence: $\|P-Q\|_{\mathrm{TV}} \le \sqrt{\chi^2(P,Q)}$. The key point in the proof is to find an upper bound on the $\chi^2$-divergence. This upper bound relies on a new technically demanding bound on the Laplace transform of the Rademacher chaos proven in~\cite{Cortinovis2021}. For a Rademacher vector $\zeta$ we show that
	$$
	1+\chi^2(P,Q) \lesssim \E\exp\biggl(\frac{T \mathcal R_{n,T}^2}{n\omega_n} \cdot \frac 12\zeta^T(\one_n\one_n^T-\id_n)\zeta\biggr)
	\lesssim\exp\left(\biggl(\frac{T \mathcal R_{n,T}^2}{\omega_n}\biggr)^2\biggl(1-\frac{T \mathcal R_{n,T}^2}{n\omega_n}\biggr)^{-1}\right)
	$$
	which implies that the $\chi^2$-divergence is small if the rate satisfies $\mathcal R_{n,T}\asymp (\omega_n/T)^{1/2}$. 
	%which may be of  independent interest.
	
\begin{remark}
Under a mild assumption that the sparsity parameter $\omega_n$ is of polylogarithmic order of $n$, that is,
$\omega_n=\Omega\left(\log^{2+\delta}n \right)$ for some $\delta>0$, the minimax detection boundary satisfies 
$$
 (2C_\eta)^{1/4} \Bigl(\frac {\omega_n}{T}\Bigr)^{1/2}
	\le \mathcal R^*_{n,\tau} \le 4 \sqrt 2 \Bigl(\frac{\omega_n}T\Bigr)^{1/2}.
$$
%$$
%\Bigl(2\log\bigl(1+4(1-\eta)^2\bigr) \Bigr)^{1/4} \Bigl(\frac {\omega_n}{T}\Bigr)^{1/2}
%\le \mathcal R^*_{n,\tau} \le 4 \sqrt 2 \Bigl(\frac{\omega_n}T\Bigr)^{1/2}.
%$$
Note that the lower and upper bounds match up to a constant with the detection rate $(\omega_n/T)^{1/2}$.
Under the polylogarithmic condition the upper bound constant does not depend on the  given global testing risk $\eta$. % $\alpha$ and $\beta$. 
On the other hand, the lower bound constant is bounded from above by $(2\log 5)^{1/4}\approx 1.34$.
\end{remark}

\begin{remark}
 If we restrict our model to the trivial case of Erd\"os--R\'enyi graph model with connection probability $\rho_n$,  it can be shown that the minimax detection boundary is given by
	$$
 C_\eta^{1/2}\Bigl(\frac{\rho_n}T\Bigr)^{1/2} \le \mathcal R_{n,\tau}^*\le  \Bigl(\frac 1\alpha+\frac1\beta\Bigr)^{1/2} \Bigl(\frac{\rho_n}T\Bigr)^{1/2} 
%\sqrt{ \log(1+4(1-\eta)^2)\,	 \frac{\rho_n}T} \le \mathcal R_{n,T}^*\le  \sqrt{\Bigl(\frac 1\alpha+\frac1\beta\Bigr)\frac{\rho_n}T}.
	$$
	and the test 
$$
\psi_{\alpha,n}^\tau={\displaystyle\one_{\bigl\{(n-1)|S^{\tau}_{n,T}|>\sqrt{\rho_n/\alpha} \bigr\}}}
$$
based on the test statistic estimating the change in the connection probability
$$
S^\tau_{n,T}= \sqrt{\frac{\tau(T-\tau)}{T}} \frac1{n(n-1)} \left(\frac 1\tau \sum_{t=1}^\tau \sum_{i\neq j} Y_{ij}^t - \frac 1{T-\tau}\sum_{t=\tau+1}^T \sum_{i\neq j} Y_{ij}^t \right)
$$
is rate optimal. 
Note that in this case test~\eqref{eq:test:P1_wo_log} is suboptimal.
\end{remark}

\subsection{Testing for change-point  at an unknown location with missing links}\label{sec:misslinks}
In this section, we  study problem (P2) of testing the presence of a change at an unknown location.  
To build a test in the case of  unknown change-point location a natural idea is to use a decision rule based on the maximum of the matrix CUSUM statistic over a subset $\mathcal T$ of $\cpset_T$: $L_{n,T}(Y)= \max_{t\in \mathcal T}\|Z_{T}(t)\|_{2\rightarrow 2}$.   For example,  we can take the whole set $\mathcal T=\cpset_T$. However, the choice of the set $\mathcal T$ may be optimized by taking  a dyadic grid of $\{1,\dots,T-1\}$ (see, for example, \citep{LiuGaoSamworth2021}). 
Define the dyadic grid on the set $\cpset_T=\{1,\dots,T-1\}$ as $\mathcal T=\mathcal T^L \cup \mathcal T^R$, where 
\begin{equation}\label{eq:dyadic_grid}
\mathcal T^L=\Bigl\{2^{k},\  k=0,\dots, \lfloor \log_2 (T/2)\rfloor  \},\quad \mathcal T^R= \Bigl\{T-2^{k},\  k=0,\dots, \lfloor \log_2 (T/2)\rfloor  \Bigr\}.
\end{equation}
Note that $|\mathcal T|=2+2\lfloor \log_2(T/2)\rfloor \le 2\log_2(T)$. 

Our decision rule  is based on the maximum of the matrix CUSUM statistic over the dyadic grid:
\begin{equation}\label{test_unknowntau_wolog}
\psi_{n,T}(Y)=\one_{\bigl\{\max\limits_{t\in\mathcal T}\|Z_T^Y(t)\|_{2\to2} >H_{\alpha,n,T}\bigr\}}
\end{equation}
where for some $\epsilon\in(0,1/2]$, the threshold is given by 
\begin{equation}\label{threshold_unknowntau_wolog}
H_{\alpha,n,T}=2\sqrt 2(1+\epsilon) \sqrt{\omega_n} + C_\epsilon \log \left(\frac{4n \log_2(T)}{\alpha}\right).
\end{equation}
 The parameter $\omega_{n}$ in \eqref{threshold_unknowntau_wolog} is an upper bound on the mean of the observed degree of each node. Note that, using Kolmogorov's law of large numbers, it  can be estimated by the maximum observed degree $\widehat\omega_{n}= \max_{1\le t\le T}\| Y^t\|_{1,\infty}$.
 The following theorem provides an upper detection condition for our test.

 \begin{theorem}\label{th:upper_bound_missing_loglog}
 	Let $\alpha,\beta \in(0,1)$ be given significance levels. 
 	 Suppose that for some $\epsilon\in(0,1/2]$ and for a constant $C$ depending only on $\eps$,
$$
	\mathcal R_{n, \cpset_T} \ge   4(1+\epsilon) \Bigl(\frac{6\omega_n}{ T}\Bigr)^{1/2}+\frac {C}{\sqrt T} \log \left( \frac{4n\log_2(T)}{\alpha\wedge\beta}\right)
$$
 	Then, for the test defined in~(\ref{test_unknowntau_wolog})--(\ref{threshold_unknowntau_wolog}), we have   $\alpha(\psi_{n,T} )\le \alpha $ and $\beta (\psi_{n,T},\mathcal R_{n,\cpset_T})\le \beta$.
 \end{theorem}
 
 \begin{proof}
 The proof follows from  Lemmas~\ref{lem:typeIerrorMissing_wolog} and~\ref{lem:typeIIerror_missing_wolog}.
 \end{proof}
 \begin{remark}
 As in the case of testing at a given change-point location, assume that $\omega_n=\Omega\left(\log^{2+\delta}n \right)$ for some $\delta>0$.  Then, if $n\ge \log T$, the minimax detection boundary satisfies 
  	$$
(2C_\eta)^{1/4} \Bigl(\frac{\omega_n}{ T}\Bigr)^{1/2}
 \le \mathcal R^*_{n,\cpset_\tau} \le  4 \sqrt 6\Bigl(\frac{\omega_n}{ T}\Bigr)^{1/2}
 $$
% 	$$
%  \Bigl(2\log\bigl(1+4(1-\eta)^2\bigr) \Bigr)^{1/4} \sqrt{\frac {\omega_n}{T}}
% 	\le R^*_{n,\tau} \le  4 \sqrt{\frac{6\omega_n}{ T}}
% 	$$
 	and our test is minimax rate-optimal in the case of unknown change-point as well. Note that Theorems~\ref{th:upper_bound:P1_wo_log} and \ref{th:upper_bound_missing_loglog}  imply that the unknown location of the change point may have an impact on the detection boundary  only if  $\omega_n\lesssim (\log n + \log\log T)^2$,  that is, in the case of a very sparse network or a very large number of observations $T$.
 	
 	%in a very sparse case of $\omega_n\lesssim\log^{2+\delta}(n)$ and  if the number of time points is small, that is $\log (T)<n$.
 \end{remark}

%Our decision rule  is based on the maximum of the operator norm of the matrix CUSUM statistic over the dyadic grid: 
%\begin{equation}\label{eq:test_general_loglog}
%\psi_{n,T}(Y) = \one_{\bigl\{ \displaystyle{\max\limits_{t\in \mathcal T}\|Z_T(t)\|_{2\rightarrow 2}> H_{\alpha,n,T}}\bigr\}},
%\end{equation}
%where
%\begin{equation}\label{choiceH_loglog}
%H_{\alpha,n,T} =c^* \sqrt{\omega_n \log\Bigl(\frac{2n}\alpha\log_2 (T)\Bigr)}
%\end{equation}
% and $c^{*}$ is an absolute constant provided in Lemma \ref{lm:matrixBernstein}. 

%The following theorem provides an upper detection condition for our test: 
%\begin{theorem}\label{th:upper_bound_missing_loglog}
%	Let $\alpha,\beta \in(0,1)$ be given significance levels.
%	% and let $n$ be large enough with    $n\ge \max(1/\alpha,1/\beta)$. 
%	Assume  that, for all $\tau\in \cpset_T$, we have $9\frac{\tau(T-\tau)}T \omega_{n}\geq \log\left (\frac{2n\log_2(T)}{\alpha\wedge\beta}\right )$
%%	\begin{equation*}%\label{Condition_Sparsity_Missing}
%%	9\frac{\tau(T-\tau)}T \omega_{n}\geq \log\left (\frac{2n\log_2(T)}{\alpha\wedge\beta}\right )
%%	\end{equation*}
%	and
%	\begin{equation}\label{detection_boundary2_loglog}
%	R_{n,\cpset_T} \ge 	 \sqrt 3 c_{*}\left(\frac{\omega_n}{T}\right)^{1/2}\left \{\sqrt{\log \left (2n/\alpha\right )+\log \log_2(T)}+ \sqrt{\log \left (n/\beta\right )}\right \}.		
%	\end{equation}	
%	Then, for the test defined in~(\ref{eq:test_general_loglog}) with threshold~(\ref{choiceH_loglog}), we have   $\alpha(\psi_{n,T} )\le \alpha $ and $\beta (\psi_{n,T},R_{n,\cpset_T})\le \beta$.
%\end{theorem}

Theorems \ref{th:upper_bound_missing_loglog} and \ref{th:lb_tau_unknown}
provide the minimax separation rate in terms of the  operator norm of the Hadamard product of  $\Delta\Theta^\tau$ and $\Pi$. 
 Let us first consider the case of the uniform sampling with $\Pi_{ij}=p_n$ for all $(i,j)$ such that  $i\not =j$. In this case we have  $ \|\Pi\odot \Delta \Theta^\tau\|_{2\to 2}=p_n  \| \Delta \Theta^\tau\|_{2\to 2}$ and we get  the following result that provides the  minimax separation rate for the problem of testing for a change at an unknown location in a dynamic sparse network with links missing uniformly at random.
%Let $\alpha\in(0,1)$ and $\beta\in (0,1-\alpha]$. Set $\eta=1-\alpha-\beta$.

\begin{Corollary}[{\bf Minimax separation rate for spectral norm separation}]\label{cor:sep_rate}
	Consider problem \eqref{eq:h_0_missing}--\eqref{eq:h_1_missing} with $\omega_n=\Omega\left(\log^{2+\delta}n \right)$ for some $\delta>0$.  Assume that  links are missing uniformly at random with probability $p_n$ and that conditions of Theorems \ref{th:lb_tau_unknown} and \ref{th:upper_bound_missing_loglog}  are satisfied.  Then, for $n\ge \log T$,  the minimax separation rate $\varphi_{n,\cpset_T}$ 	for spectral norm separation satisfies:
$$
	\frac{(2C_\eta)^{1/4} }{p_n}\Bigl(\frac{\kappa_n}{T}\Bigr)^{1/2}\leq \varphi_{n,\cpset_T}\leq  \frac{4 \sqrt 6}{p_n} \Bigl(\frac{\kappa_n}{T}\Bigr)^{1/2}.
$$
\end{Corollary}

%This Theorem implies, in particular, that in the case of full observations ($\Pi=\one_n\one_n^T$) the minimax separation rate for the problem of testing a change in a dynamic sparse network at an unknown location is equal to $\sqrt{\dfrac{\kappa_n}{Tq(\tau/T)}}+ \sqrt{c^* \log\left (\frac{nT}\alpha\right )}+\sqrt{c^* \log\left (\frac{n}\beta\right )}$. Compared to the case of known change-point location we pay the price of two additional factors $ \sqrt{ \log\left (T/\alpha\right )}$ and $\sqrt{\log\left (1/\beta\right )}$.
%{\color{red} \bf 
%\begin{remark}
%One may think that considering more structured network models, for example, the stochastic block model, may allow for faster 	detection rate. A closer inspection of the proof of the lower bound shows that it is  valid for a quite restrictive set of parameter of matrices of rank 1. A minor modification in the construction of the prior allows easily  prove that the
%rate in \eqref{minimax_rate} is also minimax optimal for the set of all
%probability matrices corresponding to 2-class stochastic block model.\end{remark}
%
%}
  For more general sampling schemes, we can quantify  how the missing observations affect our detection problem 
%by the minimal eigenvalue of the matrix $\tilde \Pi=(\tilde \Pi_{ij})$ with $\tilde \Pi_{ii}=0$ and $\tilde \Pi_{ij}=\Pi_{ij}^{-1}\ge 1$. We denote
by the following "distortion" parameter:
\begin{equation}\label{def_K_n}
\delta_n:=\delta_n(\Pi,\Theta)=  \frac{\min_{ij} \Pi_{ij}}{\sqrt {r\vee 1}}\quad \text{where}\quad r=\mathrm{rank} (\Pi\odot \Theta).
\end{equation}
%where $r=\mathrm{rank} (\Pi\odot \Theta)$.
% and $\lambda_{\min}(\Pi^{-1})$ is the minimal eigenvalue of the matrix $\Pi^{-1}=( \Pi^{-1}_{ij})$. % with $\tilde \Pi_{ii}=1$ and
%$\tilde \Pi_{ij}=\Pi_{ij}^{-1}\ge 1$. 
%If we consider the case of the  uniform sampling with $\Pi_{ij}=p_n\in(0,1]$ $\forall i\neq j$ and $\diag(\Pi)=\one_n$,
% Obviously, 
%	$$
%	\|\Pi_n\odot \Delta\Theta\|_{2\to2} =p\| \Delta\Theta\|_{2\to2} %.
%	$$
%we have that $\lambda_{min}(\Pi^{-1})=1 -1/p_n$. Thus, in the case of uniform sampling
%$
%\delta_n(\Pi,\Theta)=p^{-1}_n.
%$

Using Lemma \ref{lem:Hadamard_LB2}, we have that $$\|\Pi\odot \Delta\Theta\|_{2\to 2}\ge \delta_n\|\Delta\Theta\|_{2\to 2}.$$ 
Now, Theorem \ref{th:upper_bound_missing_loglog} implies that, for any fixed $\alpha,\beta\in(0,1)$, $\omega_n=\Omega(\log^{2+\delta}n)$ and $n\ge \log T$, the risk of our test is bounded by $\eta=\alpha+\beta$ if the spectral norm of the jump matrix is sufficiently large:
$$
\|\Delta\Theta^\tau\|_{2\to2}\ge \dfrac{4}{\delta_n}\sqrt{\dfrac{6\,\omega_n}{ Tq(\tau/T)}}.
%+\left \{\sqrt{\log (2n\log_2(T)/\alpha) }+ \sqrt{\log \left (n/\beta\right )}\right \}.
$$

\begin{remark}
	Note that using Lemma \ref{lem:Hadamard_LB2} we can improve \eqref{def_K_n} and replace $r=\mathrm{rank} (\Pi\odot \Theta)$ by 
	\[r^*=\underset{M=(M_{ij}):\,M_{ij}=(\Pi\odot \Theta)_{ij}\;\text{for}\;i\not = j}{\min} \mathrm{rank} (M).\]
	That is, we can modify the diagonal of $\Pi\odot \Theta$ to get a matrix of a smaller rank as, for example, in the case when bo%
	%Then, we have that $\delta_n(\Pi,\Delta\Theta)\leq p_n/ \sqrt{2s}$ (for more details on this example see Section \ref{Sc:example_two_communities}).
	th $\Pi$ and 	$\Theta$ have a community structure.
\end{remark}

{\bf Example.} Let us consider  a network with two communities and missing communication. We have two communities with $k$ and $n-k$ nodes such that the links between the members of each community are fully observed. However, the links between the members of two different communities can be missing and are observed with the rate $p_n$. 
Then, after a suitable permutation,  $\Pi$ has the following block structure: $\Pi_{ij}=1$ if $1\le i,j\le k$ and $k+1\le i,j\le n$; otherwise $\Pi_{ij}=p_n$.

If the whole community changes its connection pattern, then  $\Delta \Theta$  also follows the Stochastic Block Model  with the same community structure as $\Pi$ and we have 
$$
\|\Pi\odot \Delta\Theta\|_{2\to2}\geq \frac{p_n}{\sqrt 2}\| \Delta\Theta\|_{2\to2}. 
$$ 

In this case, under the conditions of Corollary~\ref{cor:sep_rate}, the  minimax separation rate $\varphi_{n,\cpset_T}$ satisfies 
$$
\frac{ (2C_\eta)^{1/4} }{p_n}\Bigl(\frac{\kappa_n}{T}\Bigr)^{1/2}\leq \varphi_{n,\cpset_T}\leq  \frac{8}{p_n}\Bigl(\frac{3\kappa_n}{T}\Bigr)^{1/2}.
$$
%$$
%\frac{ \Bigl(2\log\bigl(1+4(1-\eta)^2\bigr) \Bigr)^{1/4} }{p_n}\sqrt{\frac{\kappa_n}{T}}\leq \varphi_{n,\cpset_T}\leq  \frac{8\sqrt 3}{p_n} \sqrt{\frac{3\kappa_n}{T}}.
%$$

%--------------------------------------
\subsection{Change-point localization}\label{sec:estimation}
In this section, we turn to the twin problem of change-point  localization which  amounts
to  estimating the position of the change-point $\tau$. That is, we seek for an estimator of $\tau$,  $\widehat\tau_n$, such that $\vert \widehat\tau_n - \tau\vert\leq \eps$ with high probability. The ratio $\eps/T$ is usually called \textit{localization rate} of an estimator and the
estimator is deemed consistent if its localization rate vanishes as $T\rightarrow \infty$. 
%Recall that we observe the data $Y=(Y^1,\dots,Y^T)$ from model~(\ref{Missing_model}) or $Y=A$ with $A$ from~(\ref{complete_model}) where the connection probability matrix $\Theta^t$ changes at some unknown time moment $\tau$ and satisfies~(\ref{eq:change_Theta}). The goal is to estimate $\tau$.
As in the case of testing, our estimator is  based on  the operator norm of the Matrix CUSUM statistic and it is defined as follows:
\begin{equation}\label{eq:cp_estimator}
\widehat\tau_n\in\arg\max_{1\le t\le T-1} \|Z_T(t)\|_{2\to 2}.
\end{equation}
 Let $\Delta = \|\Pi\odot \Delta\Theta\|_{2\to2}$ denote  the magnitude of the  change in the matrix of connection probabilities. For presenting the results, we scale all the time points to the $[0, 1]$ interval,
 by dividing them by $T$. Let $x^*=\tau/T$ and $\widehat{x}=\widehat\tau_n/T$. The next proposition shows that our estimator is consistent.

 \begin{proposition}\label{thm_estimation}
 Let $\gamma\in(0,1)$. 
 %Assume 
 %that   $ 4.5\,\omega_{n}\geq \log\left (\tfrac{nT}{\gamma}\right )$. 
For any $\epsilon\in(0,1/2]$, with probability larger than $1-\gamma$,  the estimated change-point $\widehat{x}=\widehat{\tau}/T$ returned by \eqref{eq:cp_estimator} satisfies 
 \begin{equation*}%\label{Condition_Sparsity_Missing}
\vert \widehat{x}-x^*\vert \leq \dfrac{6\sqrt 2(1+\epsilon)}{\Delta\, q(x^*)}\Bigl(\frac{\omega_n}{T}\Bigr)^{1/2} +\dfrac{C}{\Delta \,q(x^*)}\frac{\log\left (2nT/\gamma\right )}{\sqrt{T}},
 \end{equation*}	
 where $C$ is a universal constant depending only on $\eps$.
 \end{proposition}

\begin{remark}
Note that the minimax lower bound on the localization error obtained in  \citep{YuChangePointGraphs} (without missing links, that is, when $\omega_{n}=\kappa_{n}$) implies that if  $ \Delta\leq \dfrac{\sqrt{\kappa_{n}}}{ \sqrt{33T}q(x^*)}$, then no consistent change-point estimator can exist.  
\end{remark}

 %----------------------------------------------------

 \section{Testing the change in the sparse graphon model}\label{changing_nodes}
In this section, we generalize our model to the case when the set of network's nodes may change  and to the case of dependent networks.
% For many real-life networks, such that, for example, the  social networks that are constantly growing with time.  
Considering graphs with varying number of nodes call for a more general non-parametric model. The idea is to introduce a well-defined "limiting object" independent of the networks size $n$ and such that the observed  stochastic networks can be viewed as  partial observations of this limiting object.
Such objects, called \textit{graphons}, play a central role in the theory of graph limits introduced by  \cite{LovaszSzegedy} (see, for example, \cite{LovaszBook}, \cite{borgs2006convergent}).  Graphons are symmetric  measurable functions $W :[0,1]^{2}\rightarrow [0,1]$ and, in the sequel, the space of graphons is denoted by $\mathcal W$.  Note that the graphon model  is quite popular  in several applications involving network data, e.g.,  for U.S. congress roll-call votes \citep{navarro}, the optimal control problem for dynamic systems \citep{Gao}, and interacting particles systems \citep{Ayi2023GraphLF}. 

%\subsection{Testing the change in sparse graphon model}
%In this section, we consider a generalization of our model to the case of  dynamic networks with the size depending on time. It can often be observed in practice that the network size is not constant, for example, One of such models the recently have attracted a lot of attention is the sparse graphon model \cite{klopp_graphon}. 

%\subsection{Statement of the problem}
Assume that we observe a dynamic network $A^t$ ($1\le t\le T$) where each adjacency matrix $A_t=(A_{ij}^t)$ is of size $n_t$ that may change with time. 
We allow nodes to appear and disappear from our networks and denote by $\mathcal V$ the set of all nodes of all networks, that is $\vartheta\in \mathcal V$ if there exists a time moment $1\le t\le T$ such that $\vartheta$ is a node of $A^t$. Let $N=|\mathcal V|$ be the  cardinality of $\mathcal V$. 

We consider the sparse graphon model and assume that each vertex $\vartheta \in\mathcal V$ is assigned to an independent random variable $\varepsilon_\vartheta$ uniformly distributed over $[0,1]$. Now, the edge $(i,j)$ is independently included in the network with probability $\Theta_{ij}^t=\rho_n W^t(\varepsilon_i,\varepsilon_j)$, where $W^t$ is the graphon which defines the probability distribution of the network at time instant $t$ and  $\rho_n$ is the sparsity parameter.
Conditionally on $\boldsymbol\varepsilon=(\varepsilon_{\vartheta})_{\vartheta\in \mathcal V}$, we assume that $A_{ij}^t$ are independent Bernoulli random variables with parameter $\Theta_{ij}^t$, $A_{ij}^t\sim\mathrm{Bernoulli}(\Theta_{ij}^t)$. Note that the networks  $A^t\in\{0,1\}^{n\times n}$  ($1\le t\le T$) are dependent through  the latent variables $\varepsilon_i$.
This is similar to the time-dependent network models  recently introduced in  \citep{padille2022change} and \citep{xu2023online}.  Dealing with more general case without assuming conditional independency is an interesting line of research, see, e. g.,  \citep{JMLR:v24:22-0845}.

We suppose that the function $W^t$ might change at some unknown time moment $\tau$, that is, we assume that 
\begin{equation}\label{eq:CP_Wt}
W^t(x,y)=W_1(x,y)\one_{\{1\le t\le \tau\}}+W_2(x,y)\one_{\{\tau+1\le t\le T\}}.
\end{equation}
Additionally, we can assume that the latent variables $\varepsilon_\vartheta$ may also change at time $\tau$. We will denote the new latent variables by  $\varepsilon'_\vartheta$ (for each $\vartheta$, $\varepsilon'_\vartheta$ may be depended on $\varepsilon_\vartheta$). Given the observations $A_t$ ($1\le t\le T$) of networks following the graphon model, we would like to test whether there is a change in the function $W^t$ at some unknown time moment $\tau$. 
%Set $\Delta W^\tau (x,y)=W^{\tau}(x,y)-W^{\tau+1}(x,y)$.

Different graphons can be at the origin of the same distribution on the space of graphs. More precisely, two graphons $U_1$ and $U_2$ define the same probability distribution on graphs if and only if there exists two measure-preserving maps $\phi_1$, $\phi_2$ such that, for all $(x,y)\in[0,1]^2$, we have $U_1(\phi_1(x),\phi_1(y)) =U_2(\phi_2(x),\phi_2(y))$. 
	%This corresponds to possible labels switching. 
	Thus, we need to consider the quotient space of graphons that defines the same probability distribution on graphs equipped with the following distance:
$$
		\delta_2^2(U_1,U_2) = \inf_{\iota\in \mathcal M} \iint_{[0,1]^2} |U_1(\iota(x),\iota(y))-U_2(x,y)|^2\,dx\,dy,
$$
	where $\mathcal M$ is the set of all measure-preserving bijections $\iota:[0,1]\to [0,1]$.  Here $\iota\in \mathcal M$ accounts for possible labels switching.

% and define $\Delta f^\tau(x,y)=\rho_n \Delta W^\tau(x,y)$.
Let  $W(x,y):[0,1]^2\to [0,1]$ be an integrable function. We can define the following operator $W:L_2[0,1]\to L_2[0,1]$:
$$
Wf =\int_0^1 W(x,y)f(y)\,dy
$$
and denote by $\| W \|_{2\to2}$ the corresponding operator norm.  Now we can define the following operator-norm based  distance in the quotient space of graphons:
$$
	\delta(W^{\tau},W^{\tau+1})	= \inf_{\iota\in \mathcal M} \Vert W^{\tau}(x,y)-W^{\tau+1}(\iota(x),\iota(y))\Vert_{2\to 2}.	
$$

The problem of detecting a possible change in the graphon function can be written as the following hypotheses testing problem:
	\begin{align*}
\Hyp_0:\  \delta(W^{\tau},W^{\tau+1})=0\quad\mbox{for all} \  \tau\in\cpset_T\;\mbox{against}\; \Hyp_1:\ \delta(W^{\tau},W^{\tau+1})\ge \delta_{n,T} \;\mbox{for some $\tau\in\cpset_T.$}
\end{align*}

%, our goal is to obtain the detection boundary conditions for this problem in terms of the radius $r_n$. 
 %\subsection{Testing problem for graphons}
As in the  case of the inhomogeneous random graph model that we considered in Section \ref{sec:change_point}, our test is based on the   matrix CUSUM statistics.  Let $\tilde A^s$ denote the restriction of $A^s$ on the set of nodes that is common to all the networks. 
Denote by $n$ the size of this set. As we allow for a change in the latent variables $\varepsilon_\vartheta$ we should account for a possible mismatch of the labels. Let $\pi$ denote a permutation of $\{1, \dots,n\}$. For any matrix $A\in \mathbb{R}^{n\times n}$ denote by $A\circ \pi$ a matrix obtained from $A$ permuting simultaneously its rows and columns. Then, define
$$
Z^{\pi}_T(t) = \frac 1t \sum_{s=1}^t \tilde A^s - \frac 1{T-t}\sum_{s=t+1}^T \tilde A^s\circ\pi.
$$
Let $\pi^{*}\in \underset{\pi}{\argmin} \Vert Z^{\pi}_T(t)\Vert_{ 2 \rightarrow 2}$ and
$Z_T(t)=Z^{\pi^{*}}_T(t)$. Now, the test is  defined as follows:

\begin{equation}\label{test_graphon}
\psi_{n,T}^*(A)= \one_{\left\{\displaystyle{\max_{t\in \mathcal T} \frac{\|Z_T(t)\|_{2\to 2}}{H_{\alpha,n,t}^*}> 1}\right\}}
\end{equation}
where $\mathcal T$ is the dyadic grid~\eqref{eq:dyadic_grid} and  $H_{\alpha,n,t}^*$ is a threshold which will depend on the class of graphons under consideration. Note that we consider a  tricky situation when we allow for a change in the features of nodes. If we consider a simpler  setting, assuming that only the graphon function can change but features remain fixed, we can consider a simpler test, without taking a minimum over all possible permutations with the threshold given by \eqref{threshold_unknowntau_wolog}. The general case with possible change of the features that require minimization over all  permutations is, in general, computationally hard.  There are cases in which such minimization can be achieved through relaxation using the approach proposed in \cite{Bach_permutations}.

For $i\in\{1,2\}$, let $\mathcal W_i\subset \mathcal W$ be two classes of graphons, as, for example, step-function graphons or H\"older--smooth graphons. 
Define the following set of pairs of  graphons  separated by the distance at least~$\delta_{n,T}> 0$:
$$
\mathcal W(\delta_{n,T})=\Bigl\{\{W_1,W_2\}\in\mathcal W_1 \times \mathcal W_2:\ \delta(W_1,W_2)\ge \delta_{n,T}\Bigr\}.
$$

Denote by $\P_{W}$ the measure of observations of the networks $A^1,\dots,A^T$ with the connection probabilities given by $\Theta_{ij}^t=\rho_n W(\eps_i,\eps_j)$. Similarly, given a pair of graphons $W_1$, $W_2$ denote by $\P_{W_1,W_2}^\tau$ the measure of observations $A^1,\dots,A^T$   with a change at $\tau$. 
Then, the type I and II errors of testing for the graphon change-point detection are defined as follows,
\begin{align*}
\alpha(\psi_{n,T})&=\sup_{W\in \mathcal W_0} \P_{W}\Bigl\{\psi_{n,T}=1\Bigr\}\\
\beta(\psi_{n,T},\delta_{n,T})&=\sup_{\tau\in\cpset_T}\sup_{W_1,W_2\in \mathcal W(\delta_{n,T})} \P_{W_1,W_2}^\tau\Bigl\{\psi_{n,T}=0\Bigr\}.
\end{align*}
Let $\delta_{n,T}^*$ denotes the minimax separation rate for this problem. We consider two particular classes of graphons that have been considered in the literature on sparse graphon estimation: the class of step function graphons and the class of smooth graphons. We start by defining the class of step function graphons with $K$ steps:
\begin{definition}\label{def:K_step_graphon}
	 $\mathcal W_K$, the collection of $K$-step graphons, is the subset of graphons $W\in\mathcal W$ such that, for some symmetric matrix $Q\in \mathbb R^{K\times K}$ and some $\phi:[0,1]\to \{1,\dots,K\}$
	$$
	W(x,y)=Q_{\phi(x),\phi(y)}\quad \mbox{for all}\quad x,y\in[0,1].
	$$
\end{definition}
%{\color{blue} Define the set of pairs of $K$-step graphons with the distance at least~$\delta\ge 0$:
%$$
%\mathcal W_K(\delta)=\Bigl\{W_1,W_2\in\mathcal W_K:\ \delta(W_1,W_2)\ge \delta\Bigr\}.
%$$
%The detection problem in the problem of $K$-step graphons is written as the test of 
%$$
%\Hyp_0: (W^\tau,W^{\tau+1})\in \mathcal W_K(0) \quad \forall \tau\in\cpset_T\quad \mbox{against}\quad 
%\Hyp_1: \exists\ \tau\in\cpset_T:\ (W^\tau,W^{\tau+1})\in \mathcal W_K(\delta_{n,T}), 
%$$
%where $\delta_{n,T}>0$ is a minimal detectable distance between the $K$-step graphons. 
%}
	
Assuming that the underlying graphons belong to  $\mathcal W_{K_i}$ for some $K_i\le K$ (we allow for a change in the number of steps) we define the following threshold:

\begin{equation}\label{threshold_kstep_graphon}
	H^{*}_{\alpha,n,t}=2(1+\epsilon)\sqrt{2n\rho_n } +4n^{3/4}\rho_n\sqrt{2T} q\Bigl(\frac t T\Bigr)\Bigl(K\log n  \Bigr)^{1/4}+C_\epsilon \log\left(\frac{8n\log_2(T)}\alpha\right),
\end{equation}
where $\epsilon\in(0,1/2]$ and $C_\epsilon$ is a universal constant depending only on $\eps$.

We have the following upper detection bound in the case of step graphons: 
\begin{theorem}\label{th:Kstepgraphon_upperbound}
	For $i=1,2$ let $W_i\in \mathcal W_{K_i}$ with $K_{1},\,K_{2}\le \frac{n}{\log n}$ and let $\alpha,\beta \in(0,1)$ be given significance levels such that $\alpha\wedge \beta \geq 8/n$.  Let $K_1,K_2\le K$. Assume that our observations $A^1,\dots,A^T$ follow graphon model \eqref{eq:CP_Wt} and that,
	% and let $n$ be large enough with    $n\ge \max(1/\alpha,1/\beta)$. 
	for some $\epsilon\in(0,1/2]$ 
	%for all $\tau\in \cpset_T$,  
	%$\tfrac{9\tau(T-\tau)}{T}\rho_{n}n\geq \log\left (\tfrac{nT}{\alpha\wedge\beta}\right )$ 
%	\begin{equation*}%\label{Condition_Sparsity_Missing}
%	\dfrac{9\tau(T-\tau)}{T}\rho_{n}n\geq \log\left (\tfrac{nT}{\alpha\wedge\beta}\right )
%	\end{equation*}
\begin{equation}\label{detection_boundary_graphon_stepf}
	q\Bigl(\frac\tau T\Bigr)	\delta(W_1,W_2) %\delta(W^{\tau},W^{\tau+1}) 
	\ge \frac{4\sqrt 6 (1+\epsilon)}{\sqrt {n\rho_n T}}  +10 q\Bigl(\frac\tau T\Bigr)\Bigl(\frac {K\log  n}{n} \Bigr)^{1/4} +\dfrac{4 C_\epsilon}{n\rho_n\sqrt{T}}\log\left(\frac{12n\log_2(T)}{\alpha\wedge \beta}\right).	
\end{equation}
 Then, for the test defined in~(\ref{test_graphon}) with threshold~(\ref{threshold_kstep_graphon}), we have that  $\alpha(\psi_{n,T}^{*} )\le \alpha $ and $\beta (\psi_{n,T}^{*},\delta_{n,T})\le \beta$. 
\end{theorem}

The detection boundary  in \eqref{detection_boundary_graphon_stepf} has three parts. The one of the  order $(K/n)^{1/4}$ is due to the `agnostic error"   coming from the variability of latent variables in the graphon model. The second one,  $(n\rho_nT)^{-1/2}$, is related to the sampling of the observed dynamic network.  The third part comes from the Bonferroni device in order to control the errors at the levels $\alpha$ and $\beta$.
The next theorem provides a lower bound on the minimax separation rate:
\begin{theorem}\label{th:Kstepgraphon_lowerbound}
	Let $\alpha\in(0,1)$ and $\beta\in(0,1-\alpha]$ be given significance levels. Let $\eta=\alpha+\beta$  and $C_\eta=\log(1+4(1-\eta)^2)$. Then the $(\alpha,\beta)$-minimax detection boundary for the change-point detection in step graphons model satisfies the following lower bound:
		$$
	q(\tau/T)\delta^*_{n,T} \ge  \left[\left (\frac{8}{3}(1-\eta)^2\right )^{1/4} \; \frac{q(\tau/T)}{n^{1/4}} \right] \vee\frac {\sqrt 2 (1-\eta) e^{-C_\eta/2}}{\sqrt{n\rho_nT}}.
	$$
%		$$
%	q(\tau/T)\delta^*_{n,T} \ge  \left[\left (\frac{8}{3}(1-\eta)^2\right )^{1/4} \; \frac{q(\tau/T)}{n^{1/4}} \right] \vee\frac { \Bigl(0.5(1-\eta)^{-1}+2\Bigr)^{-1} }{\sqrt{n\rho_nT}}.
%	$$
\end{theorem}
The proof of the lower bound consists of two parts. The first one is obtained by bounding the Kullback--Leibler divergence between measures in the case when the node assignment vector changes but the connection probability matrix remains unchanged. This case covers, in particular, a possible mismatch between the labels before and after the change. We derive a bound similar to the data processing inequality from the information theory which  is then reduced to the divergence between two multinomial distributions.  The second part is similar to the one considered in Theorem~\ref{th:lb_tau_unknown} and considers the case of a change in the transition matrix with fixed node assignment. We control the chi-squared divergence between measures under the null and alternative hypotheses. 

Combined with Theorem \ref{th:Kstepgraphon_upperbound}, this result provides  the minimax separation rate for the problem of testing  a change in step-function graphon model:
% which is optimal up to the number of blocks $K$.
\begin{remark}
Assume that the network satisfies the polylogarithmic assumption on the sparsity, $n\rho_n =\Omega(\log^{2+\delta}n)$ for some small $\delta>0$ and that $n\ge \log T$. Then,  Theorems~\ref{th:Kstepgraphon_upperbound} and~\ref{th:Kstepgraphon_lowerbound} imply that the minimax separation rate for step graphons satisfies the following upper and lower bounds which are given up to a constant,
$$
\frac {q(\tau/T)  }{n^{1/4}} +\frac 1{\sqrt{n\rho_nT}} \lesssim \,	q(\tau/T)\delta^*_{n,T}\, \lesssim q(\tau/T)\Bigl(\frac {K\log n}{n} \Bigr)^{1/4}+ \frac{1}{\sqrt{n\rho_n T}}.
$$
The network sampling rate $(n\rho_n T)^{-1/2}$ is exact (up to a constant) and the agnostic error rate $q(\tau/T)n^{-1/4}$ is optimal up to a $\log  n$ factor and the number of blocks.
\end{remark}

%\subsection{H\"older-smooth graphons} 

Next, we consider the class of smooth graphons.
\begin{definition}\label{def_Holder_graphon}
	For any $\gamma>0$, $L>0$,  the class of $\gamma$-H\"older continuous functions $\Sigma(\gamma,L)$ is the set of all functions $W:[0,1]^2\to [0,1]$ such that for all $(x',y'),(x,y)\in[0,1]^2$, 
	$$
	|W(x',y')-\mathcal P_{\lfloor \gamma \rfloor} ((x,y),(x'-x,y'-y))|\le L\left(|x'-x|^{\gamma-\lfloor \gamma\rfloor } +|y'-y|^{\gamma-\lfloor \gamma\rfloor }\right).
	$$
	where $\lfloor\gamma\rfloor$ is the maximal integer less than $\gamma$ and the function $(x',y')\mapsto \mathcal P_{\lfloor \gamma \rfloor}((x,y),(x'-x,y'-y))$ is the Taylor polynomial of degree $\lfloor\gamma\rfloor$ at point $(x,y)$. 
\end{definition}
Assuming that the underlying graphons belong to  $\Sigma(\gamma_{i},L_{i})$ for some $(\gamma_{i},L_{i})$ with $\gamma_i\ge \gamma>0$,  we define the following threshold:
\begin{equation}\label{threshold_holder_graphon}
	H_{\alpha,n,t}^*=4(1+\epsilon)n\rho_n\sqrt{2T}  +2n\rho_n \sqrt{T}q\Bigl(\frac{t} T\Bigr)\left(\frac{\log n}n\right)^{\frac{\gamma\wedge 1}2} +C_\epsilon \log\left(\frac{8n\log_2(T)}\alpha\right),
\end{equation}
where $\epsilon\in(0,1/2]$ and $C_\epsilon$ is a universal constant depending only on $\eps$.

The following theorem provides  the upper detection condition for H\"older continuous graphons: 
\begin{theorem}\label{th:Holdergraphon_upperbound}
For $i=1,2$ let $W_i$ be $\gamma_i$-H\"older continuous functions and let $\alpha,\beta \in(0,1)$ be given significance levels  such that $\alpha\wedge \beta \geq 1/n$.
% and let $n$ be large enough with    $n\ge \max(1/\alpha,1/\beta)$. 
Assume 
that, for some $\epsilon\in(0,1/2]$
%$\tfrac{9\tau(T-\tau)}{T}\,\rho_{n}n\geq \log\left (\tfrac{nT}{\alpha\wedge\beta}\right )$ 
%\begin{equation*}%\label{Condition_Sparsity_Missing}
%\dfrac{9\tau(T-\tau)}{T}\,\rho_{n}n\geq \log\left (\tfrac{nT}{\alpha\wedge\beta}\right )
%\end{equation*}

 \begin{equation}\label{detection_boundary_graphon_holderf}
q\Bigl(\frac\tau T\Bigr)\delta^*_{n,T}	
%\delta(W^{\tau},W^{\tau+1})
\ge  \frac{4\sqrt 6(1+\epsilon)}{\sqrt{n\rho_n T}}  +2q\Bigl(\frac\tau T\Bigr)\left(\frac{\log n}n\right)^{\frac{\gamma\wedge 1}2} +\dfrac{4C_\epsilon}{\rho_n\,n\sqrt{T}}\log\left(\frac{12n\log_2(T)}{\alpha\wedge \beta}\right).
\end{equation}
where  $\gamma=\min(\gamma_1,\gamma_2)$.	Then, for the test defined in~(\ref{test_graphon}) with threshold~(\ref{threshold_holder_graphon}), we have  $\alpha(\psi_{n,T}^{*} )\le \alpha $ and $\beta (\psi_{n,T}^{*},\delta_{n,T}^*)\le \beta$. 
\end{theorem}

\section{\bf Numerical experiments}\label{sec:simulations}
In this section, we provide the study of numerical performance of our method.
For each setting we applied three tests: the test $\psi_{n,T}^\tau$  at the given change-point $\tau$ defined in~\eqref{eq:test:P1_wo_log}, the test $\psi_{n,T}$ over the dyadic grid $\mathcal T^d$ defined in~\eqref{test_unknowntau_wolog} (we add the point $\lfloor T/2\rfloor$ to the grid in our simulations) and   the  test $\psi_{n,T}^{full}(Y)$ based on the maximum over the whole set  $\cpset_T=\{1,\dots,T-1\}$. Each test is calibrated to the significance level $\alpha=0.05$.

The test defined in~\eqref{eq:test:P1_wo_log} is based on the threshold~\eqref{eq:H_P1_wo_log} depending on some given value $\epsilon\in(0,1/2]$ and an unknown universal constant $C_\epsilon$. 
%Tuning this unknown constant  makes practical implementation of this test cumbersome.
 To overcome this difficulty,
we use a slightly different
threshold obtained from the matrix Bernstein inequality  (see, for example, Theorem 1.4 in \citep{tropp-user}). 
%It can be shown,  that the tests based on this inequality have a $\log n$-suboptimal detection rate: for problem (P1), assuming $\omega_n=\Omega(\log n)$, we can obtain the  following upper bound for $n\ge \max(1/\alpha,1/\beta)$,
%		$$
%		\mathcal R_{n,\tau} \ge c_{*}\left(\frac{\omega_n}{T}\right)^{1/2}\left \{\sqrt{\log \left (n/\alpha\right )}+ \sqrt{\log \left (n/\beta\right )}\right \},
%		$$
%		where $c_*>0$ is a universal constant; for problem (P2), if $\omega_n=\Omega(\log(n\log_2T))$, it can be shown that the change is detectable if 
%		$$
%		\mathcal R_{n,\cpset_T} \ge 	 \sqrt 3 c_{*}\left(\frac{\omega_n}{T}\right)^{1/2}\left \{\sqrt{\log \left (2n/\alpha\right )+\log \log_2(T)}+ \sqrt{\log \left (n/\beta\right )}\right \}.	
%		$$
%		This rate is suboptimal up to $\log(n\log T)$ factor. 
The threshold for the test $\psi_{n,T}^\tau$ is given by
$$
\tilde H_{\alpha,n,T}^\tau=\frac13 \frac{\log(n/\alpha)}{\sqrt Tq(\tau/T)} +\Bigl(\frac19 \frac{\log^2(n/\alpha)}{Tq^2(\tau/T)}+2\kappa_n\log(n/\alpha)\Bigr)^{1/2}.
$$
For test over the dyadic grid $\psi_{n,T}(Y) $ and the test $\psi_{n,T}^{full}(Y)$ we use the same threshold 
$$
\tilde H_{\alpha,n,T}(t)=\frac13 \frac{\log(n|\mathcal T|/\alpha)}{\sqrt Tq(t/T)} +\Bigl(\frac19 \frac{\log^2(n|\mathcal T|/\alpha)}{Tq^2(t/T)}+2\kappa_n\log(n/\alpha)\Bigr)^{1/2},\quad t\in \mathcal T. 
$$

In order to compare the performance of the tests under different regimes $(n,T,\tau)$, we introduce ``energy-to-noise ratio" defined by
$$ 
\mathrm{ENR}:=\mathrm{ENR}_{n,T}(\tau/T,\Delta\Theta^\tau)=\frac{q(\tau/T)\|\Delta\Theta^\tau\|_{2\to2}}{\sqrt{\kappa_n/T}}.
$$
This ratio  provides  a numerical upper bound on the minimax testing constant (see Theorem~\ref{th:upper_bound:P1_wo_log}). 
We denote by $\mathrm{ENR}^\tau$,  $\mathrm{ENR}^d$,  $\mathrm{ENR}^f$ the minimal detectable ENR for the tests $\psi_{n,T}^\tau$, $\psi_{n,T}$, and $\psi_{n,T}^{full}$, respectively. Here ``detectable" means that the average power of the corresponding test is equal to 1 over 100 simulations. Note that the lower bound constant for any test of level $\alpha$ with $\beta=0$ is equal to $c^*=\log^{1/4}(1+(1-\alpha)^2)/(4\sqrt 2)$ (see Theorem~\ref{th:lb_tau_unknown}). 

\subsection{Results for Stochastic Block Models}\label{sec:SBMsim}
In this section, we apply our method to four different scenarios of Stochastic Block Models (SBM). Recall that for an SBM model with $K$ communities and connection probability matrix $Q$ between the communities the matrix of connection probabilities $\Theta$ is defined as $\Theta=Z^T QZ$, where $Z\in\{0,1\}^{K\times n}$ is the membership matrix. Each row $i$ of the matrix $Z$ contains only zeros except  one entry $Z_{ij}$ that is equal to 1  if the node $i$ belongs to the community $j$. We suppose in these simulations that the membership matrix $Z$ does not change.  

\begin{paragraph}{\it Scenario 1: SBM with 2 communities, change in connection probability between communities.}
	We suppose that the network  follows the Stochastic Block Model with two balanced communities (block sizes are $\lfloor n/2\rfloor$ and $n-\lfloor n/2\rfloor$.)  The probabilities of connection between the communities change at  some point and  are given by the following matrices $	Q_1$ (before the change) and $	Q_2$ (after the change point):
	$$
	Q_1=\rho_n\begin{pmatrix}
	0.6 & 1\\
	1 & 0.6
	\end{pmatrix},\quad 	Q_2=\rho_n\begin{pmatrix}
	0.6 & \delta\\
	\delta & 0.6
	\end{pmatrix},\quad \delta\in[0,1].
	$$
	%	The block sizes are $\lfloor n/2\lfloor$ and $n-\lfloor n/2\lfloor$.
\end{paragraph}
\vspace*{-5mm}
\begin{paragraph}{\it Scenario 2: SBM with 2 communities, change  in connection probability within one community.}
	As in the previous scenario, we assume that the network follow the stochastic block model with two balanced communities. In this scenario, the matrices $	Q_1$ (before the change) and $	Q_2$ (after the change point) are defined by
	$$
	Q_1=\rho_n\begin{pmatrix}
	1 & 0.5\\
	0.5 & 0.6
	\end{pmatrix},\quad 	Q_2=\rho_n\begin{pmatrix}
	\delta & 0.5\\
	0.5& 0.6 
	\end{pmatrix},\quad\delta\in[0,1].
	$$
	%The block sizes are $\lfloor n/2\lfloor$ and $n-\lfloor n/2\lfloor$.
\end{paragraph}
\vspace*{-5mm}
\begin{paragraph}{\it Scenario 3: SBM with 2 communities, change in connection probability within two communities.}
	Same setting as before, but now connection probabilities inside of both communities change: 
	$$
	Q_1=\rho_n\begin{pmatrix}
	1 & 0.2\\
	0.2 & 1
	\end{pmatrix},\quad 	Q_2=\rho_n\begin{pmatrix}
	\delta & 0.2\\
	0.2& \delta
	\end{pmatrix},\quad\delta\in[0,1].
	$$
	%	The block sizes are $\lfloor n/2\lfloor$ and $n-\lfloor n/2\lfloor$.
	%	\item [Scenario 5:] 
\end{paragraph}
\vspace*{-5mm}
\begin{paragraph}{\it Scenario 4: SBM with 3 communities  and change in connection probability between communities.}
	We suppose that the network follow the stochastic block model with three balanced communities (the block sizes are $k_1=k_2=\lfloor n/3\rfloor$ and $k_3=n-k_1-k_2$).  The probabilities of connection between the communities change at  some point and  are given by the following matrices $Q_1$ (before the change) and $	Q_2$ (after the change point)
	%We suppose that the connection probability matrices $\Theta^0$ and $\Theta^1$ are generated by the following block-connection probabilities:
	$$
	Q_1=\rho_n\begin{pmatrix}
	0.6 & 1 & 0.6\\
	1 & 0.6 & 0.5\\
	0.6 & 0.5 & 0.6
	\end{pmatrix},\quad 	
	Q_2=\rho_n\begin{pmatrix}
	0.6 & 1-\delta & 0.6\\
	1-\delta & 0.6 & 0.5+\delta\\
	0.6 & 0.5+\delta & 0.6
	\end{pmatrix},\quad\delta\in[0,0.5].
	$$
	%The block sizes are $k_1=k_2=\lfloor n/3\lfloor$ and $k_3=n-k_1-k_2$. 
	This scenario was  considered in~\citep{yu2021optimal} with the constant parameter $\delta=0.5$. In our  study we vary $\delta$ and report the test power in terms of the change in the  energy. 
\end{paragraph}

\begin{paragraph}{\it Scenario 5: SBM with 2 communities,  change in connection probability within communities.}
		We suppose that the network  follows the Stochastic Block Model with two balanced communities (block sizes are $\lfloor n/2\rfloor$ and $n-\lfloor n/2\rfloor$.)  The probabilities of connection within the communities change at  some point and  are given by the following matrices $	Q_1$ (before the change) and $Q_2$ (after the change point):
		$$
		Q_1=\rho_n\begin{pmatrix}
			0.6 & 1\\
			1 & 0.6
		\end{pmatrix},\quad 	Q_2=\rho_n\begin{pmatrix}
			\delta & 1\\
			1 & \delta
		\end{pmatrix},\quad \delta\in[0,1].
		$$
		%	The block sizes are $\lfloor n/2\lfloor$ and $n-\lfloor n/2\lfloor$.
	\end{paragraph}

The sparsity $\rho_n$ is set to $n^{-1/2}$. The sparsity parameter $\kappa_n$ is set to $\kappa_n=\frac 12n\rho_n (\|Q_1\|_{1,\infty}\vee\|Q_2\|_{1,\infty})$ for Scenarios~1--3, 5 and to  $\kappa_n=\tfrac13n\rho_n (\|Q_1\|_{1,\infty}\vee\|Q_2\|_{1,\infty})$ for Scenario~4. In all the scenarios we have $\kappa_n\le \sqrt{n}$.

%\begin{equation}\label{eq:test_tau_sim}
%\psi_{n,T}^{\tau}(Y) = \mathbb 1\biggl\{  \|Z_{T}(\tau)\| > q_{\alpha,n,T}^\tau\biggr\}
%\end{equation}
%with 
%$$
%q_{\alpha,n,T}^\tau=\frac13 \frac{\log(n/\alpha)}{\sqrt Tq(\tau/T)} +\Bigl(\frac19 \frac{\log^2(n/\alpha)}{Tq^2(\tau/T)}+2\kappa_n\log(n/\alpha)\Bigr)^{1/2},
%$$
%the test over the dyadic grid $\mathcal T=\mathcal T^d$ 
%\begin{equation}\label{eq:test_dyadic}
%\psi_{n,T}(Y) = \mathbb 1\biggl\{ \max_{t\in \mathcal T^d} \frac{\|Z_{T}(t)\|_{2\rightarrow 2}}{ q_{\alpha,n,T}(t)}>1\biggr\},
%\end{equation}
%and the full test over the whole grid $\mathcal T=\cpset_T$:
%\begin{equation}\label{eq:test_full}
%\psi_{n,T}^{full}(Y) = \mathbb 1\biggl\{ \max_{t\in \cpset_T} \frac{\|Z_{T}(t)\|_{2\rightarrow 2}}{ q_{\alpha,n,T}(t)}>1\biggr\}
%\end{equation}
%with the threshold 
%$$
%q_{\alpha,n,T}(t)=\frac13 \frac{\log(n|\mathcal T|/\alpha)}{\sqrt Tq(t/T)} +\Bigl(\frac19 \frac{\log^2(n|\mathcal T|/\alpha)}{Tq^2(t/T)}+2\kappa_n\log(n/\alpha)\Bigr)^{1/2},\quad t\in \mathcal T. 
%$$

\subsubsection{Varying $n$ and $T$}

In this part we study the dependency of the energy-to-noise ratio $\mathrm {ENR}$ on $n$ and $T$. We report the results of simulations for five scenarios in Table~\ref{table:results}. We see  that globally the ENR decreases when the number of observations  $T$ increases. Some changes cannot be detected by our tests.  For example, for Scenarios 1, 2 and 4 the change-point is undetectable for $T=20$ and $n=100$ by any test. It can be explained by the small number of observations $T$ implying the threshold that is systematically greater than the value of the test statistic. Scenario 4 seems to be more difficult than the other ones, it might be, in particular, due to the fact that the allowed changes are within the interval $[0,0.5]$ that is smaller than in Scenarios 1--3. 

\begin{table}[htbp!]
	\footnotesize
	\begin{center}
		\begin{tabular}{|l|c|c|c|c|c||l|c|c|c|c|c|} 
			\hline
			& $n$ & $T$ & $\mathrm{ENR}^\tau$ & $\mathrm{ENR}^{d}$ & $\mathrm{ENR}^{f}$ && $n$ & $T$ & $\mathrm{ENR}^\tau$ & $\mathrm{ENR}^{d}$ & $\mathrm{ENR}^{f}$ \\
			\hline
			%			Scenario 1  & 100 & 20 &2.2274&2.5456&2.7047& 	Scenario 1  & 100 & 20 &1.4128&1.6419&1.7374\\
			%			$\tau/T=0.5$ & 100 & 50 & 2.1802 & 2.5156 & 2.6800 &	$\tau/T=0.1$ & 100 & 50&1.3282   &{\bf 1.7810}&1.7207 \\
			%							  & 100 & 100 & 2.1345 & 2.4903 & 2.7275&& 100 & 100 &1.3234 & {\bf 1.7930} & 1.7076 \\
			%			  				  & 100 & 250 &2.2500&2.6250&2.8125&&100 & 250 &1.2825 &1.6875 &1.7550\\
			%			  				  & 150 & 20 &2.2378&2.5322&2.6500&&150 & 20 &1.4204   & 1.6324 &  1.6960 \\
			%			  				  & 150 & 50 &2.2347&2.5140&2.7002&& 150 & 50 &1.3743   & {\bf 1.7765} &  1.7095\\
			%			  				  & 150 & 100&2.2385&2.5019&2.7652&&150 & 100 &1.3273  & {\bf 1.7540}  & 1.7065\\
			%			  				  & 150 & 250&2.2902&2.4984&2.9148&&150 & 250 &1.3491 &   1.7239 &   1.7239\\
			%			 \hline
			Scenario 1  & 100 & 20 &NA&NA&NA& 	Scenario 2 & 100 & 20 &NA&NA&NA\\
			$\tau/T=0.5$ & 100 & 50 & 2.1546   & 2.4397 &   2.6298 &	$\tau/T=0.5$&100 & 50 &2.1550  &  2.4766 &   2.7018 \\
			& 100 & 100 &2.0612   & 2.4197 &   2.6885&& 100 & 100 & 2.1834  & 2.5018  &2.7292 \\
			& 100 & 250 &  2.0546  &  2.4089 &   2.6923&&100 & 250 & 2.0857  & 2.5172   & 2.8049\\
			& 150 & 20 &2.1928&NA&NA&&150 & 20 &2.2389&NA&NA \\
			& 150 & 50 &2.1363  & 2.4515 &  2.6266&& 150 & 50 &2.1812  &2.5387 &  2.7175\\
			& 150 & 100& 2.0801 & 2.4763  &  2.6745&&50 & 100 &2.1239  &  2.5284   & 2.7812\\
			& 150 & 250&2.0360 & 2.4276  &  2.7408&&150 & 250 & 2.1588   & 2.5586  &  2.7984\\
			\hline		  
			Scenario 3  & 100 & 20 &2.2098&NA&NA&Scenario 4 & 100 & 50 &NA &NA & NA\\
			$\tau/T=0.5$& 100 & 50 &2.1253&2.4495&2.6296& 	$\tau/T=0.5$&100 & 100 &NA &NA & NA \\
			& 100 & 100 &2.0377&2.4452&2.6999&& 100 & 150  & 2.0235 & NA & NA \\
			& 100 & 250 &2.0137&2.4146&2.7386&&100 & 250 &2.0483&2.3748&2.7013\\
			& 150 & 20 &2.2528&NA&NA&&150 & 50 &NA &NA & NA \\
			& 150 & 50 &2.1212  & 2.4814  &  2.6815&& 150 & 100 &NA &NA & NA\\
			& 150 & 100&2.1508  & 2.4904  &  2.7168&&150 & 150 &2.0664 &2.4236 & NA\\
			& 150 & 250& 2.0583  & 2.4163  & 2.7743&&150 & 250 &2.0420&2.3713&2.2007\\
			\hline		  
		\end{tabular}
		
	\end{center}
	\normalsize
	\caption{The results of simulations for five proposed scenarios for $n=100$ and 150.}
	\label{table:results}
\end{table}

The ENR of the test $\psi_{n,T}^\tau$ is always smaller than the ENR of two other tests. Concerning the tests over the dyadic grid and over the whole set of observations, the test $\psi_{n,T}$ outperforms the test $\psi_{n,T}^{full}$ in the majority of parameter settings and scenarios.

\subsubsection{Estimating the sparsity}\label{sec:estim_sparsity}

In this section, we study the performance of our  tests  with the thresholds based on the estimated sparsity parameter $\widehat \kappa_n$. Taking the maximum of $1,\infty$-norms $\max_t \max_{j} A^t_{\cdot j}$ as an estimator will systematically overestimate $\kappa_n$.  Indeed, since $\E\max_j \xi_j =\sqrt{2\log n}(1+o(1))$, for $\xi_j\sim \cN(0,1)$ i.i.d., using the Gaussian approximation for binomial variables we can show that $\E (\max_{j} A_{\cdot j}^t)\asymp \kappa_n+ \sqrt{2 \kappa_n\log n}/n$.   To estimate $\kappa_n$,  we first  calculate $A^t_{\cdot j}=\sum_{i=1}^n A_{ij}^t$   for each $t=1,\dots,T$. Then, we obtain a robust estimator of the sparsity taking the 0.9-level empirical quantile of $A^t_{\cdot j}$. The final estimator of $\kappa_n$  maximizes the obtained estimated sparsities $\widehat\kappa_n^t$ over $t$:
	$
	\widehat \kappa_n=\max_t Q\Bigl(\Bigl\{\sum_{i=1}^n A_{ij}^t,\ j=1,\dots,n\Bigr\},0.9\Bigr).
	$
	Here $Q(Z,\alpha)$ denotes the $\alpha$-level empirical quantile of the sample $Z$.   The relative risk of estimation of $\kappa_n$ is shown for different values of $n$ and  sparsity $\rho_n$ in Table~\ref{tab:risk_kappa}. We can see that the choice of  the quantile $Q=0.9$ guarantees  $\widehat\kappa_n\ge \kappa_n$ which is important to maintain the test  significance at level  smaller than $\alpha$  and  it also has a good estimation precision for a large range of graph sparsity regimes. 
	\begin{table}[htbp!]
		\footnotesize
		\begin{center}
			\begin{tabular}{|l|| c|| c|c|c|c|c|} 
				\hline
				Sparsity  & $Q$ & $n=100$ & $n=200$ & $n=500$ & $n=1000$ & $n=1200$\\
				\hline
				& 0.7 &   -0.0334 &	-0.0352	& -0.0407& 	-0.0294	& -0.0212\\
				$\rho_n=n^{-0.1}$  & 0.8 &    -0.0018	 &   0.0070	  & -0.0082	  &  -0.0200	&   -0.0067 \\
				& 0.9 &       0.0586	&    0.0545	&    0.0230	  &  0.0029	 &   0.0154 \\
				\hline
				& 0.6 &   0.0239	& 0.0090&	-0.0008	& -0.0066 &	-0.0114\\
				$\rho_n=n^{-0.25}$  & 0.8 &       0.0350	&    0.0823	 &   0.0460	  &  0.0310	 &   0.0269\\
				& 0.9 &         0.1342	  &  0.1631	&    0.1029	 &   0.0746	 &   0.0677 \\		 
				\hline
				& 0.7 &        0.1299 &	0.1412	&0.0837	&0.0806	& 0.0850\\
				$\rho_n=n^{-0.5}$  & 0.8 &         0.2712	&    0.2189	 &   0.1800	 &   0.1489	 
				&   0.1486\\
				& 0.9 &            0.4892	&    0.3719	 &   0.3396	  &  0.2623	  &  0.2577 \\		 
				\hline  
			\end{tabular}
		\end{center}
		\normalsize
		\caption{Approximation of the relative risk of estimation $R(\widehat\kappa_n)= \E(\widehat \kappa_n-\kappa_n)/\kappa_n$ over 100 simulations for $T=10$, $\tau=T/2$, $n\in\{100,200,500,1000,1200\}$,  three different levels of sparsity $\rho_n$ and the quantiles $Q\in\{0.7,0.8,0.9\}$.}
		\label{tab:risk_kappa}
	\end{table}

In Fig.~\ref{fig:Scenario2_5:n100T100} we compare the performance of the test  adaptive to the unknown sparsity level with the test where we use the true value $\kappa_n$. We consider Scenario 1 with $n=100$, $T=100$ and  Scenario 4 with $n=100$, $T=250$.  For Scenario 1, $\kappa_n\approx 0.8n\rho_n$  and, for Scenario 4, $\kappa_n\approx 0.77 n\rho_n$. In our simulations  the change-point is located in the middle. For Scenario 1 with $T=100$, our estimator  $\widehat\kappa_n$ slightly overestimates the true value $\kappa_n=7.94$ with the average value $\widehat\kappa_n=12.9129$ calculated over 100 simulations and over all values of the parameter $\delta$. For Scenario 4  and $T=250$ we obtain a better  estimation that is equal to $\widehat\kappa_n=11.6858$ while the true value  $\kappa_n=7.26$. For Scenario 1, the tests adaptive to the unknown sparsity level behave quite well with a reasonable power and with the ENR that is about 1.25 times greater than the ENR of the corresponding test with known $\kappa_n$. For Scenario 4 the proportion between the ENRs is about 1.07 times. Better performances of the adaptive test for Scenario 4 is expected as in this case the change in the sparsity level of the network is less important than in the  Scenario 1. Our test construction is based on the upper bound for the sparsity level for all $t$ and naturally  gives better results in the case of the change which is more homogeneous in terms of sparsity.
%\sout{ It would be interesting to propose a test that will detect a change in the network adjacency matrix $\Theta^t$ taking into account a possible change in the sparsity parameter $\kappa_n$. }

\begin{figure}[htbp!]
	\centering
	\includegraphics[width=0.48\linewidth]{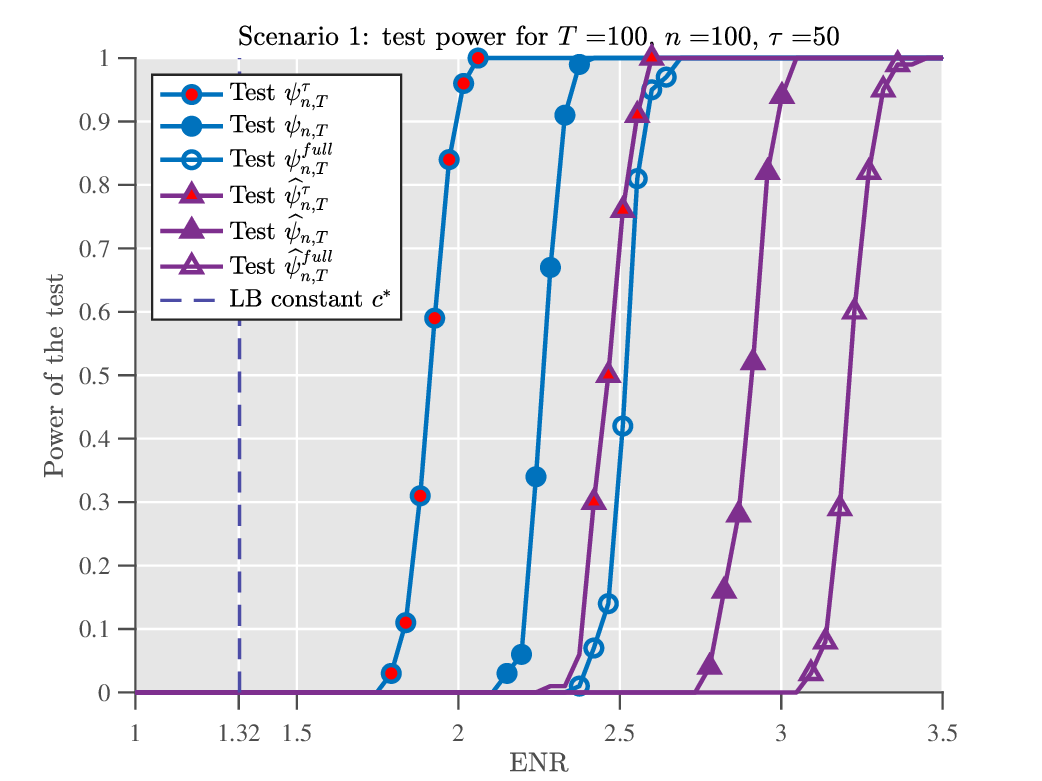}
	\includegraphics[width=0.48\linewidth]{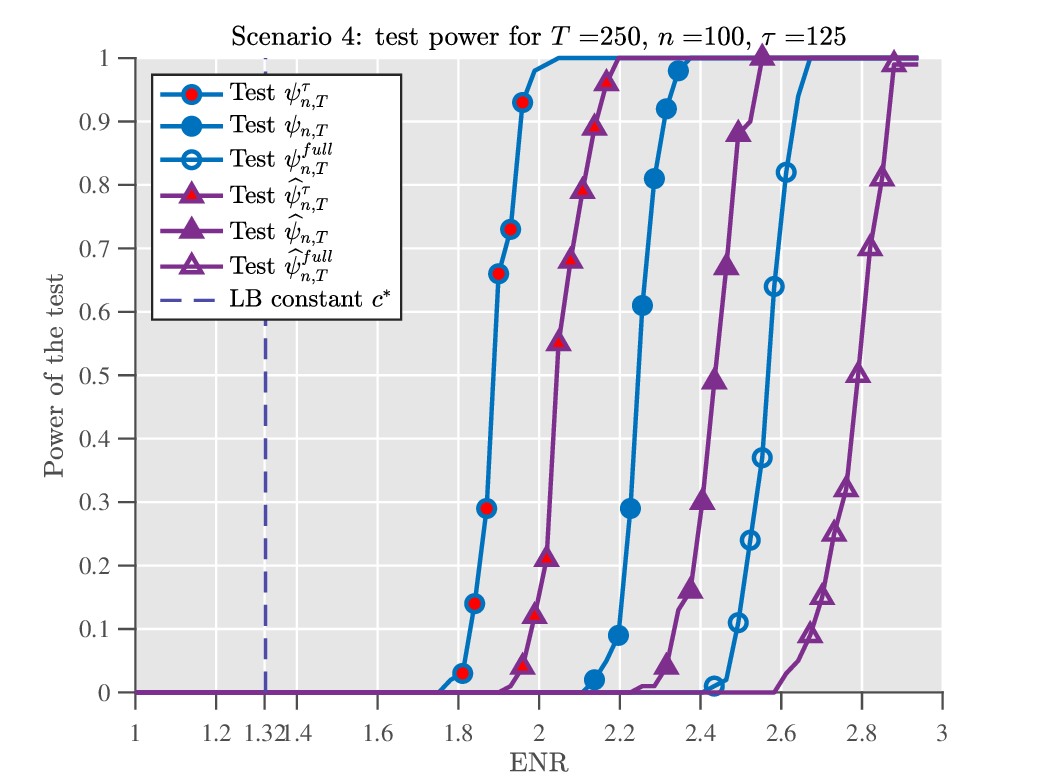}
	\caption{The test powers with known (tests $\psi$) and estimated sparsity parameters (tests $\widehat\psi$) for $n=100$, $\tau/T=0.5$ and $T=100$ for Scenario 1 (on the left) and $T=250$ for Scenario 4 (on the right).}
	\label{fig:Scenario2_5:n100T100}
\end{figure}

%\subsubsection{Coping with missing links}

\subsubsection{Change-point localization in the case of uniform sampling}

We have simulated the networks of size $n=100$ from Scenario 1  with $T=100$ and a change-point located at $\tau\in\{5,25,50\}$.  The networks have missing links generated according to  the uniform sampling matrix $\Pi=p_n (\one_n\one_n^t-\id_n)$ with the sampling rate $p_n\in (0,1]$. We compute the average absolute error of our estimator $\widehat \tau$ defined in \eqref{eq:cp_estimator} over $N=100$ simulations normalized by the number of observations $T$: 
$R_N(\widehat\tau,\tau) =(NT)^{-1} \sum_{i=1}^N |\widehat\tau_i-\tau|$.
%This error approximates the normalized risk $T^{-1} \E|\widehat\tau-\tau|$.
We present the dependence of this risk on the sampling rate $p_n$ and on the norm of the jump $\Delta\Theta^\tau$. 

In Fig.~\ref{fig:risk_SBMmodel2} we observe the dependence of the risk on the location $\tau$: the closer $\tau$ is to the middle of the interval, the easier the estimation is. The dependence of the rate of convergence of $\widehat\tau$ on the norm $\|\Pi\odot\Delta\Theta^\tau\|_{2\to2}=p_n\|\Delta\Theta^\tau\|_{2\to2}$ is represented by the level curve  $p_n\|\Delta\Theta^t\|_{2\to2}\approx \const$ separating the black area corresponding to a low change-point localization error from the light one with the higher error. 

\begin{figure}[htbp!]
	\centering
	\includegraphics[width=0.32\linewidth]{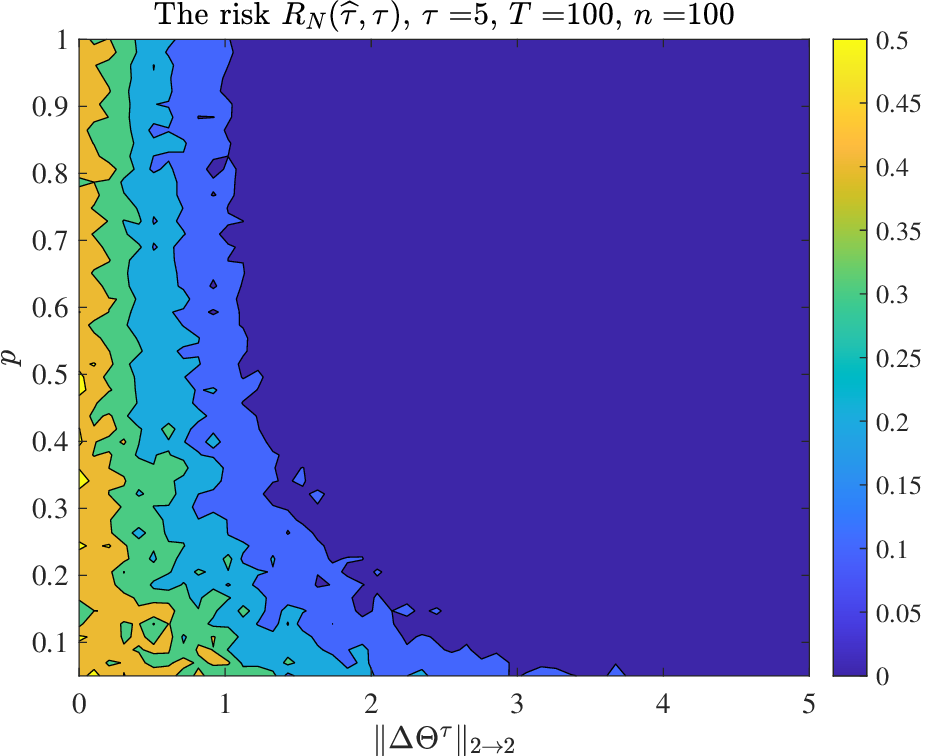}
	\includegraphics[width=0.32\linewidth]{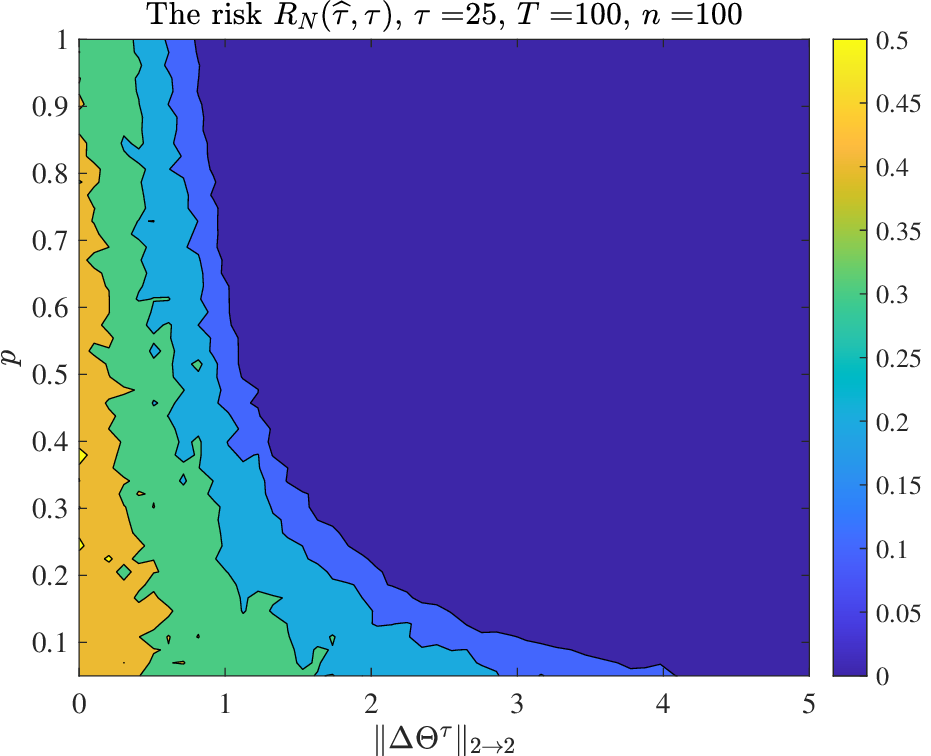}
	\includegraphics[width=0.32\linewidth]{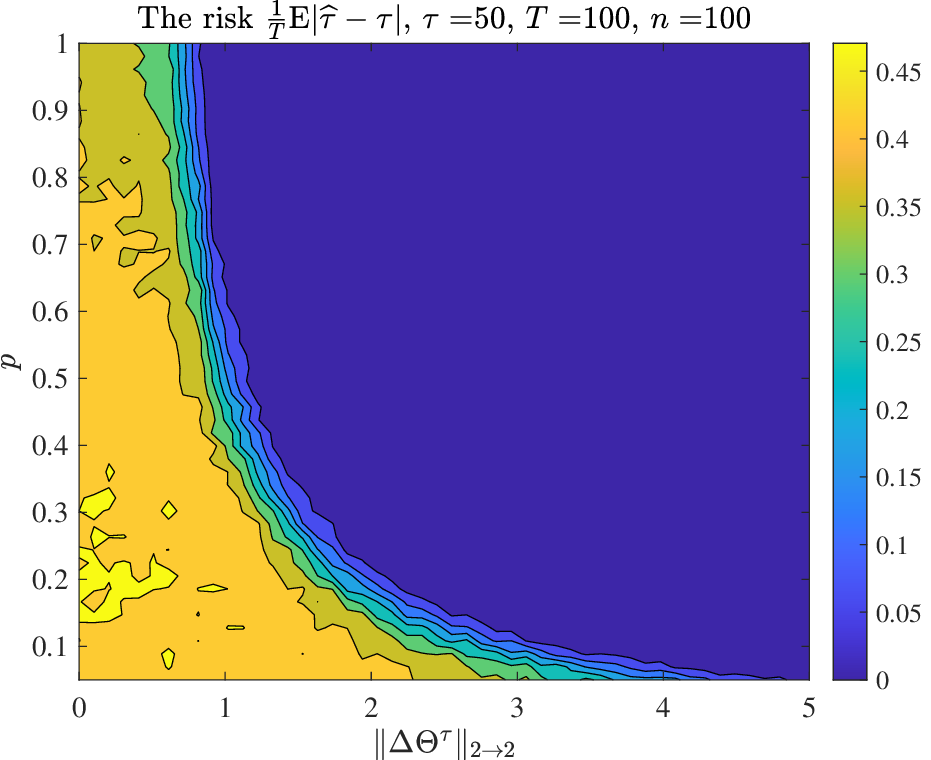}
	\caption{The risk of the change-point estimator under Scenario 1 for $T=100$, $n=100$ and the change-points $\tau\in\{5,25,50\}$ (left to right). The  links are observed at the constant sampling rate $p_n\in(0,1]$. The estimation is easier when the change-point is located in the middle (the graph to the right).}
	\label{fig:risk_SBMmodel2}
\end{figure}

\subsubsection{Change-point localization for  non-uniform sampling patterns}

We have simulated the networks of size $n=100$ with the change point in the middle of $T=100$ observations following three different sampling patterns:
\begin{description}
\item[\it Setting~A. Change-point in missing communication.]  The networks before and after the change follow Scenario~1.
% with two balanced communities, the connection probabilities are given by 
%$$
%Q_1=\rho_n\begin{pmatrix}
%0.6 & 1\\
%1 & 0.6
%\end{pmatrix},\quad 	Q_2=\rho_n\begin{pmatrix}
%0.6 & \delta\\
%\delta & 0.6
%\end{pmatrix},\quad \delta\in[0,1].
%$$
% and some membership matrix $Z\in[0,1]^{n\times 2}$. 
The sampling matrix $\Pi=Z^T\tilde\Pi Z$ has the same community structure as the networks before and after the change and follows ``missing in communication" pattern: $\tilde \Pi = \begin{pmatrix} 1 & p\\ p& 1\end{pmatrix}$. This example was considered in Section~\ref{sec:misslinks}. 
 \item[\it Setting~B. Change-point in communication, within groups missing values]\ The networks before and after the change follow Scenario~1.
 %with two balanced communities
% and some membership matrix $Z\in[0,1]^{n\times 2}$. 
% The change-point is in the 
 The sampling matrix $\Pi=Z^T\tilde\Pi Z$ has the same community structure as the networks before and after the change and follows ``missing within groups" pattern: $\tilde \Pi = \begin{pmatrix} p & 1\\ 1& p\end{pmatrix}$. 
 \item[\it Setting~C. Change-point and missing links withing communities.] The networks before and after the change follow Scenario~5. %with two balanced communities, the connection probabilities are given by 
% $$
% Q_1=\rho_n\begin{pmatrix}
% 0.6 & 1\\
% 1 & 0.6
% \end{pmatrix},\quad 	Q_2=\rho_n\begin{pmatrix}
% \delta & 1\\
% 1 & \delta
% \end{pmatrix},\quad \delta\in[0,1].
% $$
% and some membership matrix $Z\in[0,1]^{n\times 2}$. 
The sampling matrix $\Pi=Z^T\tilde\Pi Z$ has the same community structure as the networks before and after the change and follows ``missing in communication" pattern: $\tilde \Pi = \begin{pmatrix} p & 1\\ 1& 1-p\end{pmatrix}$. 
\end{description} 
 In Fig.~\ref{fig:risk_SBMmodel2_miss} we present the results of the simulations that show the dependence of the risk of estimation on the sampling probability $p$ and the jump norm $\|\Delta\Theta\|_{2\to2}$. 
  Under Setting~A, the distortion parameter  is $\delta_n(\Pi,\Theta)=p/\sqrt 2$ and we see that the level curve is given by $p\|\Delta\Theta^\tau\|_{2\to 2}\approx\const$ since the links between the communities with a change-point are sampled at the uniform rate $p$.
 Under Setting~B, the missing links do not affect the change-point estimation since $\|\Pi\odot \Delta\Theta\|_{2\to 2}=\|\Delta\Theta\|_{2\to 2}$, the level curve is constant and the change-point estimation risk is independent of the link sampling.
Finally, under Setting~C, the distortion parameter is equal to $\delta_n(\Pi,\Theta)=(p\wedge 1-p)/\sqrt 2$ and we see a different level curve $(p\vee 1-p)\|\Delta\Theta^\tau\|_{2\to 2}\approx\const$ with the change in connection probabilities within blocks and links observed with different sampling probabilities $p$ and $1-p$. 
\begin{figure}[htbp!]
	\centering
	\includegraphics[width=0.32\linewidth]{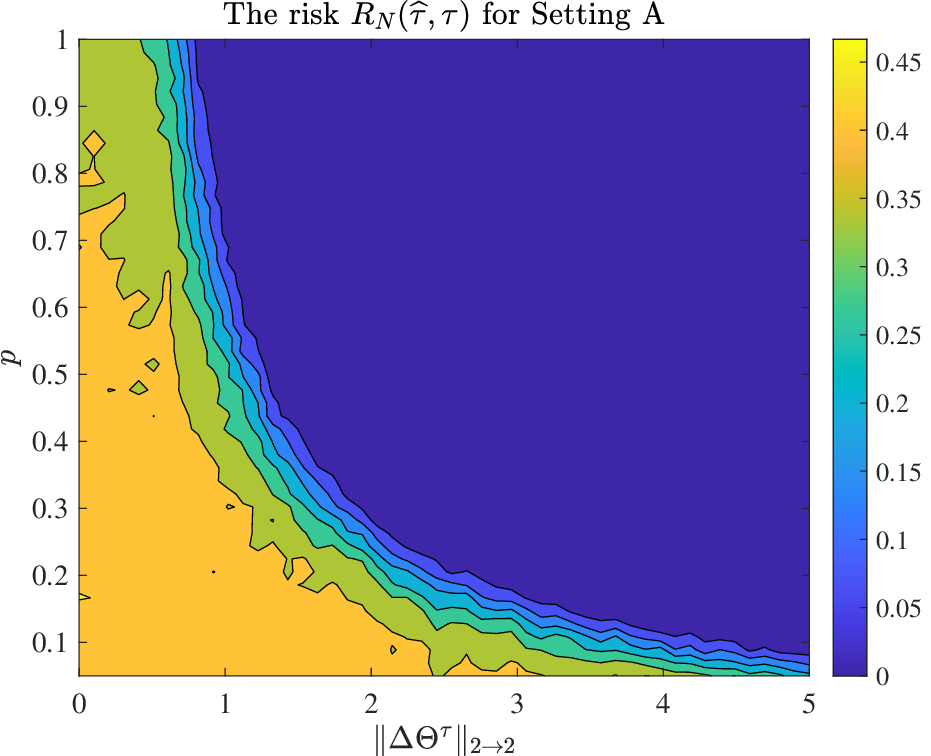}
	\includegraphics[width=0.32\linewidth]{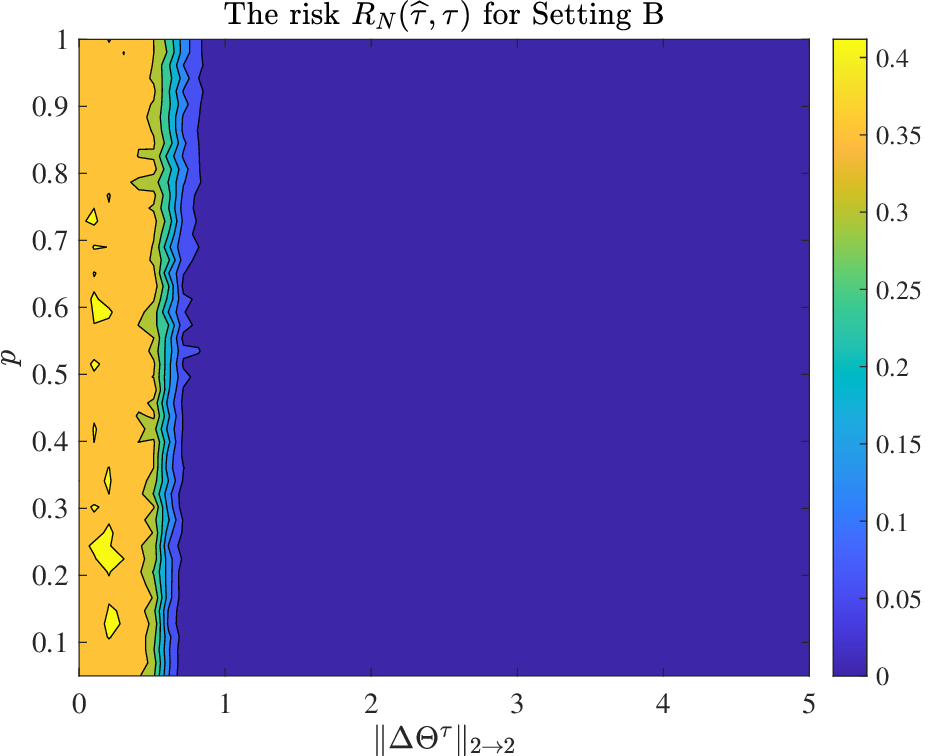}
	\includegraphics[width=0.32\linewidth]{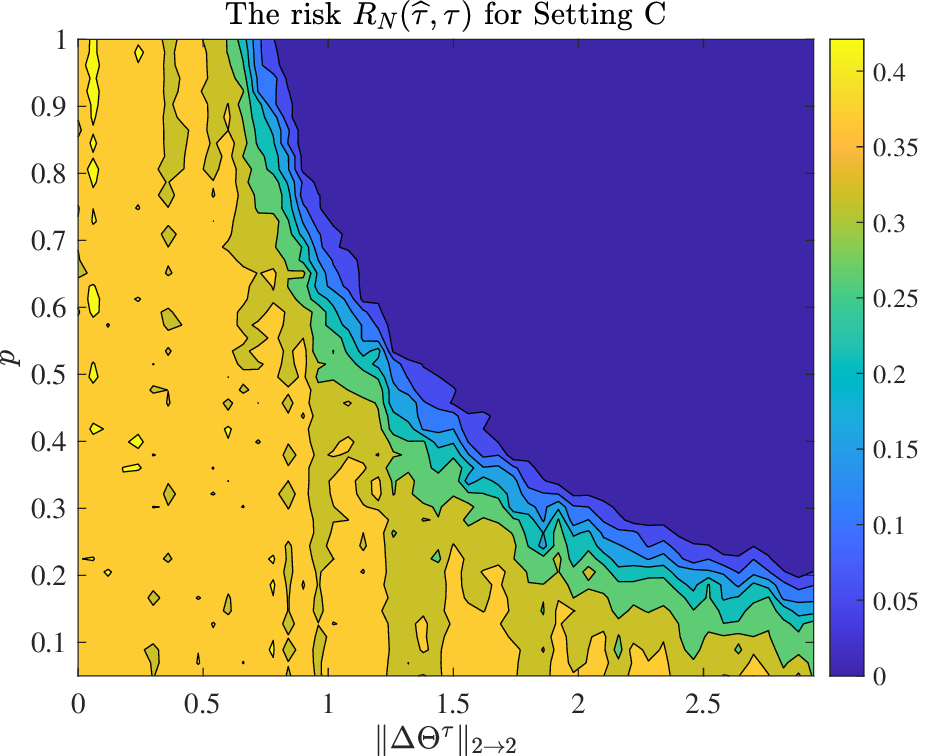}
	\caption{The risk of the change-point estimator under three different patterns of missing links for $T=100$, $n=100$ and $\tau=50$ (Setting~A--C, from left to right).}\label{fig:risk_SBMmodel2_miss}
\end{figure}

\subsubsection{Change-point detection under temporal dependence}
Following the reviewer's suggestion, we study the robustness of our testing procedure  to the temporal dependency in the observations. We observe a realization of the network $A=\{A^t,\ 1\le t\le T\}$ without missing links  and with temporal dependence for given nodes $i,j$. We suppose that  the process $A_{ij}^t$ is a Markov chain with values in $\{0,1\}$ and the stationary connection probabilities $\Theta_{ij}^0$ for $1\le t\le \tau$ and $\Theta_{ij}^1$ for $\tau<t\le T$.  The  Markov model for Bernoulli trials that we will use was proposed and studied  by~\cite{klotz1973}. Let $\lambda\in(0,1)$ be some given value and 
$$
\Pi_{ij}=\begin{pmatrix} 
	1-(1-\lambda)\frac{\Theta_{ij}}{1-\Theta_{ij}} & (1-\lambda)\frac{\Theta_{ij}}{1-\Theta_{ij}} \\
	1-\lambda & \lambda
\end{pmatrix}
$$
be the transition matrix of the Markov chain $(A_{ij}^t)$ for given $i$,$j$ with the connection probability $\Theta_{ij}=\Theta_{ij}^0\one_{\{1\le t\le\tau\}}+\Theta_{ij}^1\one_{\{\tau< t\le T\}}$.
Note that for $\lambda=\Theta_{ij}$ we have $\Pi_{ij}=\Theta_{ij}$ and the observations $A_{ij}^t$ are independent. 
We have $\P\Bigl\{A_{ij}^{t+1}=1 | A_{ij}^{t}=1\Bigr\}=\lambda$ and 
$$
\P\Bigl\{A_{ij}^{t+1}=1 | A_{ij}^t=0\Bigr\}= (1-\lambda)\frac{\Theta_{ij}}{1-\Theta_{ij}}.
%\begin{cases} (1-\lambda)\frac{\Theta^0_{ij}}{1-\Theta^0_{ij}}& 1\le t<\tau\\
% 	(1-\lambda)\frac{\Theta^1_{ij}}{1-\Theta^1_{ij}} & \tau\le t<T.
% 	\end{cases}
$$
It can be easily seen that if $\lambda>\Theta_{ij}$ the probability to observe a link between $i$ and $j$ given the presence of  a link in the past is greater than $\Theta_{ij}$ and the probability of absence of link given the absent link in the past is greater than $1-\Theta_{ij}$. Thus, in this case we will have the observations $\{A_{ij}^t, 1\le t\le T\}$ of zeros and ones that form clusters. Vice versa, if $\lambda<\Theta_{ij}$, it will be less probable that the observations form clusters. The correlations between $X_t$ and $X_{t+h}$ are given by $\rho_{ij}(X_t,X_{t+h})=\Bigl(\frac{\lambda-\Theta_{ij}}{1-\Theta_{ij}}\Bigr)^{|h|}$, $h\in\mathbb Z$ and decrease exponentially with~$h$. It suggests that for small values of $\lambda$, when the dependence becomes weak, the testing procedure would give better results rather than for the values of $\lambda$  that are close to one. 

We will use the Matrix CUSUM test statistic for detection of a change-point at a given location for different degrees of dependency. In order to check the test behavior under the dependency, we will trace the ROC curves for $\lambda\in\{2^{-k},\ 1\le k\le 5\}$ and compare them to the ROC curve obtained for the independent case. We have performed 250 simulations of networks with and without a change following Scenarios 1 and 4 for $n=100$, $T=100$, and with the change-point located in the middle. Then, according to the values of the test statistic, we calculated the True Positive Rate (TPR), corresponding to the correctly detected alternative hypothesis of presence of a change and the False Positive Rate (FPR) that corresponds to incorrectly rejected null hypothesis. In Fig.~\ref{fig:ROC_dependent} the ROC curves are presented for the SBM networks following Scenario 1 and 4. 
It turns out that for $\lambda=1/2$, when the dependency is high, the test  behaves almost as a random guess. On the other hand, for small values of $\lambda=1/32,1/16$ (the blue curves),  the test statistic works better on the dependent data  than for the independent case. It is known that the testing procedure can benefit from the presence of dependency, see, for example, \citep{enikeeva2020} where the influence of  dependency on the change-point detection was studied in the case of Gaussian time series. Finally, we obseve that for $\lambda=1/8$ the test statistic have the same performance as in the independent case.

\begin{figure}[htbp!]
\centering
\includegraphics[width=0.48\linewidth]{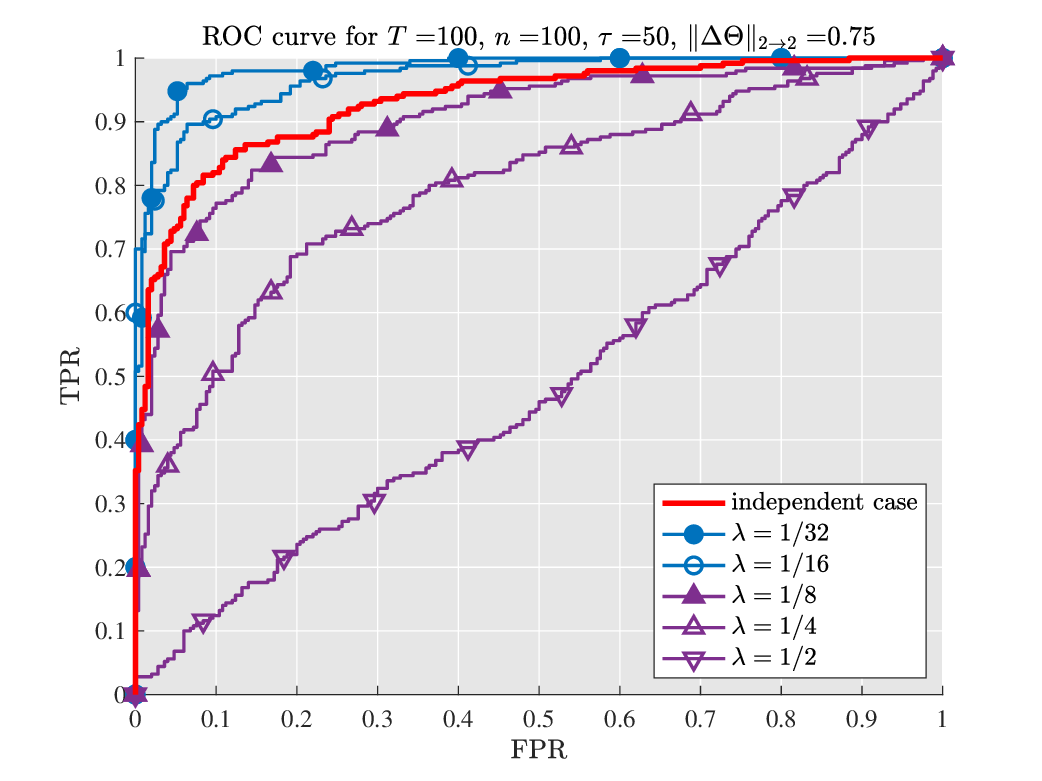}
\includegraphics[width=0.48\linewidth]{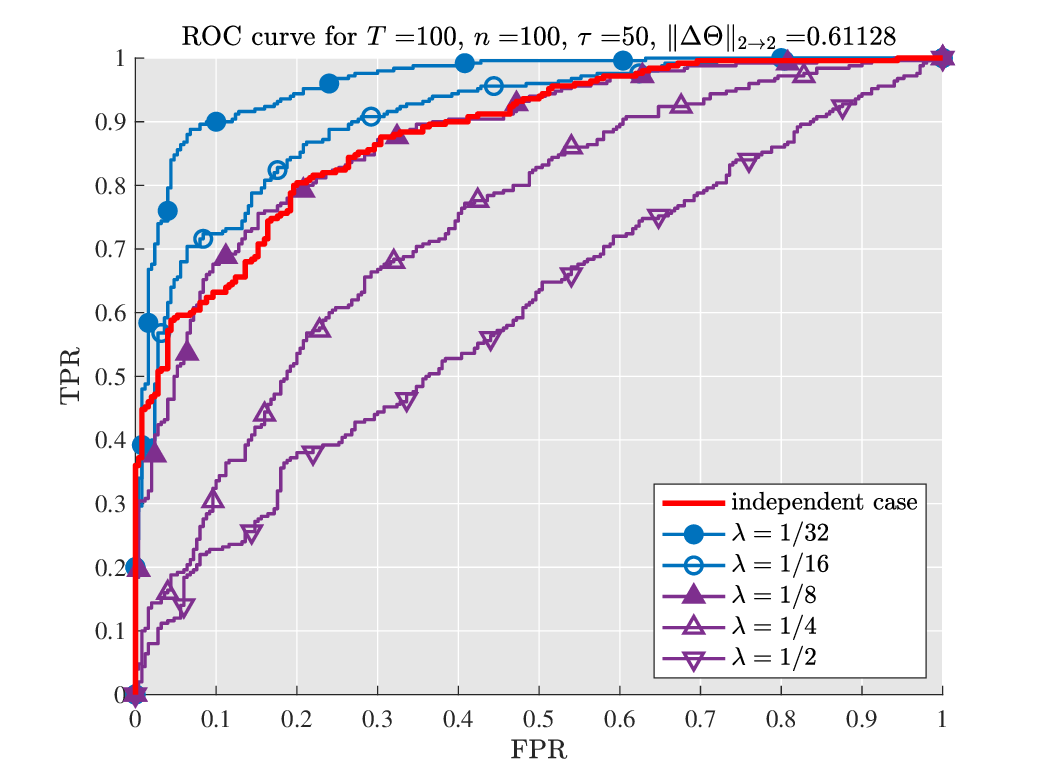}
\caption{The test power for the  Markov dependent networks with $n=100$ vertexes and $\lambda =0.6$ following Scenario 1 (on the left) for $T=100$, and Scenario 4 (on the right)  for $T=250$.}
\label{fig:ROC_dependent}
\end{figure}
We have also tested the influence of the dependency on the estimation procedure. We performed 100 simulations of networks following Scenarios 1, 4 and 5 for $T=100$, $n=100$ and the change-point in the middle. In Fig.~\ref{fig:est_dependent} we present  the empirical risk of the change-point estimator $\widehat\tau$ depending on the norm of the  jump $\|\Delta\Theta\|_{2\to2}$ and on  $\lambda\in (0,1)$.
\begin{figure}[htbp!]
	\centering
	\includegraphics[width=0.32\linewidth]{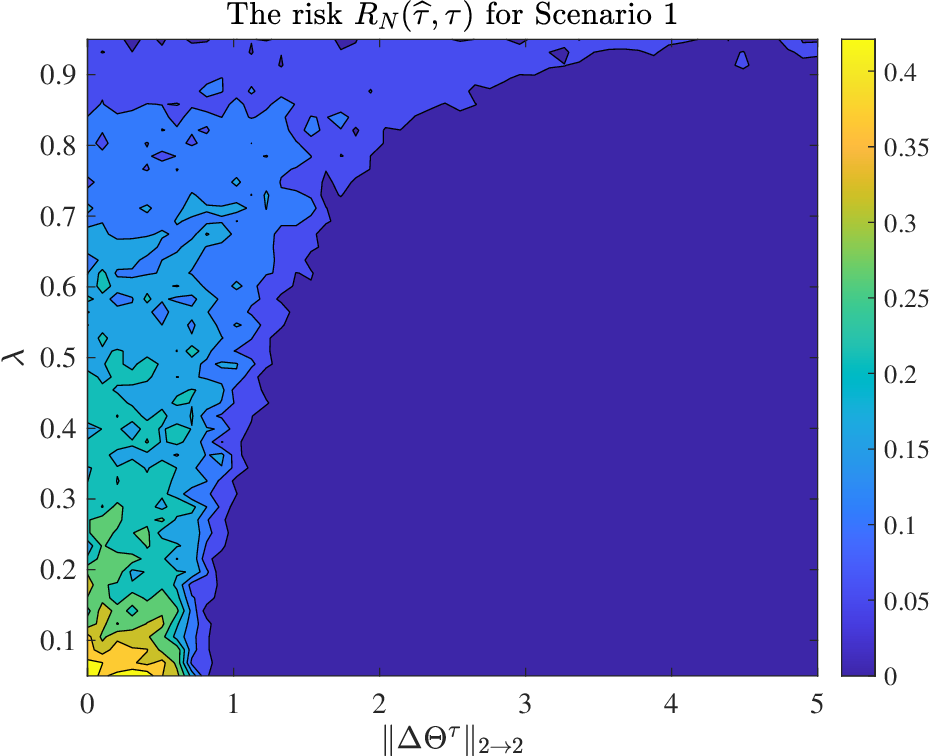}
	\includegraphics[width=0.32\linewidth]{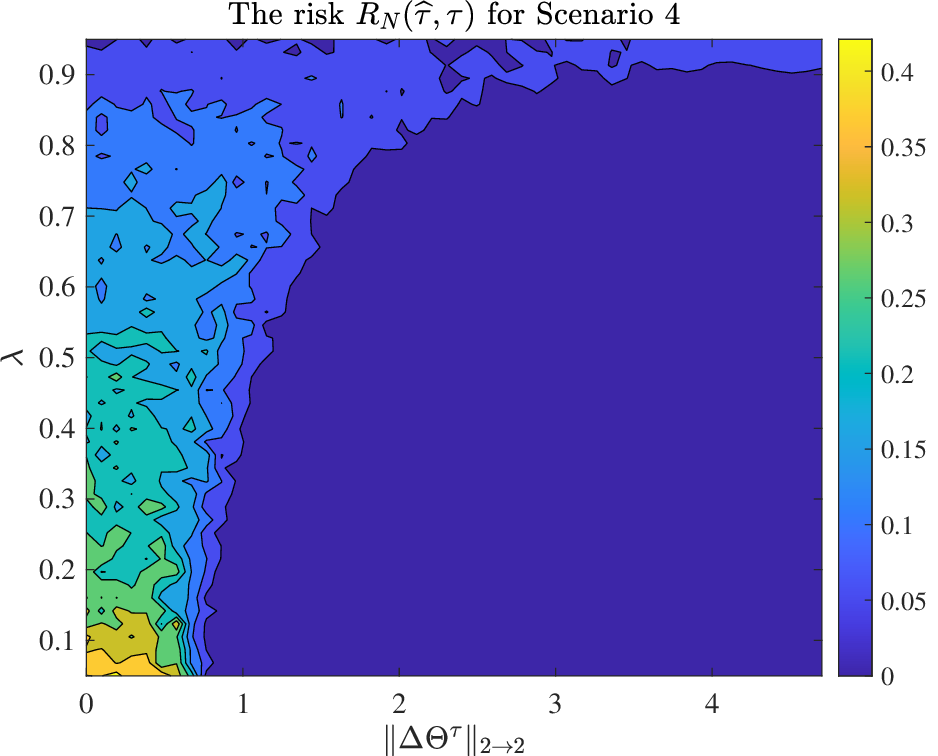}
	\includegraphics[width=0.32\linewidth]{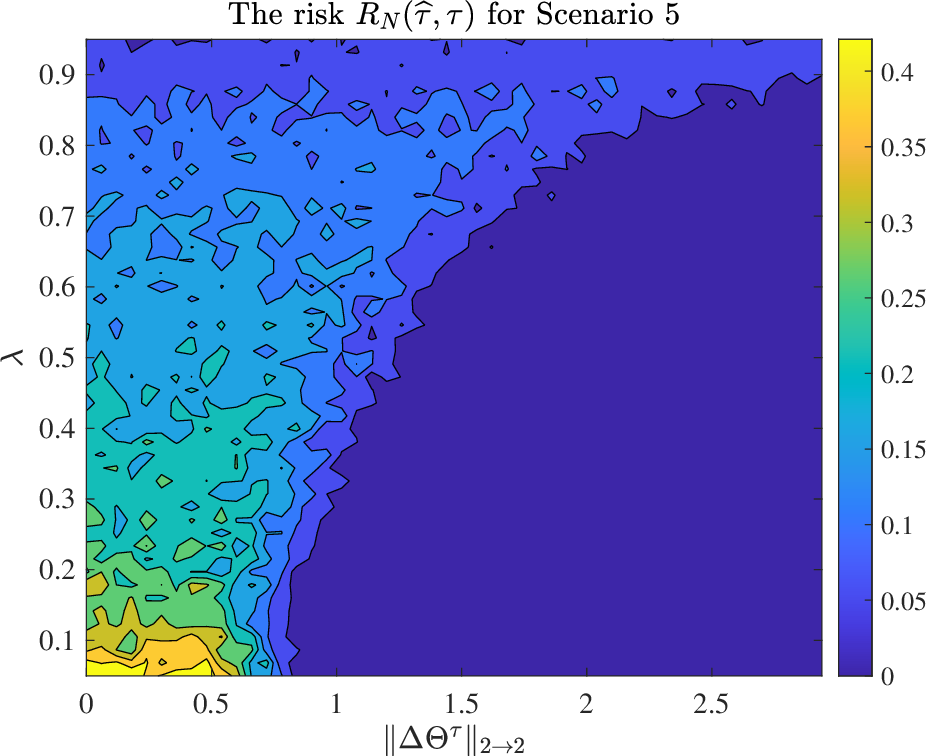}
	\caption{The risk of the change-point estimator  for $T=100$, $n=100$ and the change-point in the middle, $\tau=50$  for Scenarios 1, 4, and 5.}
	\label{fig:est_dependent}
\end{figure}
We see that the results for all scenarios show the same dependence of the risk on $\lambda$ and $\|\Delta\Theta\|_{2\to2}$.  In the detectability zone (deep blue and blue zones),  we see that the performance of the estimator increases if $\lambda$ decreases, as in the case of testing.  If we compare the result for Scenario 1 to the one for the independent case presented in Fig.~\ref{fig:risk_SBMmodel2_miss} in the middle (there is no influence of the missing links here), we see that beyond the detectability zone when $\|\Delta\Theta\|_{2\to2}<0.5$, the change-point estimation benefits from the dependency if $\lambda$ increases.  Overall, the Matrix CUSUM statistic shows a pretty good performance for this model of dependency both in terms of estimation and testing. 
%{\color{red} \bf The study of the influence of parameter $\lambda$ on the risk involves the proof of the concentration inequalities of the Matrix CUSUM process with dependency. This analysis is left for future work.}
\subsection{Results for graphon model}
%We have considered the following graphons: 
%\begin{align*}
%W_1(x,y)&=\frac12\sin\bigl(5\pi (x+y-1)+1\bigr)+0.5,\\
%W_2(x,y)&=\frac13(x^2+y^2)\cos\bigl(1/(x^2+y^2)\bigr)+0.15,\\
%W_3(x,y)&=1-\Bigl(1+\exp\bigl\{15(0.8|x-y|)^{0.8}-0.1\bigr\}\Bigr)^{-1},\\
%W_4(x,y)&=xy,\quad  W_5(x,y)=x^2y^2,\quad W_6(x,y)= x^{1/2}y^{1/2}.
%\end{align*}
%Their $L_2$-norms are $\|W_1\|_2=0.6179$, $\|W_2\|_2=0.2637$, $\|W_3\|_2=0.9374$, $\|W_4\|_2=1/3$, $\|W_5\|_2=1/5$, $\|W_6\|_2=1/2$.  The graphons $W_1$, $W_2$ and $W_3$ are borrowed from~\citep{Levinagraphon:2017}. 

In this section, we simulate a dynamic network  from graphon model with the graphon function in H\"older classes with
$$
W_1(x,y)=xy \quad\mbox{before the change and}\quad W_2^\gamma(x,y)=(xy)^\gamma \quad \text{after the change.}
$$
Here $\gamma\ge 1$ is the smoothness parameter that defines the impact of the change. 
%In each of 100 simulations we sample the node assignment vectors $\epsilon\in[0,1]^n$ with uniform i.i.d.  coordinates. 
We suppose that the assignment vector  $\eps$ does not change.
% For each pair of graphons $W_1$, $W_2^\gamma$, we generate a dynamic network with the connection probability matrices $(\Theta^1_{ij})_{1\le i,j\le n}=\rho_n \Bigl(W_1(\eps_i,\eps_j)\Bigr)_{1\le i,j\le n}$ and $(\Theta^2_{ij})_{1\le i,j\le n}= \rho_n\Bigl(W_2^\gamma(\eps_i,\eps_j)\Bigr)_{1\le i,j\le n}$ before and after the change-point $\tau$. 
The sparsity parameter is set to  $\rho_n=1/\sqrt n$. 
%Since the vector $\eps$ does not change with time, we can use the same tests for inhomogeneous random graph that were used for the Stochastic Block Models in Section~\ref{sec:SBMsim} with the same estimator of sparsity $\widehat{\kappa_n}$. 
The smoothness parameter $\gamma$ varies from 1 to 5, the change-point is at the middle of the interval, that is $\tau=T/2$, and the number of observations is $T=100$. The matrix size varies from 50 to 400. 

On the left hand side of Fig.~\ref{fig:graphon} we see  the dependence of the power of three tests $\psi_{n,T}^\tau$,  $\psi_{n,T}$, $\psi_{n,T}^{full}$ on the smoothness parameter $\gamma$. 
The results are similar to those obtained for the SBMs, the test over the dyadic grid outperforms the test over the whole grid and both are less powerful that testing at a given change-point $\tau$. The change in graphons with smoothness $\gamma>3$ can be detected with power close to 1. The graph on the right hand side of Fig.~\ref{fig:graphon} shows the power of the test $\psi_{n,T}^\tau$ for different sizes of networks $n\in\{50,100,200,300,400\}$ and for different values of the smoothness $\gamma$. We can see that the detection power grows with $n$ that confirms the detection rate $1/\sqrt{n\rho_nT}$. On the other hand, the smaller is the smoothness $\gamma$, the harder the detection will be. For example, if $\gamma=4$, the detection power is 1 starting from $n=100$ and for $\gamma=2$, the detection power becomes close to 1 only for $n= 400$. 

\begin{figure}[htbp!]
	\centering
	\includegraphics[width=0.48\linewidth]{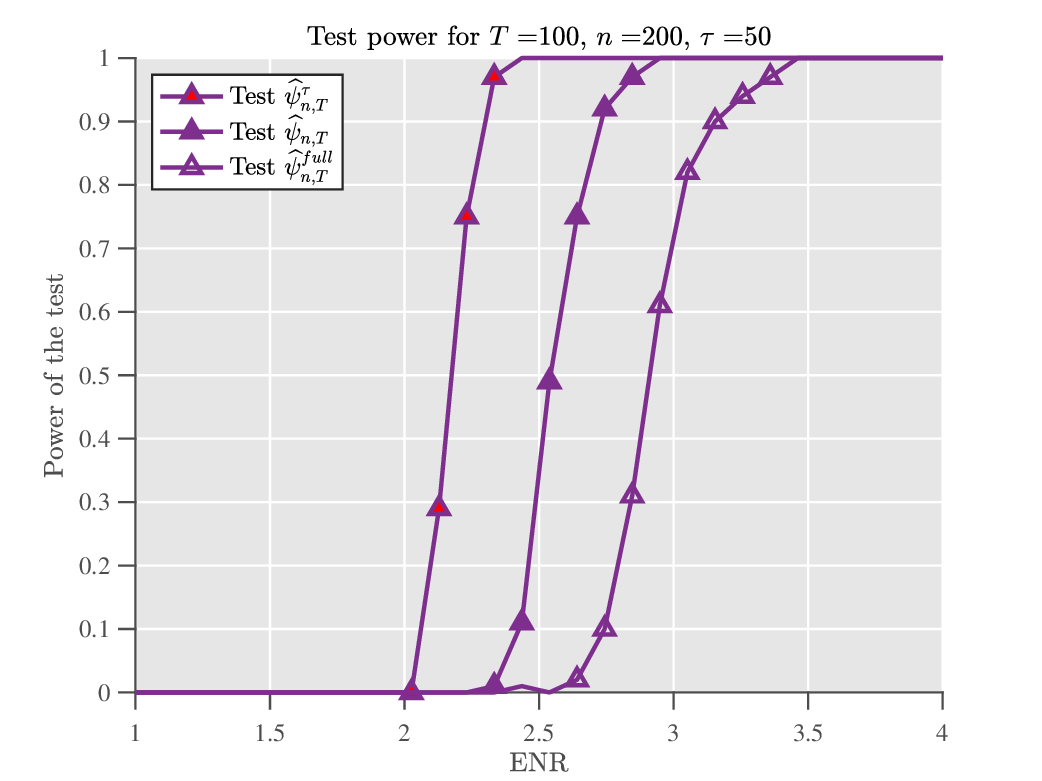}
	\includegraphics[width=0.48\linewidth]{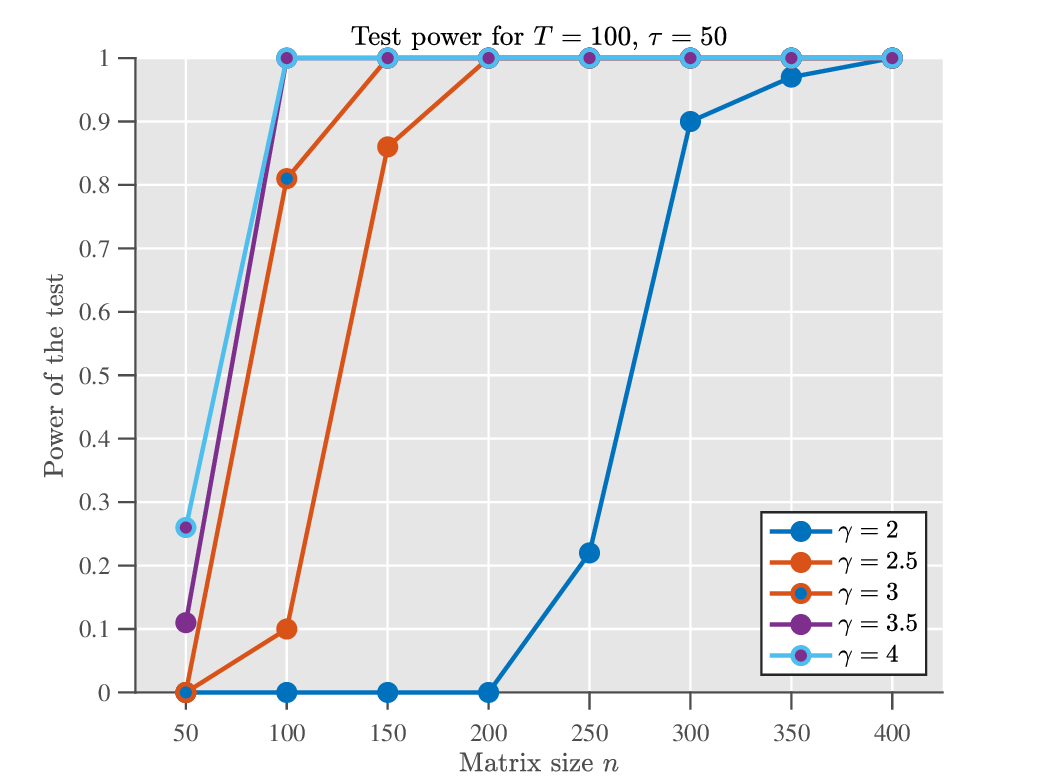}
	\caption{Power of  testing the change in the H\"older class graphons at $\tau=T/2$, $T=100$. The graph to the left displays the power for three different test for $n=200$. The graph to the right shows the power of the test $\psi_{n,T}^\tau$ for different values of $\gamma$ depending on the matrix size $n$. }
	\label{fig:graphon}
\end{figure}

\subsection{Transport for London (TfL) Open Data}\label{London}
In this section, we apply our test to the real data coming from the Transport for London (TfL) Open Data API\footnote{Acces to the data via \url{https://api.tfl.gov.uk}}. The data contains information about London Bicycle Sharing Network collected since 2012. The dataset contains  the following information: the ID of each bicycle, the ID and name of the origin and the destination trip stations, the journey (rental) starting and ending time and date, and the unique ID and the duration of each trip. 

We have analyzed the data during the two-month period from June 24, 2012 to August 31, 2012. The  summer of 2012 is a remarkable  period because of the Games of the XXX Olympiad that was held from July 27 to August 12, 2012 in London.
The dynamic network is a sequence of $T=69$ daily observations. Each observation is a graph with  $n=595$ vertices corresponding to the bike rental stations. We say that two vertices are connected if the minimal trip duration between the corresponding stations is not less than 3 minutes and the number of trips is greater than a predefined threshold. For each day, the threshold on the number of trips is equal to the 0.9975-level  empirical quantile of the distribution of the total  number of trips between every couple of stations excluding disconnected stations (zero trips during the day). The obtained network has the average  sparsity  $\bar\kappa_n=43.2319$ (over $T=69$ observations). The corresponding value of $\rho_n=\kappa_n/n=0.0727\asymp n^{-0.4}$. %The obtained network is sparse. 

 Fig.~\ref{fig:cusumolympicsperiod} on the left shows the graph of the matrix CUSUM statistic calculated over the whole period from June 24, 2012 to August 31, 2012. We can see that the maximum of the statistic is attained at the position corresponding to July 22, 2012. This date corresponds to the day of the arrival of the Olympic Torch to London.\footnote{The details about the traffic perturbation in London on July 22, 2012 can be found at the TfL website: \url{https://tfl.gov.uk/info-for/media/press-releases/2012/july/olympic-torch-relay-has-arrived-in-london--plan-your-travel-and-get-ahead-of-the-games-tomorrow--sunday-22-july-2012}.}. 
Our test detects this change-point at the significance level $\alpha=0.05$ and our estimator correctly estimates it. The value of the test statistic is 53.3311, the corresponding threshold is equal to 40.5244.

\begin{figure}[htbp!]
	\centering
	\includegraphics[width=0.495\linewidth]{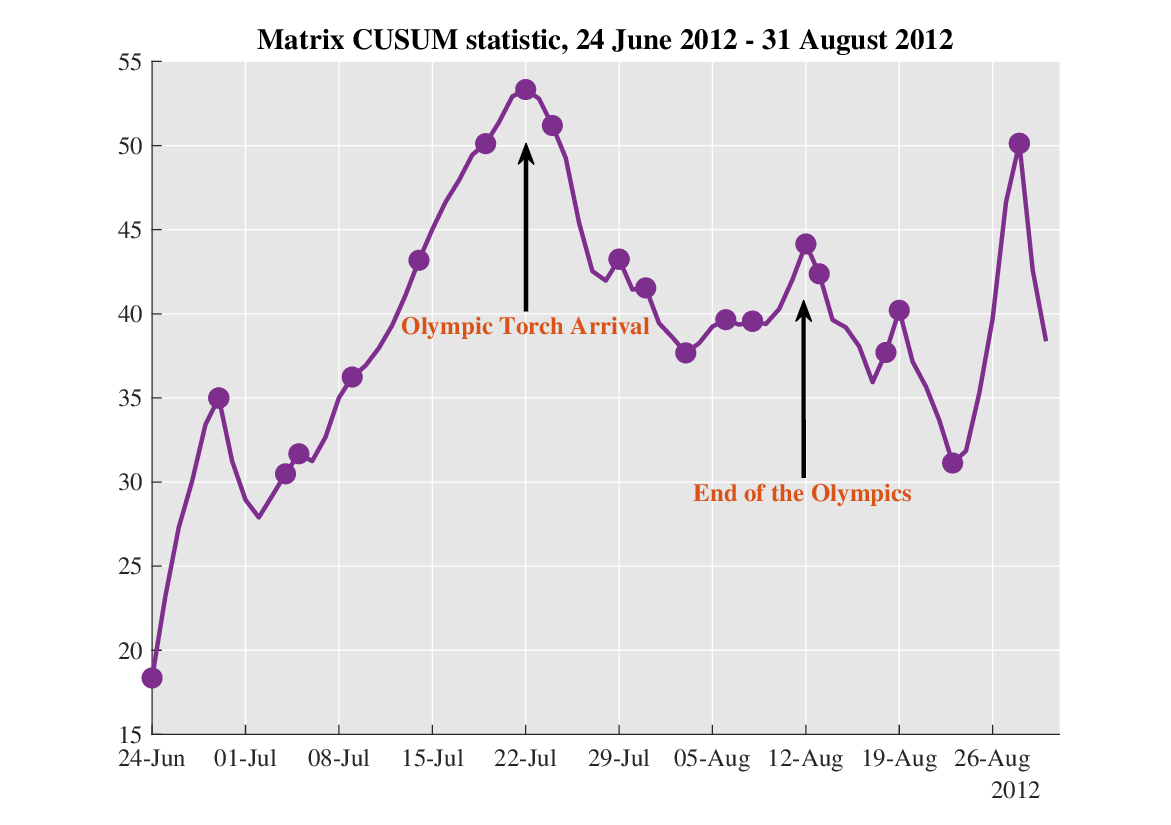}
	\hfill
	\includegraphics[width=0.495\linewidth]{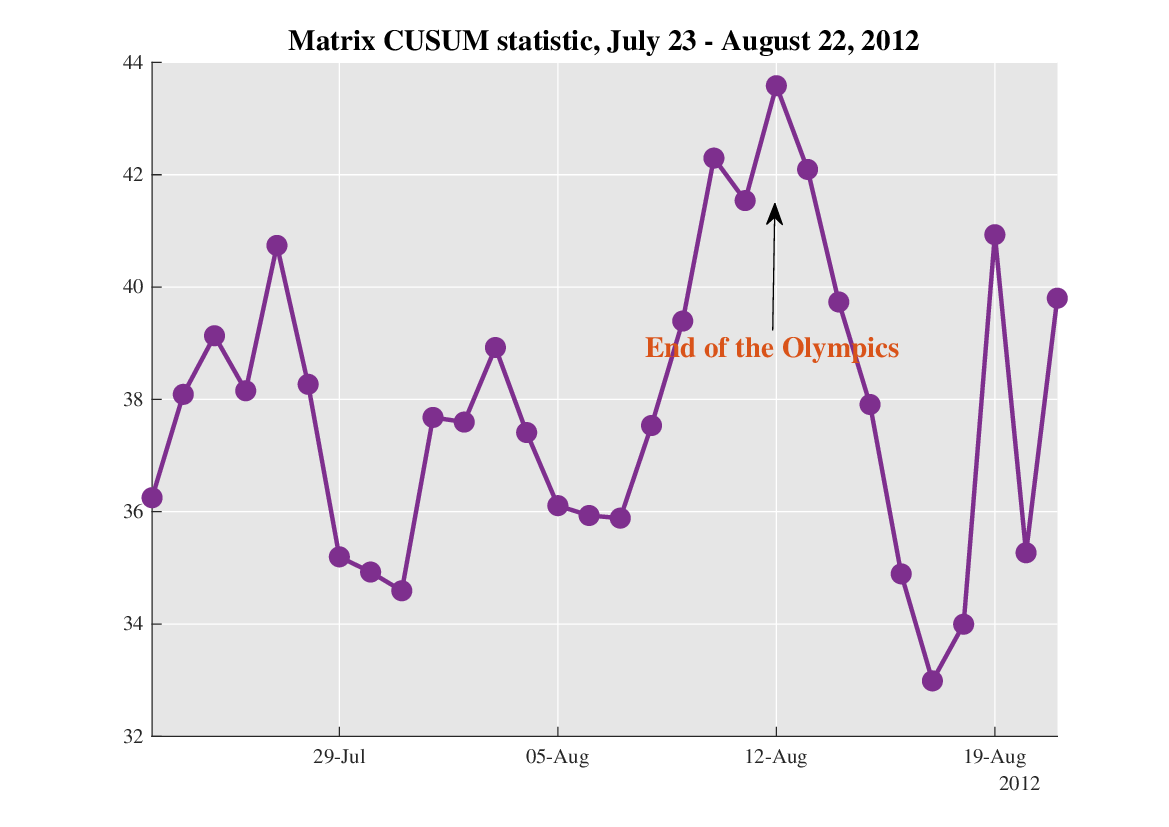}
	\caption{On the left: values of the matrix CUSUM statistic calculated during the whole period of observations. On the right: values of the matrix CUSUM statistic calculated during 31 days from July, 23 to August, 22.}
	\label{fig:cusumolympicsperiod}
\end{figure}

We see several peaks on the graph of the matrix CUSUM statistics which may imply that actually this data exhibits several change points.  One of them corresponds to the end of the Olympics on August 12, 2012. It is possible to combine our test with the segmentation methods for  multiple change-point localization (see, for example,  SMUCE \citep{fms14}, WBS \citep{Fryzlewicz:2014} or the method proposed recently in \citep{verzelen2020optimal}).

For example, if we take  the data covering the  period from July 23 to August, 22 (see Fig.~\ref{fig:cusumolympicsperiod}), 
%so that we observe the network during $T=31$ days. 
 our test detects the change-point corresponding to the date of the London Olympics closing ceremony on August 12, 2012. Here we  observe the network during $T=31$ day and the values of the test statistic and the corresponding threshold at the level $\alpha=0.05$ are, respectively, 43.5854 and 41.7997. The estimator estimates correctly the change-point. 
%
%\begin{figure}[htbp!]
%	\centering
%	\includegraphics[width=0.7\linewidth]{figures/CUSUM_olympics_period}
%	\caption{The value of the matrix CUSUM statistic calculated during 31 day from July 23 to August, 22.}
%	\label{fig:cusumwholeperiod}
%\end{figure}

\bibliographystyle{apalike}
\bibliography{references.bib}

\appendix
\section{Upper bound results}
In this section, we provide the Lemmas that control the type I and II errors of the tests for problems (P1) and (P2).

\subsection{Testing at a given point}
%Consider problem of testing the presence of a change at some given point $\tau$ from fully observed data $Y=(Y^1,\dots,T^T)$, where the observations $Y^t=A^t$ follow~(\ref{complete_model}). 
%The decision rule for problem (P1) is
%\begin{equation}\label{eq:test:P1_wo_log}
%\psi_{n,T}^\tau(A)= \one \left \{\|Z_{T}(\tau)\|_{2\rightarrow 2}>H_{\alpha,n}\right \},
%\end{equation}
%where the threshold is defined by
%\begin{equation}\label{eq:H_P1_wo_log}
%H_{\alpha,n}=2\sqrt 2 (1+\epsilon) \sqrt \kappa_n + C_\epsilon \log\frac {2n} \alpha.
%\end{equation}
%Here $\epsilon\in(0,1/2]$ and $C_\epsilon$ is an absolute constant depending on $\eps$ provided in Lemma~\ref{lem:matrixBernstein}.
%-----------------------------------------------------------------
%The choice of the threshold $H$ guarantees that the global risk error $\gamma(\psi_{n,T},R_{n,\cpset_T})=\beta(\psi_{n,T},R_{n,\cpset_T})+ \alpha(\psi_{n,T})$ tends to zero as $n\rightarrow \infty$:

\begin{lemma}\label{lem:typeIerror_wo_log}
	For any $\alpha\in(0,1)$ the type I error of test~(\ref{eq:test:P1_wo_log}) is bounded by~$\alpha$.
\end{lemma}
\begin{proof} Recall that $Z_T(t)=-\mu_T(t)\Pi\odot \Delta\Theta^\tau + \xi(t)$, where 
	$$
	\xi(t)=\sqrt{\frac{t(T-t)}T} \left(\frac 1t \sum_{s=1}^t W^s -\frac1{T-t}\sum_{s=t+1}^T W^s\right)
	$$
	and that  under the null hypothesis $W^s=(W_{ij}^s)\in[-1,1]^{n\times n}$ are independent centered Bernoulli matrices with independent entries taking values in $\{1-\Pi_{ij}\Theta_{ij}^0,-\Pi_{ij}\Theta_{ij}^0\}$ with the success probability $\Pi_{ij}\Theta_{ij}^0$. 
	We have
	\begin{align*}
	\alpha(\psi_{n,T}^\tau)&=\sup_{(\Theta_0,\Pi)\in S_n(\omega_n)}\P_{(\Theta_0,\Pi)}\Bigl\{  \|Z_T(\tau)\|_{2\to2}>H_{\alpha,n}\Bigr\}\\ 
	&= \sup_{(\Theta_0,\Pi)\in S_n(\omega_n)}\P_{(\Theta_0,\Pi)}\Bigl\{\|\xi(\tau)\|_{2\rightarrow 2}>2\sqrt 2 (1+\epsilon) \sqrt \omega_n + C_\epsilon \log\Bigl(\frac {2n} \alpha\Bigr)\Bigr\}.
	\end{align*}
	Using the bound from  Lemma \ref{lem:matrixBernstein} with $\Vert \Pi\odot \Theta^0\Vert_{1,\infty}\leq \omega_n$ we get $\alpha(\psi_{n,T})\leq \alpha$.
	%taking $s=\tau$ in Lemma~\ref{lm:matrixBernstein} we immediately obtain $\alpha(\psi_{n,T})\leq \alpha$.
\end{proof}
%-----------------------------------------------------------------------
\begin{lemma}\label{lem:typeIIerror_wo_log}
	Let $\alpha,\beta\in(0,1)$ and $H_{\alpha,n}$ be given by \eqref{eq:H_P1_wo_log}. Suppose that 
	\begin{equation}\label{eq:rate_nonasymp}
	\mathcal R_{n,\tau} \ge 4\sqrt 2(1+\epsilon) \left(\frac{\kappa_n}{T}\right)^{1/2} +\frac {C_\epsilon}{\sqrt T}  \left(\log \Bigl(\frac{2n}\alpha\Bigr)+ \log\Bigl(\frac{2n}\beta\Bigr)\right)
	\end{equation}
	Then the type II error of test~(\ref{eq:test:P1_wo_log}) is bounded by $\beta$.
\end{lemma}

\begin{proof}
	For ease of notation we denote 
	$$
	\left (\bTheta,\Delta\bTheta^{\tau}\right )=\left \{ (\Theta+\Delta\Theta^\tau,\Pi ),(\Theta,\Pi ) \right \}.
	$$
	By definition of $\beta(\psi_{n,T},\mathcal R_{n,\cpset_T})$ we have
	\[
	\beta(\psi_{n,T}^\tau,\mathcal R_{n,\tau})
=\sup_{\left (\bTheta,\Delta\bTheta^{\tau}\right )\in\mathcal W_{n,T}^\tau(\omega_n, \mathcal R_{n, \cpset_T}) } \P_{\left (\bTheta,\Delta\bTheta^{\tau}\right )}\left \{\Vert Z_{T}(
	\tau) \Vert_{2\rightarrow 2}\leq H_{\alpha,n}\right \}. 
	\]
	Using the triangle inequality,  \eqref{eq:rate_nonasymp},  the choice of $H_{\alpha,n}$ and the fact that $\mu_T(\tau)=\sqrt Tq(\tau/T)$ we compute
	\begin{align*}
	\beta(\psi_{n,T}^\tau,\mathcal R_{n,\tau})&\leq \sup_{\left (\bTheta,\Delta\bTheta^{\tau}\right )\in\mathcal W_{n,T}^\tau(\omega_n, R_{n, \cpset_T}) } \P_{\left (\bTheta,\Delta\bTheta^{\tau}\right )}\Bigl\{\Vert \xi(\tau)\Vert_{2\rightarrow 2}> \mu_{T}(\tau) \Vert  \Delta \Theta^\tau\Vert_{2\rightarrow 2}-H_{\alpha,n}\Bigr\}\\
	&\leq \sup_{\left (\bTheta,\Delta\bTheta^{\tau}\right )\in\mathcal W_{n,T}^\tau(\omega_n, \mathcal R_{n, \cpset_T}) } \P_{\left (\bTheta,\Delta\bTheta^{\tau}\right )}\Bigl\{\Vert \xi(\tau)\Vert_{2\rightarrow 2}>\sqrt T R_{n,\tau}-2\sqrt 2 (1+\epsilon) \omega_n^{1/2} - C_\epsilon \log\left(\frac {2n} \alpha\right) \Bigr\}\\
	&
	\leq \sup_{\left (\bTheta,\Delta\bTheta^{\tau}\right )\in\mathcal W_{n,T}^\tau(\omega_n, \mathcal R_{n, \cpset_T}) } \P_{\left (\bTheta,\Delta\bTheta^{\tau}\right )}\Bigl\{\| \xi(\tau)\|_{2\rightarrow 2}>2\sqrt 2(1+\epsilon)\omega_n^{1/2} +C_\epsilon \log\left(\frac{2n}\beta\right) \Bigr\}.
	\end{align*}
	Applying  Lemma \ref{lem:matrixBernstein}  and using $\Vert \Pi\odot \Theta^0\Vert_{1,\infty}\leq \omega_n$, $\Vert \Pi\odot (\Theta^0+\Delta\Theta^\tau)\Vert_{1,\infty}\leq \omega_n $ we  obtain the statement of the lemma.
\end{proof}

\subsection{Testing at an unknown change-point}

%\begin{proof}[of Theorem~\ref{th:upper_bound_missing}]
%	The proof is similar to the proof of Theorem~\ref{th:upper_bound:P1_wo_log}.  Lemma~\ref{typeIerrorMissing} implies immediately that $\alpha (\psi_{n,\cpset_T}^{\Omega})\le\alpha$. The asymptotic upper bound  condition (\ref{MinimaxRateP2_missing})  follows from the non-asymptotic rate~(\ref{MinimaxRateP2Missing}). By bounding the residual term as in case of problem (P1),  we get the following general asymptotic upper bound condition in terms of the sparsity $\omega_n$:
%	$$
%	\liminf_{n\to\infty} \min_{\tau\in \cpset_T} q(\tau/T)\frac{T K^2(\Pi)R_{n,\cpset_T}^2}{\omega_n} \ge 72
%	$$
%	and prove the theorem.
%	%Considering the case of $\cpset_T=[1:T-1]$ and $\cpset_T=[\varkappa_T,T-\varkappa_T]$, we %obtain  conditions (\ref{MinimaxRateP2_missing})
%	% and (\ref{MinimaxRateP2$'$_missing}).
%\end{proof}

%Let $\mathcal T$ be the dyadic grid on the set $\cpset_T$. Recall that $|\mathcal T|=2+2\lfloor \log_2(T/2)\rfloor \le 2\log_2(T)$.  Define the decision rule based on the maximum of the matrix CUSUM statistic over the dyadic grid:
%\begin{equation}\label{test_unknowntau_wolog}
%\psi_{n,T}(Y)=\one{\{\max_{t\in\mathcal T}\|Z_T^Y(t)\|_{2\to2} >H_{\alpha,n,T}\}}
%\end{equation}
%where the threshold is given by 
%\begin{equation}\label{threshold_unknowntau_wolog}
%H_{\alpha,n,T}=2\sqrt 2 (1+\epsilon) \omega_n^{1/2} + C_\epsilon \log \left(\frac{4n \log_2(T)}{\alpha}\right).
%\end{equation}

\begin{lemma}\label{lem:typeIerrorMissing_wolog}
	For any $\alpha\in(0,1)$ the type I error of test~(\ref{test_unknowntau_wolog}) is  less than~$\alpha$.
\end{lemma}
\begin{proof}
	Using the union bound we can bound the type I error as follows
	\begin{align*}
	\alpha(\psi_{n,T})& =\P_{\Hyp_0}\Bigl\{\max_{t\in\mathcal T}\|Z_T^Y(t)\|_{2\to2} >H_{\alpha,n,T}\Bigr\}\\
	&\le \sum_{t\in\mathcal T} \sup_{(\Theta^0,\Pi)\in\mathcal S_n(\omega_n)} \P_{(\Theta^0,\Pi)} \Bigl\{\|Z_T^Y(t)\|_{2\to2} >H_{\alpha,n,T}\Bigr\}\\
	&= \sum_{t\in\mathcal T}	\sup_{(\Theta^0,\Pi)\in\mathcal S_n(\omega_n)} \P_{(\Theta^0,\Pi)} \Bigl\{\| \xi(t)\|_{2\rightarrow 2}>2\sqrt 2 (1+\epsilon) \omega_n^{1/2} + C_\epsilon \log \left(\frac{4n \log_2(T)}{\alpha}\right)\Bigr\}.
	\end{align*}
	
	Using Lemma~\ref{lem:matrixBernstein} with $\Vert \Pi_n\odot\Theta^0\Vert_{1,\infty}\leq \omega_n$, we immediately get for every $t\in\mathcal T$
	$$
	\sup_{(\Theta^0,\Pi)\in\mathcal S_n(\omega_n)} \P_{(\Theta^0,\Pi)} \Bigl\{\| \xi(t)\|_{2\rightarrow 2}>2\sqrt 2 (1+\epsilon) \omega_n^{1/2} + C_\epsilon \log \left(\frac{4n \log_2(T)}{\alpha}\right)\Bigr\} \le \frac{\alpha}{2\log_2(T)}.
	$$
	Using the fact that $|\mathcal T|\le 2\log_2 (T)$, we obtain $\alpha (\psi_{n,T})\le\alpha$.
\end{proof}
%-----------------------------------------------------------------------
\begin{lemma}\label{lem:typeIIerror_missing_wolog}
	Let $\alpha,\beta\in(0,1)$ and $ H_{\alpha,n,T}$ be given by \eqref{threshold_unknowntau_wolog}.  Suppose that for some $\epsilon\in(0,1/2]$
$$
	\mathcal R_{n, \cpset_T} \ge  4\sqrt 6 (1+\epsilon) \left(\frac{\omega_n}{ T}\right)^{1/2}+\frac {\sqrt 3 C_\epsilon}{\sqrt T} \left(\log \frac{4n\log_2(T)}\alpha+ \log \frac{2n}\beta\right).
$$
	Then, the type II error of the test $\psi_{n,T}$ is bounded by $\beta$.
\end{lemma}

\begin{proof}
	For ease of notation we denote 
	$$
	\left (\bTheta,\Delta\bTheta^{\tau}\right )=\left \{ (\Theta+\Delta\Theta^\tau,\Pi ),(\Theta,\Pi ) \right \}.
	$$
	By definition of $\beta(\psi_{n,T},\mathcal R_{n,\cpset_T})$ we have
	\begin{align*}
	\beta(\psi_{n,T},\mathcal R_{n, \cpset_T})
	&=\sup_{\tau\in\cpset_T}\sup_{\left (\bTheta,\Delta\bTheta^{\tau}\right )\in\mathcal W_{n,T}^\tau(\omega_n, \mathcal R_{n, \cpset_T}) } \P_{\left (\bTheta,\Delta\bTheta^{\tau}\right )}
	\left\{\max_{t\in\mathcal T} \|Z_T^Y(t)\|_{2\to 2} \le H_{\alpha,n,T}\right \}\\
	&\le   \inf_{t\in\mathcal T} \sup_{\tau\in\cpset_T}\sup_{\left (\bTheta,\Delta\bTheta^{\tau}\right )\in\mathcal W_{n,T}^\tau(\omega_n, \mathcal R_{n, \cpset_T}) }\P_{\left (\bTheta,\Delta\bTheta^{\tau}\right )}\Bigl \{\|Z_{T}^Y(t) \|_{2\rightarrow 2}\leq  H_{\alpha,n,T}  \Bigr \}\\
	& \le \inf_{t\in\mathcal T}  \sup_{\tau\in\cpset_T}\sup_{\left (\bTheta,\Delta\bTheta^{\tau}\right )\in\mathcal W_{n,T}^\tau(\omega_n, \mathcal R_{n, \cpset_T}) }\P_{\left (\bTheta,\Delta\bTheta^{\tau}\right )}\left \{\| \xi(t)\|_{2\rightarrow 2}> \mu_T^\tau(t) \Vert  \Pi\odot\Delta \Theta ^\tau\Vert_{2\rightarrow 2}
	- H_{\alpha,n,T}\right \}.
	\end{align*}
	If $\tau\le T/2$,  there exists a $t^*\in \mathcal T^L$ such that $\tau/2\le  t^*< \tau$. It is easy to see that 
	$$
	\mu_T^\tau(t^*)=\sqrt{\frac{t^*(T-\tau)}{(T-t^*)\tau}}\sqrt{\frac{\tau(T-\tau)}{T}}\ge \sqrt{\frac{\tau/2(T-\tau)}{(T-\tau/2)\tau}} \sqrt{\frac{\tau(T-\tau)}T}\ge \frac 1{\sqrt 3} \sqrt Tq(\tau/T),
	$$
	since $\tau< T/2$ iff $(T-\tau)/(2T-\tau)> 1/3$. If $\tau\ge T/2$, noting that $\mu_T^\tau(t)=\mu_T^{T-\tau}(T-t)$, we can reduce the estimation of $\mu_T^\tau(t)$ to the previous case: there exists $T-t'\in \mathcal T^R$ such that $(T-\tau)/2<T-t'<T-\tau$ and 
	$\mu_T^\tau(t')=\mu_T^{T-\tau}(T-t')\ge \frac1{\sqrt 3}\sqrt Tq(\tau/T)$.
	
	Thus, for any $\tau\in\cpset_T$ we have the following bound of the type II error
	\begin{align*}
	\sup_{\left (\bTheta,\Delta\bTheta^{\tau}\right )\in\mathcal W_{n,T}^\tau(\omega_n, \mathcal R_{n, \cpset_T}) }&\P_{\left (\bTheta,\Delta\bTheta^{\tau}\right )}\Bigl\{\| \xi(t^*)\|_{2\rightarrow 2}>\mu_T^\tau(t) \Vert  \Pi\odot\Delta \Theta ^\tau\Vert_{2\rightarrow 2}
	- H_{\alpha,n,T}\Bigr\}\\
	&\hskip -2cm\le \sup_{\left (\bTheta,\Delta\bTheta^{\tau}\right )\in\mathcal W_{n,T}^\tau(\omega_n, \mathcal R_{n, \cpset_T}) }\P_{\left (\bTheta,\Delta\bTheta^{\tau}\right )}\Bigl\{\| \xi(t^*)\|_{2\rightarrow 2}>\frac 1{\sqrt 3} \sqrt Tq(\tau/T) \Vert  \Pi\odot\Delta \Theta ^\tau\Vert_{2\rightarrow 2}
	- H_{\alpha,n,T}\Bigr\}\\
	&\hskip -2cm\le \sup_{\left (\bTheta,\Delta\bTheta^{\tau}\right )\in\mathcal W_{n,T}^\tau(\omega_n, \mathcal R_{n, \cpset_T}) }\P_{\left (\bTheta,\Delta\bTheta^{\tau}\right )}\Bigl\{\| \xi(t^*)\|_{2\rightarrow 2}>\frac 1{\sqrt 3} \sqrt T\mathcal R_{n,\cpset_T}
	- H_{\alpha,n,T}\Bigr\}\\	
	&\hskip -2cm\le \sup_{\left (\bTheta,\Delta\bTheta^{\tau}\right )\in\mathcal W_{n,T}^\tau(\omega_n, \mathcal R_{n, \cpset_T}) }\P_{\left (\bTheta,\Delta\bTheta^{\tau}\right )}\Bigl\{\| \xi(t^*)\|_{2\rightarrow 2}> 2\sqrt 2(1+\epsilon)\omega_n^{1/2}+C_\epsilon\log\frac{2n}\beta\Bigr\} \le \beta.
	\end{align*}
	The first inequality follows from the fact that $\mu_T^\tau(t)\ge \sqrt(T/3)q(\tau/T)$, the second inequality is a general fact, the third one follows from the definition of  $\mathcal R_{n,\cpset_T}$ and in  the last one we use  Lemma \ref{lem:matrixBernstein} together with  $\Vert \Pi\odot\Theta\Vert_{1,\infty}\leq \omega_n$ and  $\Vert \Pi\odot\left (\Theta+\Delta\Theta\right )\Vert_{1,\infty}\leq \omega_n$.
\end{proof}

\section{Lower bound results}

\subsection{General idea of the lower bound construction}

Let $Y=(Y_1,\dots,Y_T)$ be the observations of the dynamic network following the inhomogeneous random graph model  defined on a probability space $(\Omega,\mathcal A,\P)$. Using \eqref{eq:lbTV}, we can see that a lower bound on the type II error can be obtained by bounding from above the total variation distance between the measures of $Y$ under the null and the alternantive hypotheses. The total variation distance is usually hard to bound and we can use instead the chi-squared or the Kullback--Leibler divergences  as $\|\P_1-\P_0\|_{\mathrm{TV}}\le \sqrt{\chi^2(\P_1,\P_0)}$ and $\|\P_1-\P_0\|_{\mathrm{TV}}\le \sqrt{2\mathrm{KL}(\P_1,\P_0)}$. Thus, the problem of bounding the TV-distance is reduced to the problem of bounding one of these two divergences. 
This is can be done using  {\it the second moment method} or {\it the fuzzy hypotheses method} as follows.

Let $\pi_{n,0}$ and $\pi_{n,1}$ be some prior distributions on the set of parameters $(\Theta,\Pi)$ and $(\Theta+\Delta\Theta^\tau,\Pi)$ of the network under $\Hyp_0$ and $\Hyp_1$, respectively.  Define the mixture distributions 
$\pp_{n,0}^T(Y)=\E_{\pi_{n,0}^T} \P(Y)$ and  $\pp_{n,1}^T(Y)=\E_{\pi_{n,1}^T} \P(Y)$, where $\P$ is the probability measure of $Y$. The expectations w.r.t. to the measures $\pp_{n,i}^T$ are denoted by $\E_{n,i}^T$, $i=0,1$.
The following bounds hold true (see, for example, \citep{Ingster&Suslina:2003}):
\begin{align}
\inf_{\psi^\tau_{n,T}\in \Psi_\alpha}\beta (\psi^\tau_{n,T},\mathcal R_{n,\tau}) 
&= \inf_{\psi_{n,T}^\tau\in \Psi_\alpha} \sup_{(\Theta^\tau,\Theta^{\tau+1})\in \mathcal W_{n,T}^\tau(\omega_n,\mathcal R_{n,\tau})} \P_{(\Theta^\tau,\Theta^{\tau+1})} \Bigl\{\psi_{n,T}^\tau=0\Bigr\}\nonumber\\
&\ge 1-\frac12 \| \pp_{n,1}^T-\pp_{n,0}^T\|_{\mathrm{TV}}-\alpha\label{eq:lb_general}\\
&\ge  1-\frac 12 \sqrt{\chi^2(\pp_{n,1}^T,\pp_{n,0}^T)}-\alpha\nonumber\\
&= 1-\frac12\left(\E_{n,0}^T\left[ \frac{d\pp_{n,1}^T}{d\pp_{n,0}^T}(Y)\right]^2-1\right)^{1/2}-\alpha.\nonumber
\end{align}

Let $\alpha\in(0,1)$ and $\beta\in (0,1-\alpha]$. Set $\eta=\alpha+\beta$. To establish a non-asymptotic  lower  bound $\inf\limits_{\psi_{n,T}^\tau\in \Psi_\alpha}\beta (\psi_{n,T},\mathcal R_{n,\tau}) \ge\beta$ and the corresponding $(\alpha,\beta)$-minimax detection rate, we need to find the conditions on $\mathcal R_{n,\tau}$ such that
$$
\E_{n,0}^T\left[ \frac{d\pp_{n,1}^T}{d\pp_{n,0}^T}(Y)\right]^2 \le 1+4(1-\alpha-\beta)^2=1+4(1-\eta)^2.
$$
 In the case of problem (P2) of unknown change-point $\tau\in\cpset_T$, bounding the type II error     can be reduced to the case of a given  change-point location provided in~\eqref{eq:lb_general}:
\begin{align}
	\inf_{\psi_{n,T}\in \Psi_\alpha}&\beta (\psi_{n,T},\mathcal R_{n,\cpset_T}) \nonumber\\
	&= \inf_{\psi_{n,T}\in \Psi_\alpha} \sup_{\tau\in\cpset_T}\sup_{(\Theta^0+\Delta\Theta^\tau,\Pi),(\Theta^0,\Pi)\in \mathcal W_{n,T}^\tau(\omega_n,\mathcal R_{n,\cpset_T})} \P_{(\Theta^0+\Delta\Theta^\tau,\Pi),(\Theta^0,\Pi)} \Bigl\{\psi_{n,T}=0\Bigr\}\nonumber\\
	&\ge  \inf_{\psi_{n,T}\in \Psi_\alpha} \sup_{(\Theta^0+\Delta\Theta^{\tau^*},\Pi),(\Theta^0,\Pi)\in \mathcal W_{n,T}^{\tau^*}(\omega_n,\mathcal R_{n,\cpset_T})} \P_{(\Theta^0+\Delta\Theta^{\tau^*},\Theta^0)} \Bigl\{\psi_{n,T}=0\Bigr\},\label{eq:lbineq_tau_unknown}
\end{align}
where $\tau^*\in\cpset_T$ is any possible change-point from the set of alternatives. Thus, we can reduce the construction of the lower bound for the case of an unknown change-point to the case of a given change-point $\tau^*$. % Moreover, we will see that the minimax detection rate and the constant are independent of the change-point location.

\subsection{Auxiliary lemma}\label{sec:aux_lb}
Let  $\rho_n\in(0,1/2]$ and $q\in[-1,1]$. Denote by $\pp_0$ and $\pp_q$ the Bernoulli measures with the parameters $\rho_n$ and $\rho_n(1+q)$ with the corresponding densities $d\pp_0$ and $d\pp_q$ with respect to some dominating measure $\lambda$.   
The following simple formulas will be useful in the proof of the lower bound.
\begin{lemma}\label{lemma:means}
	Let $\rho_n\in(0,1/2]$, $q,q_1,q_2\in[-1,1]$. The following relations hold true for a Bernoulli variable $X\sim \pp_0$:
	\begin{gather*}
	\E_0\left[\frac{d\pp_{q}}{d\pp_0}\right]^2(X)=1+\frac{\rho_n}{1-\rho_n} q^2,\\
	\E_0\left[\frac{d\pp_{q_1}}{d\pp_0} \frac{d\pp_{q_2}}{d\pp_0}\right](X)=	1+\frac{\rho_n}{1-\rho_n} q_1 q_2.
	\end{gather*}
\end{lemma}

\subsection{Lower bound for Inhomogeneous Random Graph Model}
We will establish the lower bound for the case of the known change-point location $\tau$. Let $\Pi_n$ be the sampling matrix with the unit entries on the diagonal, $\diag(\Pi_n)=
\one_n$ and with non-zero entries, $\min_{ij} \Pi_{ij}>0$.  In case of $\min_{ij} \Pi_{ij}=1$ there is no missing links. Recall that $\alpha\in(0,1)$, $\beta\in (0,1-\alpha]$ and $\eta=\alpha+\beta$. 
\begin{proof}[Proof of Theorem~\ref{th:lb_tau_unknown}.]
In what follows we denote  by $\delta_x$ the Dirac measure concentrated at $x$, where $x$ can be a real or a matrix value. Denote by $\mathrm P(Y)=\prod_{t=1}^T \mathrm P(Y^t)$ the measure of the observations $Y=(Y^1,\dots,Y^T)$ from~\eqref{Missing_model}. 
	
	Denote by $\tilde \Theta^t=\Pi_n\odot \Theta^t$ the parameter of the observed adjacency matrix.  We will impose the following priors on the matrix parameters $\tilde \Theta^t$ of the dynamic network $Y=(Y^1,\dots,Y^T)$. 
	\begin{description}
		\item[\it Step 1. Priors on the transition matrices.]\ \\
		 	Set $\tilde \rho_n=(1-\eps_n)\dfrac{\omega_n}{n-1}$ for some $\eps_n\in(0,1)$ that will be chosen later.  Assume that under the null hypothesis $\Hyp_0$ all the observed connections occur independently with the same probability $\tilde \rho_n$  for all $1\le t\le T$. Set $ V_0= \tilde \rho_n (\one_n \one_n^\T-\mathrm{id}_n)$ and define the prior under $\Hyp_0$ on the sequence of the sampled connection probability matrices  $\tilde\Theta^t$ ($1\le t\le T$):
$$
		\pi_{n,0}^T(\tilde\Theta^1,\dots,\tilde \Theta^T)=\prod_{t=1}^T\delta_{V_0}(\tilde\Theta^t) = \prod_{t=1}^T\prod_{i\neq j}\delta_{\tilde \rho_n}(\tilde\Theta_{ij}^t).
$$
		Here $\delta_{V_0}$ stands for the Dirac measure concentrated at $V_0$ and defined on the set of matrices $\Pi_n\odot \Theta$ such that $(\Theta,\Pi_n)\in\mathcal S_n(\omega_n)$. The prior is indeed concentrated on  $S_n(\omega_n)$,  since  $\|V_0\|_{1,\infty} =(n-1)\tilde \rho_n =(1-\eps)\omega_n<\omega_n$. 
	
	Let us define the prior under $\Hyp_1$. Let $\zeta=(\zeta_1,\dots,\zeta_n)$ be a vector of i.i.d.\ Rademacher random variables taking values in $\{-1,1\}$ with probability 1/2. 
	Assume that	the sampled connection probability matrices before and after the change are defined by
$$
	V_{1,\zeta}=  V_0 - \Bigl(1-\frac \tau T\Bigr)\Lambda_{n,\zeta} ,\quad  V_{2,\zeta}= V_0+  \frac \tau T\, \Lambda_{n,\zeta} ,
$$
	where  $\Lambda_{n,\zeta}=\dfrac{r_{n,\tau}}{n-1}(\zeta\zeta^\T-\id_n)$ is the change matrix with
	$r_{n,\tau}=\frac{\mathcal R_{n,\tau}}{q(\tau/T)}$.
	Note that the operator norm of $\Lambda_{n,\zeta}$ is equal to $r_{n,\tau}$ and  the energy of the change-point is $q(\tau/T)\|\Lambda_{n,\zeta}\|_{2\to2}=\mathcal R_{n,\tau}$. 
		
	To define the prior concentrated on $\mathcal W(\omega_n,\mathcal R_{n,\tau})$, we need to show that $\|V_{i,\zeta}\|_{1,\infty}\le \omega_{n}$, $i=1,2$ for sufficiently large  $n$. We have that for all $n\ge 2$ and for $i=1,2$,
			\begin{align*}
			\| V_{i,\zeta}\|_{1,\infty} &\le \|V_0\|_{1,\infty} +\frac{\tau\vee (T-\tau)}T \|\Lambda_{n,\zeta}\|_{1,\infty}\\
			&= (1-\eps_n)\omega_n + r_{n,\tau}  \frac{\tau\vee (T-\tau)}T\\
			&= (1-\eps_n )\omega_n +  q\Bigl(\frac\tau T\Bigr) r_{n,\tau}  \sqrt{T-1}\\
			&\le  (1-\eps_n )\omega_n +  \mathcal R_{n,\tau} \sqrt{T}.
			%&\le (1-\eps_n )\omega_n  + \Bigl(\frac{C_\eta}2\Bigr)^{1/4}  \sqrt{\omega_n} = \omega_n\Bigl(1-\eps_n+\frac{\Bigl(\frac{C_\eta}2\Bigr)^{1/4} }{\sqrt \omega}\Bigr),
			\end{align*}
	Let $\eps_n=(2C_\eta)^{1/4} \omega_n^{-1/2}$, then for all $\omega_n > \sqrt{2C_\eta}$ we have $\eps_n\in(0,1)$.
	It will be shown later in~\eqref{eq:cond_rate_exact} that  $\mathcal R_{n,\tau}\sqrt T \le \eps_n  \omega_n$ and, consequently, $\| V_{i,\zeta}\|_{1,\infty} \le  \omega_n$.
		Thus, the prior under $\Hyp_1$  is well defined and is given by 
$$
		\pi_{n,1}^\tau(\tilde\Theta^1,\dots,\tilde\Theta^T)
		=\prod_{t=1}^\tau \delta_{ V_{1,\zeta}}(\tilde\Theta^t) \prod_{t=\tau+1}^T \delta_{ V_{2,\zeta}}(\tilde\Theta^t).
$$
	\item[\it Step 2.  Likelihood ratio of mixtures.]\ \\
	To shorten the notation,  denote 
	$$
	q_{1,\tau}:=-\Bigl(1-\frac \tau T\Bigr)\frac{r_{n,\tau}}{\tilde \rho_n(n-1)},\quad q_{2,\tau}:=\frac\tau T \frac{r_{n,\tau}}{\tilde \rho_n(n-1)}.
	$$
	We can now calculate the mixtures under $\Hyp_0$ that are given by 
	$$
	\pp_{n,0}(Y)=\E_{\pi_{n,0}^T} \mathrm P(Y)=\prod_{i>j} \prod_{t=1}^T \tilde \rho_n^{Y_{ij}^t}(1-\tilde \rho_n)^{1-Y_{ij}^t}
	$$
	and under $\Hyp_1$ that are given by
	\begin{align*}
	\pp_{n,1}^\tau(Y)=\E_{\pi_{n,1}^T} \mathrm P(Y) 
	&= \E_\zeta \left[\prod_{i>j} \left( \prod_{t=1}^\tau (\tilde\rho_n+\tilde\rho_n q_{1,\tau} \zeta_i\zeta_j)^{Y_{ij}^t}(1-\tilde\rho_n-\tilde\rho_n q_{1,\tau} \zeta_i\zeta_j)^{1-Y_{ij}^t}\right.\right.\\
	&\left.\left.\times \prod_{t=\tau+1}^T (\tilde\rho_n+\tilde\rho_n q_{2,\tau} \zeta_i\zeta_j)^{Y_{ij}^t}(1-\tilde\rho_n-\tilde\rho_n q_{2,\tau} \zeta_i\zeta_j)^{1-Y_{ij}^t}\right)\right],
	\end{align*}
	where $\E_\zeta$ stands for the expectation w.r.t. to the distribution of $\zeta$. 
	
	Denote by $\mathcal Z=\{-1,+1\}^n$ the set of all sequences $\zeta=(\zeta_1,\dots,\zeta_n)$ taking values in $\{-1,+1\}$. We can write  the likelihood ratio of mixtures,
	\begin{align*}
	\frac{d\pp_{n,1}^\tau}{d\pp_{n,0}}(Y)&= 
	\frac1{2^n} \sum_{\zeta\in\mathcal Z} \prod_{i> j} \left(\prod_{t=1}^\tau \frac{d\pp_{q_{1,\tau} \zeta_i\zeta_j}}{d\pp_0} (Y_{ij}^t)\prod_{t=\tau+1}^T \frac{d\pp_{q_{2,\tau} \zeta_i\zeta_j}}{d\pp_0}  (Y_{ij}^t)\right),
	\end{align*}
	where $\pp_0$ stands for the Bernoulli measure with the parameter $\tilde \rho_n$ and  $\pp_q$ denotes the Bernoulli measure with the parameter $\tilde \rho_n(1+q)$, as in Lemma~\ref{lemma:means}, Section~\ref{sec:aux_lb}. 	
	\item[\it Step 3. Second moment of the likelihood ratio.]\ \\
	Let $\tilde \zeta$ be  an independent copy of the vector $\zeta$. Using Lemma~\ref{lemma:means}, we can calculate the second moment of the likelihood ratio,
	\begin{align*}
	\E_0\Bigl[\frac{d\pp_{n,1}^\tau}{d\pp_{n,0}}\Bigr]^2(Y)
	&= 
	\frac1{2^{2n}} \sum_{\zeta,\tilde \zeta\in\mathcal Z} \E_0\left[\prod_{i> j} \prod_{t=1}^\tau \frac{d\pp_{q_{1,\tau} \zeta_i\zeta_j}}{d\pp_0} \frac{d\pp_{q_{1,\tau} \tilde \zeta_i\tilde \zeta_j}}{d\pp_0} (Y_{ij}^t)\prod_{t=\tau+1}^T \frac{d\pp_{q_{2,\tau} \zeta_i\zeta_j}}{d\pp_0}   \frac{d\pp_{q_{2,\tau} \tilde \zeta_i\tilde \zeta_j}}{d\pp_0}(Y_{ij}^t)\right]\\
	&= \frac1{2^{2n}} \sum_{\zeta,\tilde \zeta\in\mathcal Z} \prod_{i> j} \Bigl(1+\frac{\tilde \rho_n}{1-\tilde \rho_n}q_{1,\tau}^2\zeta_i\zeta_j\tilde\zeta_i\tilde\zeta_j\Bigr)^\tau
	\Bigl(1+\frac{\tilde \rho_n}{1-\tilde \rho_n}q_{2,\tau}^2\zeta_i\zeta_j\tilde\zeta_i\tilde\zeta_j\Bigr)^{T-\tau}\\
	&=\frac1{2^{2n}} \sum_{\zeta,\tilde \zeta\in\mathcal Z}  \exp\left(\sum_{i> j} \tau \log \Bigl(1+\frac{\tilde \rho_n}{1-\tilde \rho_n}q_{1,\tau}^2\zeta_i\zeta_j\tilde\zeta_i\tilde\zeta_j\Bigr)\right.\\
	&\left. \quad\quad\quad\quad +(T-\tau)\log \Bigl(1+\frac{\tilde \rho_n}{1-\tilde \rho_n}q_{2,\tau}^2\zeta_i\zeta_j\tilde\zeta_i\tilde\zeta_j\Bigr)\right).
	\end{align*}
	Note that 
	$$
	\frac{\tilde \rho_n}{1-\tilde \rho_n} \left(\tau q_{1,\tau}^2+(T-\tau) q_{2,\tau}^2\right)= \frac{Tq^2(\tau/T)r_{n,\tau}^2}{(1-\tilde \rho_n)\tilde \rho_n(n-1)^2}= \frac{T\mathcal R_{n,\tau}^2}{\omega_n} \frac{1}{(1-\tilde \rho_n)(n-1)(1-\eps_n)}
	$$
	and denote the last quantity by $\mu_n=\frac{T\mathcal R_{n,\tau}^2}{\omega_n} \frac{1}{(1-\tilde \rho_n)(n-1)(1-\eps_n)}$.
	Applying the inequality $\log (1+x)\le x$ and using the fact that 	
    the distribution of $\sum_{i\neq j}\zeta_i\zeta_j\tilde\zeta_i\tilde\zeta_j$ is the same as the one of $ \sum_{i\neq j}\zeta_i\zeta_j$, we obtain the upper bound
	\begin{align}\label{eq:upper_bound}
	\E_0\Bigl [\frac{d\pp_{n,1}^\tau}{d\pp_{n,0}}\Bigr]^2(Y)
	&\le \frac1{2^{2n}} \sum_{\zeta,\tilde \zeta\in\mathcal Z}  \exp\Bigl(\frac12\mu_n \sum_{i\neq j}\zeta_i\zeta_j\tilde\zeta_i\tilde\zeta_j\Bigr)\nonumber \\
	&=\E_\zeta  \exp\Bigl( \frac12\mu_n \sum_{i\neq  j} \zeta_i\zeta_j\Bigr) 
	= \E_\zeta  \exp\Bigl( \frac12\mu_n \zeta^\T (\one_n^\T\one_n -\id_n )\zeta\Bigr).
	\end{align}
	\item[\it Step 4. Upper bound on the second moment.] \ \\ 
	Using  Theorem~2 in~\citep{Cortinovis2021}, we can bound the Laplace transform of the Rademacher chaos  $\sum_{i\neq j}\zeta_i\zeta_j$ as follows. Let $A\in \bR^{n\times n}$ be a symmetric matrix with zero diagonal. Then, $\forall \ 0<\mu<1/4$,
$$
	\log \E \left (e^{\mu \zeta^\T A\zeta}\right )\le \mu\|A\|_F^2\log \frac{1-2\mu}{1-4\mu} \le \frac{2\mu^2\|A\|_F^2}{1-4\mu}.
$$
	
	%	Using the result of Proposition~8.13 in~\citep{Foucart2013}, we can bound the Laplace transform of the Rademacher chaos  $\sum_{i\neq j}\zeta_i\zeta_j$ as follows. Let $A\in \bR^{n\times n}$ be a symmetric matrix with zero diagonal. Then $\forall \ 0<\lambda<(8\|A\|_{2\to 2})^{-1}$,
	%	\begin{equation}\label{eq:Laplace_tr_chaos}
	%	\E e^{\lambda \zeta^\T A\zeta}\le \exp\left(\frac{8\lambda^2\|A\|_F^2}{1-64\lambda^2\|A\|_{2\to2}^2}\right).
	%	\end{equation}
	Using this inequality,  from \eqref{eq:upper_bound} we obtain that for any $\mu_n<1/2$ 
	$$
	\E_0\Bigl[\frac{d\pp_{n,1}^\tau}{d\pp_{n,0}}\Bigr]^2(Y) \le \E_\zeta  \exp\Bigl\{\frac12 \mu_n\zeta^\T (\one_n\one_n^\T -\id_n ) \zeta\Bigr\}\le \exp\Bigl(\frac{\frac12\mu_n^2n(n-1)}{1-2\mu_n}\Bigr).
	$$
 This bound implies that the second moment of the likelihood ratio is less than  $1+4(1-\eta)^2$ if 
	\begin{equation}\label{eq:mu_n2bound}
	\mu_n \le \frac{2C_\eta}{n^2} \left(\sqrt{1+\frac{n^2}{2C_\eta}} -1\right),
	\end{equation}
	where $C_\eta=\log(1+4(1-\eta)^2)$. 	Note that this condition will imply $\mu_n<1/2$. Now, \eqref{eq:mu_n2bound} is satisfied if 
	$$
	\frac{T\mathcal R_{n,\tau}^2}{\omega_n}\le (1-\tilde \rho_n)(1-\eps_n)\Bigl(1-\frac 1n\Bigr)  \sqrt{2C_\eta}  \left(\sqrt{1+\frac{2C_\eta}{n^2}} -\sqrt{\frac{2C_\eta}{n^2}}\right).
	$$
	Noting that $\sqrt{x+1}-\sqrt x\ge (1+2\sqrt x)^{-1}$, we get that this inequality is satisfied if 
	\begin{equation}\label{eq:cond_rate_exact}
	\frac{T\mathcal R_{n,\tau}^2}{\omega_n}\le \sqrt{2C_\eta} (1-\tilde \rho_n)(1-\eps_n)\Bigl(1-\frac 1n\Bigr)  \Biggl(1+ \frac{2\sqrt {2 C_\eta}}n\Biggr)^{-1}.
	\end{equation}
	It means that all the signals with energy $\mathcal R_{n,\tau}$ satisfying~\eqref{eq:cond_rate_exact} are not detectable by any $\alpha$-level test with the type II error smaller than $\beta$. Therefore the lower bound on the minimal detectable energy for an $\alpha$-level test with  type II errors bounded $\beta$ is given by 
	$$
	\mathcal R^*_{n,\tau} \ge \Bigl(2\log\bigl(1+4(1-\eta)^2\bigr) \Bigr)^{1/4}\sqrt{\frac{\omega_n}T} 
	$$
	and the theorem follows.
\end{description}
\end{proof}

\section{Proof of result on the change-point localization}
\begin{proof}[Proof of Proposition~\ref{thm_estimation}.]
 Lemma \ref{lem:matrixBernstein} implies that for any $s\in [T]$, with probability at least $1-\tfrac{\gamma }{T}$
\begin{equation}\label{eq:noise_bound}
\|\xi(s)\|_{2\rightarrow 2}\leq 2\sqrt{2}(1+\epsilon)\sqrt{\omega_{n} }+C_{\epsilon}\log\left (2nT/\gamma\right ).
\end{equation}
By the definition of $\widehat{\tau}_n$ we have $\|Z_T(\widehat{\tau}_n)\|_{2\to 2}\geq \|Z_T(\tau)\|_{2\to 2}$ which implies that
\[\mu_T^\tau(\tau)\Delta-\Vert \xi(\tau)\Vert_{2\rightarrow 2}\leq \mu_T^\tau(\widehat{\tau}_n)\Delta+\Vert \xi(\widehat{\tau}_n)\Vert_{2\rightarrow 2}.\]
Using \eqref{eq:noise_bound} and the union bound we get that with probability at least $1-\gamma$ 
\begin{equation}\label{eq:proof_localisation1}
\left (\mu_T^\tau(\tau)- \mu_T^\tau(\widehat{\tau}_n)\right )\Delta\leq 4\sqrt{2}(1+\epsilon)\sqrt{\omega_{n} }+C_{\epsilon}\log\left (2nT/\gamma\right ).
\end{equation}
First, consider the case $\widehat{\tau}_n\leq \tau$. Using the definition of $\mu_T^\tau(t)$ \eqref{def_mu}, we compute
\begin{align*}
\mu_T^\tau(\tau)- \mu_T^\tau(\widehat{\tau}_n)&=\sqrt{T}\left (q(x^*)-q(\widehat{x})\dfrac{1-x^*}{1-\widehat{x}}\right )\\
&=\sqrt{T}(1-x^*)\left (\dfrac{q(x^*)}{1-x^*}-\dfrac{q(\widehat{x})}{1-\widehat{x}}\right )\\
&=\sqrt{T}(1-x^*)\left (\sqrt{\dfrac{x^*}{1-x^*}}-\sqrt{\dfrac{\widehat{x}}{1-\widehat{x}}}\right )\\
&=\sqrt{\dfrac{T(1-x^*)}{1-\widehat{x}}}\dfrac{x^*-\widehat{x}}{\sqrt{\widehat{x}(1-x^*)}+\sqrt{x^*(1-\widehat{x})}}
\\
&\geq \sqrt{T(1-x^*)}\dfrac{x^*-\widehat{x}}{1.5}
\end{align*}
where we use that for any $x\in(0,1)$, $x(1-x)\leq 1/4$. Plugging this calculation into \eqref{eq:proof_localisation1} we get
\begin{equation}\label{eq:proof_localisation2}
\left (x^*-\widehat{x}\right )\Delta \leq 6(1+\epsilon)\sqrt{\dfrac{2\omega_{n}}{T(1-x^*)} }+\dfrac{1.5C_{\epsilon}\log\left (2nT/\gamma\right )}{\sqrt{T(1-x^*)}}.
\end{equation}
Now assume that $\widehat{\tau}_n\geq \tau$. Then, using the definition of $\mu_T^\tau(t)$, \eqref{def_mu}, we compute
\begin{align*}
\mu_T^\tau(\tau)- \mu_T^\tau(\widehat{\tau}_n)&=\sqrt{T}\left (q(x^*)-q(\widehat{x})\dfrac{x^*}{\widehat{x}}\right )\\
&=\sqrt{T}x^*\left (\dfrac{q(x^*)}{x^*}-\dfrac{q(\widehat{x})}{\widehat{x}}\right )\\
&=\sqrt{T}x^*\left (\sqrt{\dfrac{1-x^*}{x^*}}-\sqrt{\dfrac{1-\widehat{x}}{\widehat{x}}}\right )\\
&=\sqrt{\dfrac{Tx^*}{\widehat{x}}}\dfrac{\widehat{x}-x^*}{\sqrt{\widehat{x}(1-x^*)}+\sqrt{x^*(1-\widehat{x})}}
\\
&\geq \sqrt{T(1-x^*)}\dfrac{\widehat{x}-x^*}{1.5}
\end{align*}
which implies
\begin{equation}\label{eq:proof_localisation3}
\left (\widehat{x}-x^*\right )\Delta \leq 6(1+\epsilon)\sqrt{\dfrac{2\omega_{n}}{Tx^*} }+\dfrac{1.5C_{\epsilon}\log\left (2nT/\gamma\right )}{\sqrt{Tx^*}}.
\end{equation}
Combining \eqref{eq:proof_localisation2} and \eqref{eq:proof_localisation3} and using $q^2(x^*)\leq x^*\wedge (1-x^*)$ we get the statement of the Proposition~\ref{thm_estimation}.
 \end{proof}
%----------------------------------------------------------
\section{Proofs of results for the sparse graphon model}
We start by summarizing  some   notation that we use in the proofs.

Given a matrix $\Theta\in[ -1,1]^{n\times n}$, we define the empirical graphon associated with $\Theta$ as follows: 
\begin{equation}\label{eq:emp_graphon}
\tilde f_\Theta(x,y)=\Theta_{\lceil nx\rceil, \lceil ny \rceil},\quad (x,y)\in[0,1]^2.
\end{equation}
In the same spirit, given a vector $v=(v_1,\dots,v_n)$, for any $x\in[0,1]$, we define the following piecewise constant function
$$
\psi_v(x)=\sqrt{n} v_{\lceil nx\rceil},\quad x\in[0,1]
$$
and set
$$
\mathcal F=\Bigl\{ \psi_v:\ \|v\|_{\ell_2}\le 1\Bigr\}.
$$
We have that $\|v\|_{\ell_2}\le 1$ implies 
$$
\|\psi_v\|_{L_2[0,1]}=\frac 1n\sum_{i=1}^n nv_i^2 \le 1.
$$
We will need to work with a difference of two graphons, so we extend the definition of graphon space. In what follows $\mathcal W$ refers to the collection of bounded symmetric  measurable functions $W:[0,1]^{2}\rightarrow [-1,1]$.

\subsection{Proofs of upper bounds}
\begin{proof}[Proof of Theorem~\ref{th:Kstepgraphon_upperbound}.]
	%	The proof is similar to the proof of Theorem~\ref{th:upper_bound:P1}. 	
 We have to find a threshold $H_{\alpha,n,T}^*$ such that 
	$$
	\alpha(\psi_{n,T})= \sup_{W\in \mathcal W_0} \P_{W}\Bigl\{\psi_{n,T}=1\Bigr\}\le \alpha
	$$
	and show that 
	$$
	\beta(\psi_{n,T},\delta_{n,T})=\sup_{\tau\in\cpset_T}\left ( \sup_{W^\tau,W^{\tau+1}\in \mathcal W_0(\delta_{n,T})} \P_{W^\tau,W^{\tau+1}}\Bigl\{\psi_{n,T}=0\Bigr\}\right )\le \beta.
	$$
 Under $\Hyp_0$ there is no change in the graphon function $W$ but it can be a change in the features. Let $\Theta_{1}=\left (\rho_n W(\eps_i,\eps_j)\right )_{(i,j)\in [n]\times [n]}$ denotes the matrix of connection probabilities before the time point $\tau$ and  
	$\Theta_{2}=\left (\rho_n W(\eps'_i,\eps'_j)\right )_{(i,j)\in [n]\times [n]}$ the matrix of connection probabilities after $\tau$. 
	Let
	\begin{equation}\label{def_xi_pi}
		\xi^{\pi}(t)= \sqrt{\frac{t(T-t)}T}\left (\dfrac{1}{t}\sum_{s=1}^{t}W^{s}-\dfrac{1}{T-t}\sum_{s=t+1}^{T}
		W^{s}\circ \pi\right )
	\end{equation}
	denote  centered  random matrices of noise corresponding to the permutation $\pi$ .
	We have that
	\begin{equation*}%\label{cosum_model}
		Z_T(t)=-\mu_T^\tau(t) \Delta\Theta^\tau+\xi^{*}(t),\quad t=1,\dots,T-1,
	\end{equation*}
	where 
	\begin{equation*}%\label{def_mu}
		\mu_T^\tau(t)= \sqrt{\frac{t(T-t)}T}\left (\frac{\tau}{t}\one_{\{\tau+1\le t\le T\}}+\frac{T-\tau}{T-t}\one_{\{1\le t\le\tau\}}\right ),\quad  \Delta\Theta^\tau=\Theta_{1}-\Theta_{2}\circ\pi^{*}
	\end{equation*}
	and the random matrices
	\begin{equation*}%\label{def_xi}
		\xi^{*}(t)= \sqrt{\frac{t(T-t)}T}\left (\dfrac{1}{t}\sum_{s=1}^{t}W^{s}-\dfrac{1}{T-t}\sum_{s=t+1}^{T}
		W^{s}\circ \pi^{*}\right )
	\end{equation*}
	are centered. Let $\pi'$ be a permutation of $\{1,\dots,n\}$ such that $$\pi'\in \underset{\pi}{\argmin}\|\Theta_{1}-\Theta_{2}\circ\pi\|_{2\to2}.$$
	Let 
	\begin{equation*}%\label{def_xi}
		\xi(t)= \sqrt{\frac{t(T-t)}T}\left (\dfrac{1}{t}\sum_{s=1}^{t}W^{s}-\dfrac{1}{T-t}\sum_{s=t+1}^{T}
		W^{s}\circ \pi'\right )
	\end{equation*}
be the corresponding noise term. 	
	 Using the definition of $\pi^{*}$ and the triangle inequality, we have that
	\begin{align*}%\label{eq:bound_test_statistics}
		\|Z_T^{\pi^*}(t)\|_{2\to2}&\le \|Z_T^{\pi'}(t)\|_{2\to2}\\
		&\leq \mu_T^\tau(t)\|\Theta_{1}-\Theta_{2}\circ\pi'\|_{2\to2}+\| \xi(t)\|_{2\rightarrow 2}\\&\leq \sqrt T q( t/T)\|\Theta_{1}-\Theta_{2}\circ\pi'\|_{2\to2}+\| \xi(t)\|_{2\rightarrow 2}.
	\end{align*}
Now, using Lemma \ref{lem:Kstep_graphon}, we have that with probability at least $1-4/n$
\begin{align*}
 \|\Theta_{1}-\Theta_{2}\circ\pi'\|^{2}_{2\to2}&\leq n^{2}\delta^{2}_2(\tilde{f}_{\Theta_{1}},\tilde{f}_{\Theta_{2}})\leq  
 2 n^{2}\left (\delta^{2}_2(\tilde{f}_{\Theta_{1}},\rho_n W )+
 \delta^{2}_2(\tilde{f}_{\Theta_{2}},\rho_n W ) \right )\\
 &\leq 32 n^{2}\rho^{2}_n\sqrt{\frac{K}{n}\log(n)}
\end{align*}
where $\tilde{f}_{\Theta_{i}}$ is the empirical graphon associated with $\Theta_i$.
Note that, for any $x>0$,
\begin{align*}
\P_W(\| \xi^{\pi'}(t)\|_{2\rightarrow 2}>x)=\E_{\{\varepsilon_{\vartheta}, \varepsilon'_{\vartheta} \}_{\vartheta\in \mathcal V}}\left [\P\left (\| \xi(t)\|_{2\rightarrow 2}>x\left \vert \{\varepsilon_{\vartheta}, \varepsilon'_{\vartheta} \}_{\vartheta\in \mathcal V}\right .\right )\right ].
\end{align*}
Now, we can bound $\P\left (\| \xi^{\pi'}(t)\|_{2\rightarrow 2}>x\left \vert \{\varepsilon_{\vartheta}, \varepsilon'_{\vartheta} \}_{\vartheta\in \mathcal V}\right .\right )$ using Lemma \ref{lem:matrixBernstein}:
for every $\delta\in(0,1)$, conditionally on $\{\varepsilon_{\vartheta}, \varepsilon'_{\vartheta} \}_{\vartheta\in \mathcal V}$,  we have 
\begin{equation*}%\label{MatrixBernsteinStoch}
\| \xi^{\pi'}(t)\|_{2\rightarrow 2}\leq 2\sqrt{2}(1+\epsilon)\sqrt{\frac{t(T-t)}T} \left [\frac1{t^2}\sum_{s=1}^{t}\| \Theta^{s}\|_{1,\infty}+ \frac1{(T-t)^2}\sum_{s=t+1}^{T}\|\Theta^{s}\|_{1,\infty}\right ]^{1/2} + C_\epsilon \log\left (\frac{2n}\delta\right )
\end{equation*}
with the probability larger than $1-\delta$ where $\Theta^{t}=\rho_n\left (W^{t}(\varepsilon_i,\varepsilon_j)\right )_{(i,j)\in[n]\times [n]}$. Note that $\|\Theta^{t}\|_{1,\infty}\leq n\rho_n$ and we get
$$
\| \xi^{\pi'}(t)\|_{2\rightarrow 2}\leq 2(1+\epsilon)\sqrt{2n\rho_n} + C_\epsilon \log\left (\frac{2n}\delta\right )
$$
with the probability larger than $1-\delta$. Taking $\delta=\alpha/(2|\mathcal T|)$, using the union bound and $\alpha\geq 8/n$, we obtain
$$
\alpha(\psi_{n,T})= \sup_{W\in \mathcal W_0} \P_{W,\eps} \Bigl\{\max_{t\in\mathcal T}\|Z_T^{\pi^*}(t)\|_{2\to2} \ge H_{\alpha,n,T}^*\Bigr\}\le \alpha.
$$
Let us turn to the type II error.  The proof is close to the proof of Lemma~\ref{lem:typeIIerror_missing_wolog} but, compared to the case of inhomogeneous random graph model considered in Lemma~\ref{lem:typeIIerror_missing_wolog}, we need to account for possible change in the latent variables $\varepsilon$. The key point is to link the change in the operator norm of the matrix of parameters to the operator norm of the corresponding empirical graphon.  Following the proof of Lemma~\ref{lem:typeIIerror_missing_wolog} we can show that, if $\tau\le T/2$, there exists $t^*\in\mathcal T^L$ such that $\tau/2\le t^*<\tau$ and 
\begin{align*}
	&\beta(\psi_{n,T},\delta_{n,T})\\
	& \le \sup_{W^\tau,W^{\tau+1}\in \mathcal W_0(\delta_{n,T})} \P_{W^\tau,W^{\tau+1}} \left\{\|\xi^{*}(t^*)\|_{2\to 2} >\frac {\sqrt{T}q(\tau/T)}{\sqrt 3} \|\Delta\Theta^\tau\|_{2\to 2} -H^*_{\alpha,n,{t^*}}\right\}.
\end{align*}

 Lemma \ref{lem:operator_norm} and the definition of the empirical graphon \eqref{eq:emp_graphon} imply
\begin{align*}
\| \xi^{*}(t^{*})\|_{2\rightarrow 2}&> 	\frac {\sqrt{T}q(\tau/T)}{\sqrt 3} \|\Theta_{1}-\Theta_{2}\circ\pi^{*}\|_{2\to2} -H_{\alpha,n,t^*}^* \quad \text{(triangle inequality)}\\
%&\geq \mu_T^\tau(t)\|\Theta_{1}-\Theta_{2}\circ\pi^{'}\|_{2\to2} -H_{\alpha,n,T}^*\\
&\geq \frac {n\sqrt{T}q(\tau/T)}{\sqrt 3} \Vert\tilde{f}_{\Theta_{1}-\Theta_{2}\circ \pi^{*}}\Vert_{2\to2}-H_{\alpha,n,t^*}^* \quad \text{(Lemma \ref{lem:operator_norm})}\\
&= \frac {n\sqrt{T}q(\tau/T)}{\sqrt 3}\Vert\tilde{f}_{\Theta_{1}}-\tilde{f}_{\Theta_{2}\circ \pi^{*}}\Vert_{2\to2}-H_{\alpha,n,t^*}^*\\
&\geq\frac {n\sqrt{T}q(\tau/T)}{\sqrt 3}\delta(\tilde{f}_{\Theta_{1}},\tilde{f}_{\Theta_{2}})-H_{\alpha,n,t^*}^*\quad \text{(using \eqref{eq:emp_graphon})}
\end{align*}
Now, using twice the triangle inequality, we get
\begin{align}\label{eq:1}
	\| \xi^{*}(t^{*})\|_{2\rightarrow 2}
	\geq \frac {n\sqrt{T}q(\tau/T)}{\sqrt 3}\left (\delta(\rho_n\,W_1,\rho_n\,W_2)-\delta(\tilde{f}_{\Theta_{1}},\rho_n\,W_1)-\delta(\rho_n\,W_2,\tilde{f}_{\Theta_{2}})\right )-H_{\alpha,n,t^*}.
\end{align}
Note that Lemma \ref{lem:Kstep_graphon} and $\delta(\tilde{f}_{\Theta_{i}},\rho_n\,W_i)\leq \delta_2(\tilde{f}_{\Theta_{i}},\rho_n\,W_i)$ imply
\[
\delta(\tilde{f}_{\Theta_{i}},\rho_n\,W_i)\leq 4\rho_n\left (\dfrac{K_i\log n}{n}\right )^{1/4}
\]
with probability at least $1-2/n$.  Take $\beta\geq 6/n$. Using  Lemma \ref{lem_bernstein_sup} with $\delta=\beta/3$ we obtain
$$
	\| \xi^{\pi^*}(t^*)\|_{2\rightarrow 2}\le\sup_{\pi}\|\xi^\pi(t^*)\|_{2\to 2}\leq 2(1+\epsilon)\sqrt{2n\rho_n} + C_\epsilon \log\left (\frac{12n}\beta\right ).
$$
with probability $1-\beta/3$. This implies that $\beta(\psi_{n,T},\delta_{n,T})\le \beta$ if 
$$
\frac {n\sqrt{T}q(\tau/T)}{\sqrt 3}\left (\rho_n\delta(W_1,W_2)-4\rho_n\left (\dfrac{(K_1+K_2)\log n}{n}\right )^{1/4}\right )\ge H_{\alpha,n,t^*}^* + 2(1+\epsilon)\sqrt{2n\rho_n}+C_\epsilon \log\left(\frac{12n}\beta\right)
$$
Combining this condition with the threshold $H_{\alpha,n,t}^*$ defined in~\eqref{threshold_kstep_graphon} and the facts that $q(t^*/T)\le q(\tau/T)$ for $t^*<\tau\le T/2$ and $K_1,K_2\le K$, we obtain the detection condition~\eqref{detection_boundary_graphon_stepf}. The case of $\tau>T/2$ is analogous.

%the condition on  detection boundary given by \eqref{detection_boundary_graphon_stepf} into \eqref{eq:1} we get
%\begin{align*}%\label{eq:1}
%	\| \xi^{*}(t^{*})\|_{2\rightarrow 2}
%	&\geq 8\sqrt 2(1+\epsilon)\sqrt{n\rho_n }+4n^{3/4}\rho_n\sqrt{2T} q\Bigl(\frac\tau T\Bigr)\Bigl((K_1+K_2)\log \left (n\right ) \Bigr)^{1/4}-H_{\alpha,n,T}^*	.
%\end{align*}
%Let $\iota\in\mathcal M$ be such that
%		\begin{equation*}
%		\delta(W^{\tau},W^{\tau+1})=  \Vert W^{\tau}(x,y)-W^{\tau+1}(\iota(x),\iota(y))\Vert_{2\to 2}.
%	\end{equation*}
%Denote by $ \Delta W^\tau= W^{\tau}(x,y)-W^{\tau+1}(\iota(x),\iota(y))$ and let	 $\Delta\Theta_{ij}^\tau=\rho_n \left (W^{\tau}(\varepsilon_i,\varepsilon_j)-W^{\tau+1}(\iota(\varepsilon_i),\iota(\varepsilon_i))\right )$, 	{\color{red} where $\iota(\eps)=\eps'$.
%We need to handle the situation when  the feature vector $\eps$ changes {\bf or} the graphon $W$ changes. 

%Now we can follow the proof of Theorem \ref{th:upper_bound_missing_loglog}, the only difference being the control of the spectral norm of the noise term $\| \xi(t)\|_{2\rightarrow 2}$.  
%and Theorem \ref{th:Kstepgraphon_upperbound} follows.
 \end{proof}
\begin{proof}[Proof of Theorem~\ref{th:Holdergraphon_upperbound}.]
	%	The proof is similar to the proof of Theorem~\ref{th:upper_bound:P1}. 
	Theorem \ref{th:Holdergraphon_upperbound} follows from  combining the bounds obtained in  Lemma~\ref{lem:operator_norm}, Lemma~\ref{lem:Holder_graphon} and Theorem \ref{th:upper_bound_missing_loglog}.
 \end{proof}

\begin{lemma}\label{lem:operator_norm}
	Let $\Theta=(\Theta_{ij})\in[-1,1]^{n\times n}$ be symmetric matrix. Then
	$$
	\|\Theta\|_{2\to 2} \ge n\|\tilde f_{\Theta}\|_{2\to2}. 
	$$
\end{lemma}
\begin{proof}
	By the definition of the operator norm, we have that
	\begin{align*}
		\| \tilde  f_\Theta\|_{2\to2}&= \sup_{\psi\in L_2[0,1],\|\psi\|_{L_2}\le 1}
		\Bigl|  \iint_{[0,1]^2} \tilde f_\Theta (x,y) \psi(x)\psi(y)\,dx\,dy\Bigr|\\&=\sup_{\psi\in L_2[0,1],\|\psi\|_{L_2}\le 1}\Bigl| \sum_{ij}\Theta_{ij}\int_{i/n}^{(i+1)/n}
		\psi(x)\,dx	\int_{j/n}^{(j+1)/n}
		\psi(y)\,dy \Bigr|\\
		\\&=\sup_{\psi\in L_2[0,1],\|\psi\|_{L_2}\le 1}\Bigl|  \dfrac{1}{n}\sum_{ij}\Theta_{ij}v_i^{\psi} v_j^{\psi}	\Bigr|	
	\end{align*}
	where $v_i^{\psi}=\sqrt{n}\int_{i/n}^{(i+1)/n}
	\psi(x)\,dx	$ and $v^{\psi}=(v_i^{\psi})^{n}_{i=1}$.  Note that the Cauchy-Schwartz inequality implies \[\left (\int_{i/n}^{(i+1)/n}
		\psi(x)\,dx\right )^{2}\leq \dfrac{1}{n}\int_{i/n}^{(i+1)/n}
		\psi^{2}(x)\,dx\]
		and we get that $\|v^{\psi}\|_{\ell_2}\le 1$
		for $\psi$ such that $\|\psi\|_{L_2}\le 1$. Now we can write
	\begin{align*}
		n	\| \tilde  f_\Theta\|_{2\to2}\leq \sup_{\|v\|_{\ell_2}\le 1}\Bigl| \sum_{ij}\Theta_{ij}v_i v_j \Bigr|=	\|\Theta\|_{2\to 2}	
	\end{align*}
	and the statement of Lemma \ref{lem:operator_norm} follows.
\end{proof}

\begin{lemma}\label{lem:Kstep_graphon}
		For any $K\le \frac{n}{\log(n)}$ assume that  $W\in \mathcal W_K$. Let $\Theta=(\Theta_{ij})\in[0,1]^{n\times n}$ be symmetric matrix with entries $\Theta_{ij}=W(\xi_i,\xi_j)$ for $i<j$, where $\xi_i$ are i.i.d. uniform random variables on $[0,1]$. We have that, with probability large than $1-2/n$,
		$$
		\delta(\tilde f_\Theta,W) \le 4\Bigl(\frac Kn \log \left (n\right )\Bigr)^{1/4}.
		$$
	\end{lemma}
	\begin{proof}
		Following the proof of Proposition 3.2 in~\citep{klopp_graphon}, we get  
		$$
		\delta_2 (\tilde f_{\Delta\Theta^\tau},\Delta W^\tau)\le \frac 1n +\sum_{a=1}^K |\lambda_a-\widehat \lambda_a|,
		$$	
		with
		$$
		\widehat \lambda_a =\frac 1n \sum_{i=1}^n \one_{\{\varepsilon_i\in\phi^{-1}(a)\}}
		$$
		and  $\lambda_a=\lambda (\phi^{-1}(a))$, where $\lambda$ stands for the Lebesgue measure. Since $\varepsilon_1,\dots\varepsilon_n$ are i.i.d. uniform random variables, $n\widehat \lambda_a $ has a binomial distribution with parameters $(n,\lambda_a)$. We have $n\widehat \lambda_a -n \lambda_a =\sum_{i=1}^n (Y_i- \lambda_a )$, where $Y_i\sim \mathrm{Bernoulli}(\lambda_a)$. Applying the Bernstein inequality we obtain that for any $t>0$
		$$
		|n\widehat \lambda_a -n \lambda_a| \le \left(2t\sum_{a=1}^K \lambda_a(1-\lambda_a)\right)^{1/2}+2t/3
		$$
		with probability $1-2e^{-t}$. Taking $t=\log(nK)$ implies that with probability $1-2/(nK)$
		$$
		|n\widehat \lambda_a -n \lambda_a| \le \left(2n\lambda_a \log(nK)\right)^{1/2}+\frac 23\log (nK).
		$$
		%Since $K$ is constant, assuming $n$ to be large enough we get $|n\widehat \lambda_a -n \lambda_a| \le 2\left(n\lambda_a \log n\right)^{1/2}$.
		
		Using $K\leq n$ and the union bound we obtain that, with probability $1-2/(n)$,
		$$
		\delta^{2}(\tilde f_\theta,W)
		\le \frac 1n +\frac 2n \sum_{a=1}^K \left(n\lambda_a\log n\right)^{1/2}+\frac{4K\log(n)}{3n}\le \frac 1n +2\left(\frac{K\log n}n\right)^{1/2}+\frac{4K\log(n)}{3n}
		$$
		where we use $\sum_{a=1}^K \lambda_a=1$ and the Cauchy--Schwarz inequality. Using $\frac{K\log(n)}{n}\leq 1$ we complete the proof of Lemma \ref{lem:Kstep_graphon}.
	\end{proof}

\begin{lemma}\label{lem:Holder_graphon}
Assume that  $W\in \Sigma(\gamma,L)$. Let $\Theta=(\Theta_{ij})\in[0,1]^{n\times n}$ be symmetric matrix with entries $\Theta_{ij}=W(\varepsilon_i,\varepsilon_j)$ for $i<j$, where $\varepsilon_i$ are i.i.d. uniform random variables on $[0,1]$. We have that, 	with probability at least $1-2/n$,
	$$
	\delta(\tilde f_\Theta,W)\le 2\left(\frac{\log n}n\right)^{\frac{\gamma\wedge 1}2}
	$$

\end{lemma}
\begin{proof}
  	Following the proof of Proposition 3.6 in  \citep{klopp_graphon},  we get 
	$$
	\delta^2(\tilde f_\Theta,W)\le \frac 2n+\frac 1n \sum_{m=1}^n \Bigl| \frac m{n+1}-\varepsilon_{(m)}\Bigr|^{2\gamma'}
	$$	
	where $\gamma'=\gamma\wedge 1$ and $\varepsilon_{(m)}$ stands for the $m$-th largest element of the set $\{\varepsilon_1,\dots,\varepsilon_n\}$. Note that, the random variable $\varepsilon_{(m)}$ follows $\beta$-distribution with parameters $(m,n+1-m)$, $\varepsilon_{(m)}\sim \mathrm{Beta}(m,n+1-m)$. The $\beta$-distribution is sub-Gaussian and the proxy variance $\sigma^2$ for $\mathrm{Beta}(m,n+1-m)$ is bounded by $\frac{1}{4(n+2)}$ (see, for example, \citep{marchal2017}). By the exponential Markov inequality (see, for example, \citep{vershynin_2018}, Lemma~5.5) we get
	$$
	\P\left\{\Bigl|\varepsilon_{(m)} -\frac m{n+1}\Bigr|>t\right\}\le 2e^{-t^2/(4\sigma^2)}.
	$$
	Taking $t=(\log n/(n+2))^{1/2}$ implies that, 	with probability at least $1-2/n^2$,
	$$
	\Bigl|\varepsilon_{(m)} -\frac m{n+1}\Bigr| < \left(\frac{\log n}{n+2}\right)^{1/2}.
	$$
Now, applying the union bound we obtain
	$$
	\delta^2(\tilde f_\Theta,W)\le \frac 2n +\left(\frac{\log n}{n+2}\right)^{\gamma'}
	$$
	and Lemma \ref{lem:Holder_graphon} follows. 
 \end{proof}

\subsection{Proof of the lower bound for $K$-step graphons}

We will start with the definition  of a class of $K$-step graphons used throughout the proof. 
Let $u=(u_1,\dots,u_K)\in (-\frac 1K,\frac1K)^K$ be a given vector satisfying $\sum\limits_{k=1}^K u_k=0$. Define  the partition $\Pi=\bigcup\limits_{1\le k,l\le K} \Pi_{kl}$ of the set $[0,1]^2$ into $K^2$ blocks:
$$
\Pi_{kl}(u)=\Bigl [\frac {k-1}K +\sum_{i=1}^{k-1}u_i,\frac {k}K+\sum_{i=1}^{k}u_i\Bigr )\times\Bigl [\frac {l-1}K +\sum_{i=1}^{l-1}u_i,\frac lK+\sum_{i=1}^{l}u_i \Bigr )\quad \forall 1\le k,l\le K.
$$
Let $Q=(Q_{kl})_{1\le k,l\le K}\in [0,1]^{K\times K}$ be a  matrix of connection probabilities. The $K$-step graphon $W_u$ is a blockwise constant function defined by 
\begin{equation}\label{eq:W_u}
W_u(x,y)=\sum_{k,l\in[K]^2} Q_{kl} \mathbf 1 \{(x,y) \in \Pi_{kl}(u)\}.
\end{equation}

Let $\boldsymbol\eps=(\eps_1,\dots,\eps_n)\in[0,1]^n$ be the vector of i.i.d. features uniformly distributed over $[0,1]$. Note that the community assignment of each vertex $\vartheta$ is defined by the corresponding variable $\eps_{\vartheta}$:
$$
\P\Bigl\{\mbox{$\vartheta$ belongs to the block $k$}\Bigr\}= \P\Bigl\{\eps_\vartheta\in\Bigl [\frac {k-1}K +\sum_{i=1}^{k-1}u_i,\frac {k}K+\sum_{i=1}^{k}u_i\Bigr )\Bigr\}=\frac 1K +u_k. 
$$
We can introduce a new random variable $\xi_\vartheta=\xi_\vartheta(u)$, $\vartheta\in[n]$ of the block assignment following multinomial distribution $\mathcal M(K,p_1,\dots,p_K)$ with parameters $p_k=1/K+u_k$. Denote the corresponding vector  of i.i.d. multinomial variables by $\xi_u=(\xi_1(u),\dots,\xi_n(u))\in[K]^n$. Given $\xi_u$, the connection probabilities are given by
$$
\Theta_{ij}(\xi_u)=\begin{cases}\rho_n Q_{\xi_i(u),\xi_j(u)}, & i\neq j\\
0,& i=j.
\end{cases}
$$
Denote by $\P_{W_u,\xi_u}(A^t)$ the conditional distribution of the dynamic network at time $t$ given the node assignment $\xi_u$:
\begin{align*}
\P_{W_u,\xi_u}(A^t)&=\prod_{i<j} \Theta_{ij}(\xi_u)^{A_{ij}^t} (1-\Theta_{ij}(\xi_u))^{1-A_{ij}^t} \\
&= \prod_{i<j} \sum_{k,l\in[K]^2}(\rho_n Q_{kl})^{A_{ij}^t} (1-\rho_n Q_{kl})^{1-A_{ij}^t} \one\{ (\xi_i(u),\xi_j(u))=(k,l)\}.
\end{align*}

In what follows we denote by $A=(A^1,\dots,A^T)$ the full set of observations and by  $A^{\le \tau}=(A^1,\dots,A^\tau)$ and $A^{>\tau}=(A^{\tau+1},\dots,A^T)$ the realizations before and after the time~$\tau$.  We denote  by $\P^{\otimes \tau}(A^{\le \tau})$ and $\P^{\otimes (T-\tau)}(A^{>\tau})$  the corresponding product measures.

A $K$-step graphon depends on two main ingredients: the partition $\Pi_K$ of $[0,1]^2$  and the connection probability matrix $Q$. We will see that choosing different prior distributions on $Q$ and $\Pi$ will lead to two different lower bounds. The first lower bound, that we call {\it agnostic error lower bound}, will be derived from the uncertainty of sampling  vector of features $\eps$. The second one, that we call {\it network sampling lower bound}, comes from the uncertainty of random realizations of the network. 

Without loss of generality, we will prove the result for $K=2$.
%and for the case when $K$ is a multiple of 16. 
Indeed, for any $K>2$, 
$$
\inf_{\psi\in \Psi_\alpha}\sup_{W^\tau,W^{\tau+1}\in \mathcal W_K(\delta_{n,T})} \P_{W^\tau,W^{\tau+1}} \Bigl\{\psi=0\Bigr\}\ge \inf_{\psi\in \Psi_\alpha}\sup_{W^\tau,W^{\tau+1}\in \mathcal W_2(\delta_{n,T})} \P_{W^\tau,W^{\tau+1}} \Bigl\{\psi=0\Bigr\}
$$
and the boundary for the case of two blocks will imply the one for $K$ blocks. 
%where $K'\le K$ is such that $K-K'\le 15$ and $K'$ is either multiple of 16 or $K'=2$.
%Then, since $16 K'\ge K$ obtaining the ``agnostic" lower bound for $K'$ implies that the lower bound on the minimax separation rate for $K\ge K'$ satisfies
%$$
%\delta_{n,T}^*\ge C(K'/N)^{1/4}\ge  \frac 14 C (K/N)^{1/4}.
%$$
%We will see that the second lower bound is independent of $K$. 

\begin{proof}[Proof of Theorem~\ref{th:Kstepgraphon_lowerbound}.]
\ \\
{\it I. Agnostic error lower bound.}
The first lower bound is related to the error coming from the sampling of $\boldsymbol\eps$. We start by choosing a prior distribution on the graphons and the assignment vectors. Based on the prior, we will bound the Kullback-Leibler divergence between measures under the null and the alternative hypotheses. Note that it follows from~\eqref{eq:lb_general}  and the inequality $\frac 12\|\P_1-\P_0\|_{\mathrm{TV}}\le \sqrt{\frac12\mathrm{KL}(\P_0,\P_1)}$ that the type II error is bounded from below by $\beta$ if $\mathrm{KL}(\P_0,\P_1)\le 2(1-\alpha-\beta)^2=2(1-\eta)^2$. Thus, we need to provide an upper bound on the Kullvack--Leibler divergence that will imply the corresponding lower bound on the minimax detectable distance between graphons.

\begin{description}
	\item {\it Step 1. Choice of priors.} 
	We will use $W_u$ graphons defined in~\eqref{eq:W_u}. We suppose that the connection probability matrix $Q$ is the same under $\Hyp_0$ and under $\Hyp_1$ and is defined as $Q=\begin{pmatrix} 1 & 1\\ 1& 0\end{pmatrix}$.  
	%We will use the graphons $W_u$ depending on the vector $u$ defined above with the connection probability~$Q$.  
	 
	1. {\it Prior under $\Hyp_0$:} we choose vector $u=0$ and get  a partition $\Pi$ with blocks of a constant size $1/2$. The corresponding graphon is denoted by $W_0$.   Let $\xi_0$ and  $\xi_0'$  be two independent blocks assignment vectors following the multinomial distribution with  class probabilities $p_k=1/2$, $k=1,2$. %Since matrix $Q$ does not change, the corresponding  conditional distributions $\P_{W_0,\xi_0}$ and $\P_{W_0,\xi_0'}$ coincide.
Then, the measure under $\Hyp_0$ be given by
	\begin{align*}
	\P_0(A)
	%&=\sum_{a,b\in[K]^n} \P(\xi_0=a)\P(\xi_0'=b) \prod_{t=1}^\tau \P_{W_{\xi_0}} (A_t|\xi_0=a)\prod_{t=\tau+1}^T \P_{W_{\xi_0'}} (A_t|\xi_0'=b)\\
	&=\Bigl(\sum_{a\in\{1,2\}^n} \P(\xi_0=a) \P_{W_0, \xi_0=a}^{\otimes \tau} (A^{\le\tau})\Bigr) \Bigl(\sum_{b\in\{1,2\}^n} \P(\xi_0'=b) \P_{W_0,\xi_0'=b}^{\otimes (T-\tau)} (A^{>\tau})\Bigr).
	\end{align*}
	2. {\it Prior under $\Hyp_1$:}  fix some  $0<\lambda<1/2$. Let $u=(\lambda,-\lambda)$, $v=(-\lambda,\lambda)$ and $W_u$, $W_v$ be the corresponding graphons with the probabilities of classes  $1/2+\lambda$ and $1/2-\lambda$. The only difference between these two graphons  is a slight disequilibrium around $(x,y)\in[1/2-\lambda,1/2+\lambda]^2$. It is not difficult to see that $\delta^2(W_u,W_v)\ge 2\lambda$. 
	
Let $\xi_u\in\{1,2\}^n$ and $\xi_v\in\{1,2\}^n$ be two independent class assignment vectors before and after the change in the graphon such that $\P(\xi_i(u)=k)=\frac 12+u_k$, $\P(\xi_i(v)=k)=\frac 12+v_k$ $\forall i=1,\dots,n$, $k=1,2$.   Note that the  prior with  two different assignment vectors and the same the connection probability matrix $Q$ takes  into account the case of a possible label mismatch before and after the change. The measure under $\Hyp_1$ is defined as  
	$$
	\P_1(A)=\Bigl(\sum_{a\in\{1,2\}^n}\P(\xi_u=a)\P_{W_u,\xi_u=a}^{\otimes \tau} (A^{\le\tau})\Bigr) \Bigl(\sum_{b\in\{1,2\}^n}\P(\xi_v=b)\P_{W_v,\xi_v=b}^{\otimes (T-\tau)} (A^{>\tau})\Bigr).
	$$
	\item[\it Step 2. Bounding the divergence.] Denote for brevity $\P(\xi=a)$ by $\P_\xi(a)$. Since the matrix $Q$ is the same for all graphons, the conditional probabilities generating the networks under $\Hyp_0$ and under $\Hyp_1$ are the same. Denote
	$$
	P_Q(A^{\le \tau}|a):=\P_{W_0,\xi_0=a}^{\otimes\tau}(A^{\le \tau})=\P_{W_u,\xi_u=a}^{\otimes\tau}(A^{\le \tau})
	$$ 
	and 
	$$
	P_Q(A^{>\tau}|b):=\P_{W_o,\xi_0'=b}^{\otimes (T-\tau)}(A^{>\tau} )=\P_{W_v,\xi_v=b}^{\otimes(T-\tau)}(A^{> \tau}).
	$$
	Then, we have
	\begin{align*}
	\mathrm{KL}(\P_0,\P_1) 
	&= \sum_{A} \Bigl(\sum_{a\in\{1,2\}^n} \P_{\xi_0}(a) P_Q(A^{\le \tau}|a)\Bigr) \Bigl(\sum_{b\in\{1,2\}^n} \P_{\xi_0'}(b) P_Q(A^{> \tau}|b)\Bigr)\\
	&\times\log \left(\frac{\Bigl(\sum\limits_{a\in\{1,2\}^n} \P_{\xi_0}(a) P_Q(A^{\le \tau}|a)\Bigr) \Bigl(\sum\limits_{b\in\{1,2\}^n} \P_{\xi_0'}(b) P_Q(A^{> \tau}|b)\Bigr)}
	{\Bigl(\sum\limits_{a\in\{1,2\}^n}\P_{\xi_u}(a)P_Q(A^{\le \tau}|a)\Bigr) \Bigl(\sum\limits_{b\in\{1,2\}^n}\P_{\xi_v}(b)P_Q(A^{>\tau}|b)}\right)\\
	&= \sum_{A^{\le \tau}} \sum_{a\in\{1,2\}^n} \P_{\xi_0}(a) P_Q(A^{\le \tau}|a)\log\left( \frac{\sum\limits_{a\in\{1,2\}^n} \P_{\xi_0}(a) P_Q(A^{\le \tau}|a)}
	{\sum\limits_{a\in\{1,2\}^n}\P_{\xi_u}(a)P_Q(A^{\le \tau}|a)}\right)\\
	&+ \sum_{A^{> \tau}} \sum_{b\in\{1,2\}^n} \P_{\xi_0'}(b) P_Q(A^{>\tau}|b) \log \left( \frac{\sum\limits_{b\in\{1,2\}^n} \P_{\xi_0'}(b) P_Q(A^{>\tau}|b)}
	{\sum\limits_{b\in\{1,2\}^n}\P_{\xi_v}(b)P_Q(A^{>\tau}|b)}\right).
	\end{align*}
	Thus, taking into account that the function $f(x,y)=x\log (x/y)$ is convex, we can apply the Jensen's inequality and obtain that
	\begin{align*}
	\mathrm{KL}(\P_0,\P_1) &\le \sum_{a\in\{1,2\}^n} \P_{\xi_0}(a)\log \frac{\P_{\xi_0}(a)}{\P_{\xi_u}(a)} \sum_{A^{\le \tau}}  P_Q(A^{\le \tau}|a)  \\
	&+\sum_{b\in\{1,2\}^n} \P_{\xi_0'}(b)\log \frac{\P_{\xi_0}(b)}{\P_{\xi_v}(b)} \sum_{A^{> \tau}}  P_Q(A^{> \tau}|b)\\
	& =n\Bigl(\mathrm{KL}(\P_{\xi_0},\P_{\xi_u})+  \mathrm{KL}(\P_{\xi_0'},\P_{\xi_v})\Bigr).
	\end{align*}
	The last equality follows from the fact that $\P_{\xi}(a)$ are product probabilities. Thus we have to bound the Kullback--Leibler divergence between two binomial distributions. Using the inequality $\log(1+x)\ge x/(1+x)$ $\forall x>-1$, we obtain 
	$$
	\mathrm{KL}(\P_{\xi_0},\P_{\xi_u})=\frac 12 \log \Bigl(\frac{1/4}{1/4-\lambda^2}\Bigr)\le \frac{2\lambda^2}{1-4\lambda^2}
	$$
	which implies $\mathrm{KL}(\P_0,\P_1) \le 4n\lambda^2/(1-4\lambda^2)$. Recall that $\delta^2(W_u,W_v)\ge 2\lambda$. Consequently, if $\delta=\delta(W_u,W_v)$, we can write
	$$
	\mathrm{KL}(\P_0,\P_1) \le \frac n2 \frac{\delta^4}{1-\delta^4}\le 2(1-\eta)^2
	$$
	if 
	$$
	\delta^4\le \frac{4n^{-1}(1-\eta)^2}{1+4n^{-1}(1-\eta)^2}.
	$$
	The last inequality is true if $\delta^4\le \frac83 n^{-1} (1-\eta)^2$ for all $n\ge 8$.  It implies the lower bound condition on the distance between graphons:
	$$
	\delta(W_u,W_v) \le \Bigl(\frac 83\Bigr)^{1/4} (1-\eta)^{1/2} n^{-1/4}.
	$$
%	 Since $1/K+u_k\ge 1/K-\eps>3/(4K)$, using the fact that the Kullback--Leibler divergence  is bounded from above by the chi-squared divergence, we obtain
%	$$
%	\mathrm{KL}(\P_{\xi_0},\P_{\xi_u})=\sum_{k\in[K]}(1/K)\log\left(\frac{1/K}{1/K+u_k}\right)\le \sum_{k\in[K]} \frac{u_k^2}{1/K+u_k}\le 4K\eps^2/3.
%	$$
%	Applying the same bound to $\mathrm{KL}(\P_{\xi_0'},\P_{\xi_v})$, we obtain
%	$\mathrm{KL}(\P_0,\P_1) \le 8n K\eps^2/3$.
%	Lemma~4.5 of~\cite{klopp_graphon} implies that for any $u,v\in\mathcal C$ and for all $K$ that are multiple of 16, the graphons $W_u$ and $W_v$ are well separated with the distance $\delta^2(W_u,W_v)\ge K\eps/2^9$. In case of $K=2$ we can take $Q=\begin{pmatrix} 1 & 1\\ 1& 0\end{pmatrix}$ and show that $\delta^2(W_u,W_v)\ge \eps$. Consequently,
%	$$
%	\mathrm{KL}(\P_0,\P_1)\le C \delta^4 (n/K),
%	$$
%	where $C=2^{21}/3$ and $\delta=\delta(W_u,W_v)$. 
	
%	Thus, the type II error is bounded by $\beta$ if
%	$$
%	\mathrm{KL}(\P_0,\P_1)\le  \delta^4 (n/K)\le 2(1-\eta)^2
%	$$
%	which implies the lower bound condition on the distance between graphons:
%	$$
%	\delta(W_u,W_v) \le \frac{3^{1/4}}{32}  (1-\eta)^{1/2} \left(\frac Kn\right)^{1/4}.
%	$$
\end{description}
{\it II. Network Sampling lower bound}.
In this part we will suppose that the transition matrix $Q$ changes but the partition $\Pi$ does not change. 
%We will use the chi-squared divergence between measures in the lower bound condition as in~\eqref{eq:lb_chi}. 
In order to bound the type II error by $\beta$ from below, we need to show that the chi-squared divergence between the 
mixtures under $\Hyp_0$ and $\Hyp_1$ is smaller that $4(1-\eta)^2$. 
\begin{description}
	\item[\it Step 1. Choice of priors.] It will be sufficient to show the result for the case of $K=2$, since as we will see the lower bound on the separation rate is independent of $K$. We will work with 2-step graphons with fixed partition $\Pi$ into 4 equal blocks $\Pi_{kl}=[k-1/2,k/2)\times[l-1/2,l/2)$, $1\le k,l\le 2$.  
	
	{\it Prior under $\Hyp_0$}. We suppose that under $\Hyp_0$ the connection probabilities are all equal to $1/2$, that is  $Q=\begin{pmatrix} 1/2 & 1/2\\ 1/2& 1/2\end{pmatrix}$ and $\Theta_{ij}=\rho_n/2$, $\forall i\neq j$. The corresponding graphon is denoted by $W_0$. Denote $p_0=\rho_n/2$. Then, independently of the feature vector $\boldsymbol\eps=(\eps_1,\dots,\eps_n)$, 
	$$
	\P_0(A)=\prod_{i<j}\prod_{i=1}^T p_0^{\sum_{t=1}^T A_{ij}^t}(1-p_0)^{T-\sum_{t=1}^T A_{ij}^t}
	$$
	
	{\it Prior under $\Hyp_1$}. Denote by $Q_1$ and $Q_2$ the connection probability matrices before and after the change-point. Let $\lambda>0$. We assume that 
	$$
	Q_1=\rho_n^{-1}\begin{pmatrix} p_1 & p_2\\ p_2 & p_1 \end{pmatrix}
	\quad\mbox{and}\quad 
	Q_2=\rho_n^{-1}\begin{pmatrix} p_3 & p_4\\ p_4 & p_3\end{pmatrix}
	$$
	where 
	\begin{align*}
	p_1=\rho_n\Bigl(\frac12 +\Bigl(1-\frac \tau T\Bigr)\lambda\Bigr),&\quad p_2=\rho_n\Bigl(\frac12 -\Bigl(1-\frac \tau T\Bigr)\lambda\Bigr),\\
	p_3=\rho_n\Bigl(\frac12 -\frac \tau T\lambda\Bigr),&\quad p_4=\rho_n\Bigl(\frac12 +\frac \tau T\lambda\Bigr).
	\end{align*}
	Denote the corresponding graphons by $W_1$ and $W_2$ and the corresponding matrices of connection probabilities by $\Theta_i$, $i=1,2$. Let $\xi=(\xi_1,\dots,\xi_n)$ be the class assignment vector of i.i.d. variables taking values $\{1,2\}$ with probability 1/2. Then 
	$$
	\P_1(A)=\P_{W_1}^{\otimes\tau} (A^{\le \tau}) \P_{W_2}^{\otimes (T-\tau)} (A^{> \tau})
	$$
	where
	$$
	P_{W_1}(A^t)=\prod_{i<j} \Bigl( p_1^{A_{ij}^t} (1-p_1)^{1-A_{ij}^t} \one\{\xi_i=\xi_j\}+ p_2^{A_{ij}^t} (1-p_2)^{1-A_{ij}^1} \one\{\xi_i\neq \xi_j\}\Bigr)
	$$
	and 
	$$
	P_{W_2}(A^t)=\prod_{i<j} \Bigl( p_3^{A_{ij}^t} (1-p_3)^{1-A_{ij}^t} \one\{\xi_i=\xi_j\}+ p_4^{A_{ij}^t} (1-p_4)^{1-A_{ij}^1} \one\{\xi_i\neq \xi_j\}\Bigr)
	$$
	\item[\it Step 2. Bounding the $\chi^2$-divergence.] To prove the lower bound, we will need an upper bound on the chi-squared divergence
	$$
	\chi^2(\P_0,\P_1)=\E_{\P_0} \Bigl(\frac{d\P_1}{d\P_0}\Bigr)^2-1 \le 4(1-\alpha-\beta)^2=4(1-\eta)^2. 
	$$
	This bound will follow from the upper bound on the second moment of the likelihood ratio:
	$$
	\E_{\P_0} L^2(A) \le 1+ 4(1-\eta)^2
	$$
	where $L(A)=\frac{d\P_1}{d\P_0}(A)$. 
	%The proof is similar to the proof of the lower bound in Theorem~\ref{th:lb_tau_unknown} and uses certain results of Lemma~4.9 in~\cite{klopp_graphon}. 
	 Define the set $S=\Bigl\{\{a,b\}\in [n]^2:\ a<b,\ \xi_a=\xi_b\Bigr\}$ and its compliment  $S^c$. Denote by $N=n(n-1)/2$ the cardinality of $S\cup S^c$ and by $\mu$ the distribution of $S$. 
	Then, $L(A)=\int L_S(A)d\mu(S)$, where
	\begin{align*}
	L_S(A)
	%= \prod_{\{i,j\}\in S} \frac{p_1^{\sum\limits_{t=1}^\tau A_{ij}^t} (1-p_1)^{\tau-\sum\limits_{t=1}^\tau A_{ij}^t} p_3^{\sum\limits_{t=\tau+1}^T A_{ij}^t}(1-p_3)^{T-\tau-\sum\limits_{t=\tau+1}^T A_{ij}^t}}{p_0^{\sum\limits_{t=1}^T A_{ij}^t}(1-p_0)^{T-\sum\limits_{t=1}^T A_{ij}^t}}\\
	%&\times \prod_{\{i,j\}\in S^c} \frac{p_2^{\sum\limits_{t=1}^\tau A_{ij}^t} (1-p_2)^{\tau-\sum\limits_{t=1}^\tau A_{ij}^t} p_4^{\sum\limits_{t=\tau+1}^T A_{ij}^t}(1-p_4)^{T-\tau-\sum\limits_{t=\tau+1}^T A_{ij}^t}}{p_0^{\sum\limits_{t=1}^T A_{ij}^t}(1-p_0)^{T-\sum\limits_{t=1}^T A_{ij}^t}}\\
	&= \left(\frac{1-p_1}{1-p_0}\right)^{\tau |S|} \left(\frac{1-p_3}{1-p_0}\right)^{(T-\tau) |S|} \left(\frac{1-p_2}{1-p_0}\right)^{\tau|S^c|} \left(\frac{1-p_4}{1-p_0}\right)^{(T-\tau)|S^c| }\\
	&\times \prod_{\{i,j\}\in S}\left(\frac{p_1(1-p_0)}{p_0(1-p_1)}\right)^{\sum\limits_{t=1}^\tau A_{ij}^t}
	\left(\frac{p_3(1-p_0)}{p_0(1-p_3)}\right)^{\sum\limits_{t=\tau+1}^TA_{ij}^t}\\
	&\times\prod_{\{i,j\}\in S^c}\left(\frac{p_2(1-p_0)}{p_0(1-p_2)}\right)^{\sum\limits_{t=1}^\tau A_{ij}^t}
	\left(\frac{p_4(1-p_0)}{p_0(1-p_4)}\right)^{\sum\limits_{t=\tau+1}^TA_{ij}^t}.
	\end{align*}
	We need to find the second moment of $L(A)$. Let $S_1$ and $S_2$ be two independent copies of $S$, then $\E_{\P_0}[L^2(A)]=\int \E_{\P_0} [L_{S_1}(A)L_{S_2}(A)]d\mu(S_1)d\mu(S_2)$. We have that
	\begin{align*}
	\E_{\P_0} [L_{S_1}(A)L_{S_2}(A)]&= \left(\frac{1-p_1}{1-p_0}\right)^{\tau (|S_1|+|S_2|)} \left(\frac{1-p_3}{1-p_0}\right)^{(T-\tau)( |S_1|+|S_2|)}\\
	&\times  \left(\frac{1-p_2}{1-p_0}\right)^{\tau(|S_1^c|+|S_2^c|} \left(\frac{1-p_4}{1-p_0}\right)^{(T-\tau)(|S_1^c|+|S_2^c|) }\\
	&\times \E_{\P_0}\left[\prod_{\{i,j\}\in S_1\cap S_2} \left(\frac{p_1^2(1-p_0)^2}{p_0^2(1-p_1)^2}\right)^{\sum\limits_{t=1}^\tau A_{ij}^t}
	\left(\frac{p_3^2(1-p_0)^2}{p_0^2(1-p_3)^2}\right)^{\sum\limits_{t=\tau+1}^TA_{ij}^t}\right.\\
	&\times \prod_{\{i,j\}\in S_1^c\cap S_2^c} \left(\frac{p_2^2(1-p_0)^2}{p_0^2(1-p_2)^2}\right)^{\sum\limits_{t=1}^\tau A_{ij}^t}
	\left(\frac{p_4^2(1-p_0)^2}{p_0^2(1-p_4)^2}\right)^{\sum\limits_{t=\tau+1}^TA_{ij}^t}\\
	&\hskip - 3 cm \left.\times \prod_{\{i,j\}\in S_1 \triangle S_2}\left(\frac{p_1p_2(1-p_0)^2}{p_0^2(1-p_1)(1-p_2)}\right)^{\sum\limits_{t=1}^\tau A_{ij}^t}
	\left(\frac{p_3p_4(1-p_0)^2}{p_0^2(1-p_3)(1-p_4)}\right)^{\sum\limits_{t=\tau+1}^TA_{ij}^t} 
	\right].
	\end{align*}
	Taking into account the relations $p_1-p_0=(1-\tau/T)\rho_n\lambda$, $p_2-p_0=-(1-\tau/T)\rho_n\lambda$, $p_3-p_0=-(\tau/T)\rho_n\lambda$, $p_4-p_0=(\tau/T)\rho_n\lambda$ and Lemma~\ref{sec:aux_lb}, we obtain
	\begin{align*}
	\E_{\P_0} [L_{S_1}(A)L_{S_2}(A)]&= \left(1+\Bigl(1-\frac \tau T\Bigr)^2 \frac{\lambda^2\rho_n^2}{p_0(1-p_0)}\right)^{\tau(|S_1\cap S_2|+|S_1^c\cap S_2^c|)} \\
	&\times \left(1+\Bigl(\frac \tau T\Bigr)^2 \frac{\lambda^2\rho_n^2}{p_0(1-p_0)}\right)^{(T-\tau)(|S_1\cap S_2|+|S_1^c\cap S_2^c|)}\\
	&\times \left(1-\Bigl(1-\frac \tau T\Bigr)^2 \frac{\lambda^2\rho_n^2}{p_0(1-p_0)}\right)^{\tau |S_1\triangle S_2|}\\
	&\times \left(1-\Bigl(\frac \tau T\Bigr)^2 \frac{\lambda^2\rho_n^2}{p_0(1-p_0)}\right)^{(T-\tau)|S_1\triangle S_2|}.
	\end{align*}
	Next, using the fact that $|S_1\triangle S_2|+|S_1\cap S_2|+|S_1^c\cap S_2|=N$, we obtain
	\begin{align*}
	\E_{\P_0} [L_{S_1}(A)L_{S_2}(A)]&\le 
	\\&\hskip - 3cm \exp \left[ (|S_1\cap S_2|+|S_1^c\cap S_2^c|) \left(\tau\Bigl(1-\frac \tau T\Bigr)^2 +(T-\tau)\Bigl(\frac \tau T\Bigr)^2\right)\frac{\lambda^2\rho_n^2}{p_0(1-p_0)}\right.\\
	&\left.- |S_1\triangle S_2| \left(\tau \Bigl(1-\frac \tau T\Bigr)^2 +(T-\tau)\Bigl(\frac \tau T\Bigr)^2\right) \frac{\lambda^2\rho_n^2}{p_0(1-p_0)}\right] \\
	&=\exp\left[q^2\Bigl(\frac\tau T\Bigr)\frac{T\lambda^2\rho_n^2}{p_0(1-p_0)}(2|S_1\cap S_2|+2|S_1^c\cap S_2^c|-N)\right].
	\end{align*}
	Note that $p_0=\rho_n/2<1/2$. Thus, in order to bound the second moment likelihood ratio, we need to control the exponential moment
	$$
	\E_{\P_0}[L^2(A)] \leq \E_{S_1,S_2}\left\{\exp\left[\Bigl(2|S_1\cap S_2|+2|S_1^c\cap S_2^c|-N\Bigr)4Tq^2\Bigl(\frac\tau T\Bigr)\lambda^2\rho_n \right]\right\}.
	$$
	We need to control the exponential moment of the random variable $U=|S_1\cap S_2|+|S_1^c\cap S_2^c|$, where $S_1$ and $S_2$ are independent random variables distributed according to $\mu$. Following the last lines of Lemma~4.9 in~\citep{klopp_graphon}, denote by $\xi^{(1)}=(\xi_1,\dots,\xi_n)$ and  $\xi^{(2)}=(\xi_1,\dots,\xi_n)$ the assignment vectors corresponding to the variables $S_1$ and $S_2$, respectively. For any $(i,j)\in\{1,2\}^2$ introduce the random variable that counts the number of nodes in the classes $i$ and $j$ according to the first and the second assignement:
	$$
	N_{ij}=\Bigl|\Bigl\{a\in[n]:\ \xi_a^{(1)}=i,\ \xi_a^{(2)}=j\Bigr\}\Bigr|,\quad (i,j)\in\{1,2\}^2.
	$$
	Then $2|S_1\cap S_2|+n=N_{11}^2+N_{12}^2+N_{21}^2+N_{22}^2$ and $2|S_1^c\cap S_2^c|=2N_{11} N_{22}+2N_{12}N_{21}$. Hence, $2U+n=(N_{11}+N_{22})^2+(N_{12}+N_{21})^2$. Note that 
	$N_{11}+N_{22}+N_{12}+N_{21}=n$. Let $Z:=N_{11}+N_{22}-n/2$. It is a centered binomial random variable with parameters ($n$, 1/2) and 
	$$
	2U-N=(n/2+Z)^2+(n/2-Z)^2-n-N=2Z^2-n/2.
	$$
	 Consequently, we need to control the exponential moment of $Z^2$:
	$$
	\E_{\P_0}[L^2(A)]\le \E \exp \left[8Tq^2\Bigl(\frac\tau T\Bigr)\lambda^2\rho_n Z^2\right].
	$$
	Using Hoeffding's inequality, we can show that $\P(Z^2>t)\le 2e^{-2t/n}$, thus $Z^2$ is subexponential with the moments $\E [Z^{2k}]\le n^k k!$. Consequently, for any $\gamma_n$ such that $ 0<n\gamma_n<1$ we have 
$$
	\E e^{\gamma_n Z^2}\le 1+\sum_{k=1}^{+\infty} \frac{\gamma_n^k \E [Z^{2k}]}{k!} 
	\le \sum_{k=0}^{+\infty} (n\gamma_n)^k =\frac1{1-n\gamma_n}.
$$
	Set $\gamma_n=8Tq^2\Bigl(\frac\tau T\Bigr)\lambda^2\rho_n $. We can see that  if $n\gamma_n \le 4(1-\eta)^2(1+4(1-\eta)^2)^{-1}<1$, then $\E_{\P_0}[L^2(A)]\le 1+4(1-\eta)^2$.  
	Since $\delta^2(W_1,W_2)=4\lambda^2$, we have $\gamma_n =2Tn\rho_n q^2\Bigl(\frac\tau T\Bigr)\delta^2$ and the above condition on $n\gamma_n$ implies the lower bound
	$$
	q\Bigl(\frac\tau T\Bigr)\delta(W_1,W_2) \le \frac{\sqrt 2 (1-\eta)}{(1+4(1-\eta)^2)^{1/2}}\frac 1{\sqrt{n\rho_n T}}
	$$
	and the second part of the theorem follows. 

\end{description}
	\end{proof}

\section{Auxiliary results}

\subsection{Concentration inequalities for matrix processes}
The first result is the concentration inequality for the operator norm of a random matrix with independent entries (see \cite{bandeira2016}, Corollary 3.12 and Remark 3.13):
\begin{proposition}[Bandeira and Van Handel, 2016]\label{pr1}
	Let $W$ be an $m\times m$ symmetric matrix whose entries $W_{ij}$ are independent centered random variables bounded (in absolute value) by some $\sigma_*>0$. Then, for any $0<\epsilon\leq 1/2$ there exists a universal constant $c_{\epsilon}$ such that, for every $x\geq 0$
	$$ 
	\P\left\{ \|W\|_{2\rightarrow 2}\geq 2\sqrt 2(1+\epsilon)\sigma+x\right\}\le  m\exp\left (-\frac{x^{2}}{c_{\epsilon}\sigma^{2}_*}\right ),
	$$
	where $\sigma=\max_i \left[\sum_{j}\var(W_{ij})\right]^{1/2}$.
\end{proposition}
Next we apply this concentration inequality to the Matrix CUSUM statistics. 

Let  $X^{t}\in [-1,1]^{n\times n}$ ($1\le t\le T$) be a sequence of  matrices with independent entries $X_{ij}^t$ for any $1\le i,j\le n$ and for any $t=1,\dots T$. Assume that $X_{ij}^{t}$ are centered Bernoulli random variables taking values in $\{1-B_{ij}^t, -B_{ij}^t\}$ with success probability $B_{ij}^t$. Let $\pi$ denote a permutation of $\{1,\dots,n\}$. Consider the following centered matrix processes defined in~\eqref{def_xi} and \eqref{def_xi_pi}:
$$
\xi(t)=\sqrt{\frac{t(T-t)}T}\left(\dfrac{1}{t}\sum_{s=1}^{t}X^{s}-\dfrac{1}{T-t}\sum_{s=t+1}^{T}X^{s}\right),\quad 1\le t\le T-1,
$$
$$
\xi^{\pi}(t)= \sqrt{\frac{t(T-t)}T}\left (\dfrac{1}{t}\sum_{s=1}^{t}X^{s}-\dfrac{1}{T-t}\sum_{s=t+1}^{T}
X^{s}\circ \pi\right ), \quad 1\le t\le T-1.
$$
\begin{lemma}\label{lem:matrixBernstein}
	For any $\epsilon\in(0,1/2]$  there exists an absolute constant $C_\epsilon$  such that, for every $\delta\in(0,1)$   we have
	\begin{equation}\label{MatrixBernsteinStoch}
	\| \xi(t)\|_{2\rightarrow 2}\leq 2\sqrt{2}(1+\epsilon)\sqrt{\frac{t(T-t)}T} \left [\frac1{t^2}\sum_{s=1}^{t}\| B^{s}\|_{1,\infty}+ \frac1{(T-t)^2}\sum_{s=t+1}^{T}\|B^{s}\|_{1,\infty}\right ]^{1/2} + C_\epsilon \log\frac{2n}\delta
	\end{equation}
	with the probability larger than $1-\delta$.
\end{lemma}
%\begin{remark}
%	$c^*=4c_\eps$ if $\eps=1/2$.
%\end{remark}
\begin{proof} 
	The result follows from the direct application of Proposition~\ref{pr1}. 
	Since $X^s$ are independent, we can easily estimate $\sigma^2$ from above:
	\begin{align*}
	\sigma^2&=\max_i \sum_j \Var [\xi_{ij}(t)]\\
	&= \frac{t(T-t)}T\max_i \sum_j \left(\frac1{t^2} \sum_{s=1}^t B_{ij}^s(1-B_{ij}^s) +\frac1{(T-t)^2} \sum_{s=t+1}^T  B_{ij}^s(1-B_{ij}^s)  \right)\\
	&\le \frac{t(T-t)}T\max_i \left(\frac1{t^2} \sum_{s=1}^t \sum_jB_{ij}^s  +\frac1{(T-t)^2} \sum_{s=t+1}^T \sum_j B_{ij}^s \right)\\
	&= \frac{t(T-t)}T\left(\frac1{t^2} \sum_{s=1}^t \|B^s\|_{1,\infty} +\frac1{(T-t)^2} \sum_{s=t+1}^T  \|B^s\|_{1,\infty}\right).
	\end{align*}
	
	We will show now that  the norm $\|\xi(t)\|_\infty$ is bounded by some $\sigma^*$ with high probability. Consider the entries of the matrix $\xi(t)$ defined by 
	$$
	\xi_{ij}(t)=\sqrt{\frac{t(T-t)}T}\left(\dfrac{1}{t}\sum_{s=1}^{t}X_{ij}^{s}-\dfrac{1}{T-t}\sum_{s=t+1}^{T}X_{ij}^{s}\right)=\sum_{s=1}^T V_{ij}^s
	$$
	Since $-B_{ij}^s\le X_{ij}^s\le 1-B_{ij}^s$ we have $a_s\le V_{ij}^s\le b_s$ with 
	$$
	b_s-a_s\le \sqrt{\frac{t(T-t)}T} \left(\frac 1t\one_{\{ 1\le s\le t\}} + \frac 1{T-t}\one_{\{t<s\le T\}} \right).
	$$
	Since $X_{ij}^s$ are independent, applying the Hoeffding inequality, we obtain for any $x>0$ that for any $1\le i,j\le n$,
	$$
	\P\Bigl\{|\xi_{ij}(t)|>x \Bigr\} \le 2\exp\Biggl\{-\frac{2x^2}{\sum_s  (b_s-a_s)^2}\Biggr\}=2 e^{-2x^2}.
	$$
	Using the union bound, we get that 
	$$
	\P\Bigl\{\|\xi(t)\|_\infty>x \Bigr\} \le n^2 \P\Bigl\{|\xi_{ij}(t)|>x \Bigr\} =2n^2 e^{-2x^2}.
	$$
	Consequently, for any $\delta \in(0,1)$ we have  $\|\xi(t)\|_\infty \le \log^{1/2} (2n/\sqrt \delta)\le \log^{1/2}(2n/\delta)$ 
	with probability larger than $1-\delta/2$. Applying Proposition \ref{pr1} given the event $\Bigl\{\|\xi(t)\|_\infty \le\sigma_*\Bigr\}$ with $\sigma_*= \log^{1/2}(2n/\delta)$ we get 
	$$
	\P\Bigl\{\|\xi(t)\|_{2\to2}\ge 2\sqrt 2 (1+\epsilon) \sigma+x\Bigr\} \le n\exp\Bigl\{-\frac{x^2}{c_\epsilon \log(2n/\delta)} \Bigr\}+\frac \delta 2.
	$$
	Choosing $x = \sqrt{c_\epsilon} \log(2n/\delta)$ and $C_\epsilon=c_\epsilon^{1/2}$ we obtain \eqref{MatrixBernsteinStoch}.
\end{proof}

\begin{lemma}\label{lem_bernstein_sup}
	For any $\epsilon\in(0,1/2]$  there exists an absolute constant $C_\epsilon$  such that, for every $\delta\in(0,1)$   we have
	\begin{equation*}%\label{MatrixBernsteinStoch_sup}
		\sup_{\pi}\| \xi^{\pi}(t)\|_{2\rightarrow 2}\leq 2\sqrt{2}(1+\epsilon)\sqrt{\frac{t(T-t)}T}\left(\frac1t\left[\sum_{s=1}^{t}\| B^{s}\|_{1,\infty}\right]^{1/2}+ \frac1{(T-t)}\left[\sum_{s=t+1}^{T}\|B^{s}\|_{1,\infty}\right]^{1/2}\right)+ C_\epsilon \log\frac{4n}\delta
	\end{equation*}
	with probability larger than $1-\delta$.
\end{lemma}
\begin{proof}
Using the triangle inequality, for any fixed permutation $\pi$, we get 
\begin{align}\label{eq:proof_bernstein_sup}
	\| \xi^{\pi}(t)\|_{2\rightarrow 2}&\leq \sqrt{\frac{t(T-t)}T}\left(\left\|\dfrac{1}{t}\sum_{s=1}^{t}X^{s}\right\|_{2\rightarrow 2}+\left\|\dfrac{1}{T-t}\sum_{s=t+1}^{T}
	X^{s}\circ \pi\right\|_{2\rightarrow 2}\right)\nonumber\\
	&=\sqrt{\frac{t(T-t)}T}\left(\left\|\dfrac{1}{t}\sum_{s=1}^{t}X^{s}\right\|_{2\rightarrow 2}+\left\|\dfrac{1}{T-t}\sum_{s=t+1}^{T}
	X^{s}\right\|_{2\rightarrow 2}\right)
\end{align}
where in the last line we use the invariance of the operator norm under a permutation. Applying Proposition \ref{pr1} to each one of the terms in \eqref{eq:proof_bernstein_sup} as it is done in the proof of Lemma \ref{lem:matrixBernstein} we get the result.
\end{proof}

\subsection{Result on the Hadamard product of two matrices}

%----------------
\begin{lemma}\label{lem:Hadamard_LB2}
	Let $A=(A_{ij})\in [0,\infty )^{n\times n}$ and $B=(B_{ij})\in\bR^{n\times n}$. Assume that $\diag(B)=0$.  
	Then,
	\begin{equation}\label{Hadamard_productLB1}
	\|A\odot B \|_{2\to 2}\ge 
	% -\lambda^{-1}_{\min}(\tilde A) 
	\frac{\underset{(ij):\,i\not = j}{\min} A_{ij}}{\sqrt {r\vee 1}}
	\|B\|_{2\to 2},
	\end{equation}
where $r=\mathrm{rank} (A\odot B)$. Moreover, if $A$ and $B$ are symmetric, we have that	
	\begin{equation}\label{Hadamard_productLB2}
	\|A\odot B \|_{2\to 2}\ge 
	% -\lambda^{-1}_{\min}(\tilde A) 
	\frac{\min_{ij} A_{ij}}{2\sqrt {r^*\vee 1}}
	\|B\|_{2\to 2},
	\end{equation}
	where 
	\[r^*=\underset{M=(M_{ij}):\,M_{ij}=(A\odot B)_{ij}\;\text{for}\;i\not = j}{\min} \mathrm{rank} (M).\]
\end{lemma}
\begin{proof}
	If $\underset{(ij):\,i\not = j}{\min} A_{ij}=0$, then the statement of the Lemma is trivially true. Now assume that $\underset{(ij):\,i\not = j}{\min} A_{ij}>0$. 
		We have that 
	\begin{align*}
	\|B\|^2_{2\to 2}&\leq \|B\|^2_{F}=\sum_{i\not = j} A^2_{ij} (A^{-1}_{ij})^2 B^2_{ij}\\
	&\leq  \underset{i\not = j}{\max} (A^{-1}_{ij})^2 \|A\odot B\|^2_{F}\leq r \,\underset{i\not = j}{\max} (A^{-1}_{ij})^2 \|A\odot B\|^2_{2\to 2}
	%\\
	%	&\leq 2r^* \underset{i\not = j}{\max} (A^{-1}_{ij})^2 \|A\odot B\|^2_{2\to 2}
	\end{align*}
which implies \eqref{Hadamard_productLB1}. On the other hand,	
	let $M$ be a solution to 
	\[M \in \underset{M=(M_{ij}):\,M_{ij}=(A\odot B)_{ij}\;\text{for}\;i\not = j}{\argmin} \mathrm{rank} (M).\]
	Let $\mathrm{rank} (M)=r^* < r$. We have that 
	\begin{align*}
		\|B\|^2_{2\to 2}&\leq \|B\|^2_{F}=\sum_{i\not = j} A^2_{ij} (A^{-1}_{ij})^2 B^2_{ij}\leq \underset{i\not = j}{\max} (A^{-1}_{ij})^2\Biggl \{ \sum_{i\not = j} \left (A_{ij} B\right )^2_{ij}+\sum_{i}M^2_{ii}\Biggr \}\\
		&= \underset{i\not = j}{\max} (A^{-1}_{ij})^2 \|M\|^2_{F}\leq r^* \underset{i\not = j}{\max} (A^{-1}_{ij})^2 \|M\|^2_{2\to 2}
		%\\
	%	&\leq 2r^* \underset{i\not = j}{\max} (A^{-1}_{ij})^2 \|A\odot B\|^2_{2\to 2}
	\end{align*}
which implies
	\begin{align}\label{LemmaHadamar2_1}
	\|B\|_{2\to 2}&\leq \sqrt{ r^*} \underset{i\not = j}{\max} (A^{-1}_{ij}) \|M\|_{2\to 2}\leq 2 \sqrt{ r^*} \underset{i\not = j}{\max} (A^{-1}_{ij}) \|A\odot B\|_{2\to 2}
	%\\
	%	&\leq 2r^* \underset{i\not = j}{\max} (A^{-1}_{ij})^2 \|A\odot B\|^2_{2\to 2}
	\end{align}
	where in the last inequality we use that $\|M\|_{2\to 2}\leq 2\|A\odot B\|_{2\to 2}$. To prove it, using the triangle inequality, it is enough to prove that $\|M-A\odot B\|_{2\to 2}\leq \|A\odot B\|_{2\to 2}$.
	Let denote by   $\lambda_1(X)\leq  \lambda_2(X)\leq \dots \leq \lambda_n(X)$ the eigenvalues of a symmetric matrix $X$. Then, using Weyl's inequality, we have that
	\begin{equation}\label{Weyl}
		\lambda_{j+k-n}(A\odot B)\leq \lambda_{j}(A\odot B-M)+\lambda_{k}(M)\leq \lambda_{j+k-1}(A\odot B).
	\end{equation}
Note that $r^* < r$ implies that there exist a $k$ such that $\lambda_{k}(M)=0$ but $\lambda_{k}(A\odot B)\not =0$.	Assume first that $\|M-A\odot B\|_{2\to 2}=-\lambda_1 (M-A\odot B)$. Then, taking in \eqref{Weyl} $j=1$ we get $-\lambda_1 (M-A\odot B)\leq \lambda_k(A\odot B)\leq \|A\odot B\|_{2\to 2}$. Now, if $\|M-A\odot B\|_{2\to 2}=\lambda_n (M-A\odot B)$,  taking $j=n$ we also get $\lambda_n (M-A\odot B)\leq  -\lambda_k(A\odot B)\leq \|A\odot B\|_{2\to 2}$.

To conclude the proof, note that \eqref{LemmaHadamar2_1} implies \eqref{Hadamard_productLB2}.
 \end{proof}

 \end{document}